%% file: paper.tex
\documentclass[10pt]{article}      
\usepackage{latexsym,amsmath,amsfonts,amscd,amssymb}
\usepackage{bm}
\usepackage{empheq}
\usepackage{multirow,mathtools,amsthm}
\usepackage{epsfig,subfigure,verbatim,epstopdf,graphics}
\usepackage{diagbox}
\usepackage{graphicx,caption} 
\usepackage{color,leftidx}
\usepackage[normalem]{ulem}
\usepackage[margin=1.0in]{geometry}
\usepackage{fonts}
\usepackage{soul} 
\usepackage{cases}
\usepackage{cite}
\usepackage{booktabs} 

\usepackage{todonotes}

\newtheorem{thm}{Theorem}[section]
\newtheorem{rem}[thm]{Remark}
\newtheorem{lem}[thm]{Lemma}

\newtheorem{eg}{Example}[section]
\newtheorem{defn}[thm]{Definition}
\newtheorem{ass}[thm]{Assumption}

\newcommand{\beq}{\begin{equation}}
\newcommand{\eeq}{\end{equation}}
\newcommand{\ben}{\begin{equation*}}
\newcommand{\een}{\end{equation*}}

\input{sections/macro}

\numberwithin{equation}{section}

\begin{document}             

\title{Multiscale convergence properties for spectral approximations of a model kinetic equation%
	 \footnote{This material is based upon work supported by the U.S. Department of Energy, Office of Science, Office of Advanced Scientific Computing Research.}%
	 \footnote{This manuscript has been authored by UT-Battelle, LLC under Contract No. DE-AC05-00OR22725 with the U.S. Department of Energy. The United States Government retains and the publisher, by accepting the article for publication, acknowledges that the United States Government retains a non-exclusive, paid-up, irrevocable, world-wide license to publish or reproduce the published form of this manuscript, or allow others to do so, for United States Government purposes. The Department of Energy will provide public access to these results of federally sponsored research in accordance with the DOE Public Access Plan (\texttt{http://energy.gov/downloads/doe-public-access-plan}).}
	}
\author{
Zheng Chen\footnote{Computational and Applied Mathematics Group, Oak Ridge National Laboratory, Oak Ridge, TN 37831 USA. Email: chenz1@ornl.gov.},
Cory D. Hauck\footnote{Computational and Applied Mathematics Group, Oak Ridge National Laboratory, Oak Ridge, TN 37831 USA. Email: hauckc@ornl.gov.}}

\maketitle

\begin{abstract}
\input{sections/abstract}

\end{abstract}




\tableofcontents

\input{sections/intro}
\input{sections/outline}
\input{sections/proof}
\input{sections/numerics}
\input{sections/numerics2}

\input{sections/conclusion}


\appendix
\input{sections/appendix}

\bibliographystyle{abbrv}
\bibliography{refer}
\end{document}

%% file: sections/macro.tex
\newcommand{\eps}{\epsilon}
\newcommand{\partt}{\partial_t}
\newcommand{\partx}{\partial_x}

\newcommand{\muint}[1]{\int_{-1}^{1} #1 \, d \mu}
\newcommand{\quand}{\quad \text{and} \quad}

\newcommand{\qquand}{\qquad \text{and} \qquad}
\newcommand{\mO}{\mathcal{O}}

\newcommand{\Flow}{f^{\rm{low}}} 
\newcommand{\Fhigh}{f^{\rm{high}}} 
\newcommand{\elow}{\xi^{\rm{low}}}

\newcommand{\fl}{f_\ell}
\newcommand{\fNl}{f_\ell^N}

\newcommand{\uhigh}{u^{\rm{high}}}
\newcommand{\ulow}{u^{\rm{low}}}
\newcommand{\ulhigh}{u_\ell^{\rm{high}}}
\newcommand{\ullow}{u_\ell^{\rm{low}}}
\newcommand{\ghigh}{g^{\rm{high}}}

\newcommand{\etahigh}{\eta^{\rm{high}}}
\newcommand{\etalow}{\eta^{\rm{low}}}
\newcommand{\xihigh}{\xi^{\rm{high}}}
\newcommand{\xilow}{\xi^{\rm{low}}}
\newcommand{\flhigh}{f_\ell^{\rm{high}}}
\newcommand{\fllow}{f_\ell^{\rm{low}}}

\newcommand{\xilhigh}{\xi_\ell^{\rm{high}}}
\newcommand{\xillow}{\xi_\ell^{\rm{low}}}

\newcommand{\fNzero}{f_0^N}
\newcommand{\fNone}{f_1^N}
\newcommand{\fNtwo} {f_2^N}
\newcommand{\fNthree}{f_3^N}
\newcommand{\fNfour}{f_4^N}
\newcommand{\fNfive}{f_5^N}
\newcommand{\fNlmo}{f_{\ell-1}^N}
\newcommand{\fNlpo}{f_{\ell+1}^N}
\newcommand{\fNN}{f_N^N}
\newcommand{\fNNmo}{f_{N-1}^N}
\newcommand{\fzero}{f_0}
\newcommand{\fone}{f_1}
\newcommand{\flmo}{f_{\ell-1}}
\newcommand{\flpo}{f_{\ell+1}}

\newcommand{\fNlk}{f_{\ell,k}^N}
\newcommand{\flk}{f_{\ell,k}}
\newcommand{\flzero}{f_{\ell,0}}
\newcommand{\fzerozero}{f_{0,0}}
\newcommand{\fzerok}{f_{0,k}}
\newcommand{\fonek}{f_{1,k}}
\newcommand{\ftwok}{f_{2,k}}
\newcommand{\flpok}{f_{\ell+1,k}}
\newcommand{\flmok}{f_{\ell-1,k}}
\newcommand{\glzero}{g_{\ell,0}}
\newcommand{\gnk}{g_{n,k}}
\newcommand{\fnk}{f_{n,k}}
\newcommand{\fnpok}{f_{n+1,k}}
\newcommand{\fnmok}{f_{n-1,k}}

\newcommand{\fNzerok}{f_{0,k}^N}
\newcommand{\fNonek}{f_{1,k}^N}
\newcommand{\fNlmok}{f_{\ell-1,k}^N}
\newcommand{\fNlpok}{f_{\ell+1,k}^N}
\newcommand{\fNNk}{f_{N,k}^N} 
\newcommand{\fNNmok}{f_{N-1,k}^N}  

\newcommand{\fNpo}{f_{N+1}}
\newcommand{\fNpok}{f_{N+1,k}}

\newcommand{\elk}{\xi_{\ell,k}}
\newcommand{\ezerok}{\xi_{0,k}}
\newcommand{\eonek}{\xi_{1,k}}
\newcommand{\elpok}{\xi_{\ell+1,k}}
\newcommand{\elmok}{\xi_{\ell-1,k}}
\newcommand{\eNk}{\xi_{N,k}}
\newcommand{\eNmok}{\xi_{N-1,k}}
\newcommand{\eNmtk}{\xi_{N-2,k}}
\newcommand{\etwok}{\xi_{2,k}}
\newcommand{\enk}{\xi_{n,k}}
\newcommand{\elstarmok}{\xi_{\ell_*-1,k}}
\newcommand{\elstarmtk}{\xi_{\ell_*-2,k}}
\newcommand{\elstarmthk}{\xi_{\ell_*-3,k}}
\newcommand{\ulk}{u_{\ell,k}}
\newcommand{\ulks}{u_{\ell,k}^*}
\newcommand{\ulnk}{u_{\ell,-k}}
\newcommand{\uzerok}{u_{0,k}}
\newcommand{\uonek}{u_{1,k}}
\newcommand{\utwok}{u_{2,k}}
\newcommand{\uoneks}{u_{1,k}^*}
\newcommand{\utwoks}{u_{2,k}^*}
\newcommand{\elzero}{\xi_{\ell,0}}

\newcommand{\el}{\xi_\ell}
\newcommand{\enpok}{\xi_{n+1,k}}
\newcommand{\enmok}{\xi_{n-1,k}}

\newcommand{\eNpolN}{e^{N+1}_\ell}
\newcommand{\eNlN}{e^N_\ell}
\newcommand{\eNzeroN}{e^N_0}
\newcommand{\eNoneN}{e^N_1}
\newcommand{\eNtwoN}{e^N_2}
\newcommand{\eNzeroNpo}{e^{N+1}_0}
\newcommand{\eNoneNpo}{e^{N+1}_1}
\newcommand{\eNtwoNpo}{e^{N+1}_2}

\newcommand{\eNl}{\xi_\ell}
\newcommand{\eNzero}{\xi_0}
\newcommand{\eNone}{\xi_1}
\newcommand{\eNtwo}{\xi_2}
\newcommand{\eNthree}{\xi_3}
\newcommand{\eNfour}{\xi_4}
\newcommand{\eNfive}{\xi_5}
\newcommand{\eNlpo}{\xi_{\ell+1}}
\newcommand{\eNlmo}{\xi_{\ell-1}}
\newcommand{\eNN}{\xi_N}
\newcommand{\eNNmo}{\xi_{N-1}}

\newcommand{\etal}{\eta_\ell}
\newcommand{\xil}{\xi_\ell}

\newcommand{\im}{\operatorname{Im}}
\newcommand{\re}{\operatorname{Re}}

\newcommand{\Hjk}{\mathcal{H}^j_k}
\newcommand{\Hzerok}{\mathcal{H}^0_k}
\newcommand{\Honek}{\mathcal{H}^1_k}
\newcommand{\Hthreek}{\mathcal{H}^3_k}
\newcommand{\hgammak}{h^\gamma_k}
\newcommand{\Hjik}{\mathcal{H}_k^{j,i}} 
\newcommand{\HjNk}{\mathcal{H}_k^{j,N}} 
\newcommand{\Hzerolsmtwok}{\mathcal{H}_k^{0,\ell_*-2}} 
\newcommand{\Hzerozero}{\mathcal{H}^0_0}
\newcommand{\HzeroNmok}{\mathcal{H}_k^{0,N-1}} 

\newcommand{\an}{a_n}
\newcommand{\anK}{a_n^K}

%% file: sections/abstract.tex
	\label{sec:abstract} 
	In this work, we prove rigorous convergence properties for a semi-discrete, moment-based approximation of a model kinetic equation in one dimension.  This approximation is equivalent to a standard spectral method in the velocity variable of the kinetic distribution and, as such, is accompanied by standard algebraic estimates of the form $N^{-q}$, where $N$ is the number of modes and $q>0$ depends on the regularity of the solution.  However, in the multiscale setting, the error estimate can be expressed in terms of the scaling parameter $\epsilon$, which measures the ratio of the mean-free-path to the characteristic domain length.  We show that, for isotropic initial conditions, the error in the spectral approximation is $\mathcal{O}(\epsilon^{N+1})$.  More surprisingly, the coefficients of the expansion satisfy super convergence properties.  In particular, the error of  the $\ell^{th}$ coefficient of the expansion scales like $\mathcal{O}(\epsilon^{2N})$  when $\ell =0$ and $\mathcal{O}(\epsilon^{2N+2-\ell})$ for all $1\leq \ell \leq N$.  This result is  significant, because the low-order coefficients correspond to physically relevant quantities of the underlying system.   All the above estimates involve constants depending on $N$, the time $t$, and the initial condition.  We investigate specifically the dependence on $N$, in order to assess whether increasing $N$ actually yields an additional factor of $\epsilon$ in the error.   Numerical tests will also be presented to support the theoretical results.

%% file: sections/intro.tex
\section{Introduction} 
\label{sec:intro}
In this paper, we study the following linear kinetic model
\begin{subnumcases}{\label{eqn:transport_simple}}
\eps \partt f(x,\mu, t) + \mu \partx f(x,\mu,t) + \frac{1}{\eps}f(x,\mu,t) = \frac{1}{\eps} \bar{f}(x,t),
& $(x,\mu,t) \in [-\pi,\pi) \times [-1,1] \times (0,\infty),$ \qquad \label{eqn:transport_simple_1}\\
f(\pi,\mu,t) = f(-\pi,\mu,t),
& $(\mu,t) \in [-1,1] \times (0,\infty),$ \\
f(x,\mu, 0) = g(x,\mu),
& $(x,\mu) \in  [-\pi,\pi) \times [-1,1],$
\end{subnumcases}
where $\bar{f} = \frac12 \muint{f}$.  
In particular, we prove interesting convergence properties for spectral discretization with respect to the variable $\mu$. The function $f$ is a kinetic distribution function; the physical interpretation is that $f(x,\mu, t)$ gives the density of particles with respect to the measure $d \mu dx $ that at time $t$ are located at position $x \in [-\pi,\pi)$ and moving with velocity $\mu \in [-1,1]$.  
The parameter $\eps >0$ is a scaling parameter that measures the relative strength of different processes; more about this will be said below. 

System \eqref{eqn:transport_simple} is among the most elementary examples of a kinetic model. 
However, despite its simplicity, it shares the basic features of many kinetic equations:  particle advection (modeled by the operator $\cA \colon f \mapsto  -\mu \partx f$) and particle interactions (modeled by the scattering operator $\cL\colon f \mapsto \bar{f} - f$). 
These basic features are found in more realistic models that describe dilute gases  \cite{chapman1970mathematical, Cercignani-1988, Cercignani-Illner-Pulvirenti-1994}; neutron \cite{lewis1984computational,dautray2012mathematical,case1967linear,davison1957neutron}, photon \cite{Pomraning-1973,mihalas1999foundations}, and neutrino \cite{mezzacappa1999neutrino} radiation; charged transport in semiconductor devices \cite{selberherr2012analysis,Markovich-Ringhofer-Schmeiser-1998}; and ionized plasmas \cite{hazeltine2004framework,boyd2003physics}.   
However, connecting \eqref{eqn:transport_simple} to these more realistic models requires the introduction of more complicated geometries, global field equations, nonlinearities, more complex collision mechanisms, and physical boundary conditions.  

Existence and uniqueness results for \eqref{eqn:transport_simple} follow from classical transport theory.  See, for example, \cite[Chapter XXI]{dautray2012mathematical}. For data $g(x,\mu) \in L^2(d \mu dx)$, \eqref{eqn:transport_simple} has a unique solution $f \in C^0([0,\infty); L^2( d \mu dx))$.  If further, $ g \in D(\cA) := \{ u \in L^2( d \mu dx) : \mu \partx u \in L^2( d \mu dx) \}$, then $f \in C^1([0,\infty); L^2( d \mu dx)) \cap C^0([0,\infty); D(\cA))$.

The scattering operator $\cL$ is self adjoint in $L^2(d \mu)$ and satisfies
\begin{equation}
\label{eqn:dissipation}
\muint{\psi \cL \psi} \leq 0  
\quand
\muint{(\psi - \bar{\psi}) \cL (\psi - \bar{\psi}) } = -\muint{(\psi - \bar{\psi})^2 }
\end{equation}  
for any function $\psi \in  L^2(d \mu)$. This simple dissipative structure motivates a diffusion approximation for \eqref{eqn:transport_simple} when $\eps \ll 1$.	
In such cases, $f = f^{(0)} + \cO(\eps)$, where $f^{(0)}$ is independent of $\mu$ and satisfies the diffusion equation
\cite{Larsen-Keller-1974,Habetler-Matkowsky-1975,bardos1984diffusion,bensoussan1979boundary}
\begin{equation}
	\partt f^{(0)} - \frac13 \partx^2 f^{(0)} = 0. 
\end{equation}

The diffusion approximation is useful because it removes the need for angular discretization and is therefore relatively cheap to compute; however, it does so at the expense of an $O(\eps)$ error.
Spectral methods (see \cite{hesthaven2007spectral,canuto2010spectral} in general or \cite[Chapter 3]{Lewis-Miller-1984} for applications to kinetic transport equations), on the other hand, are more expensive but can be used to discretize \eqref{eqn:transport_simple} with respect to $\mu$ when $\eps$ is not small.  
A standard spectral method for \eqref{eqn:transport_simple} seeks an approximation  
\beq \label{eqn:fN}
f^N(x,\mu,t) =  \sum_{\ell=0}^N \fNl(x,t) p_\ell(\mu),
\eeq 
	such that  
	\begin{equation}\label{eqn:spectral}
	\eps \partt f^N = \cP \cT f^N, \qquad f^N|_{t=0} = \cP g ,
	\end{equation} 
	where $\cT = \cA + \eps^{-1}\cL$, $p_\ell$ is the normalized, degree $\ell$ Legendre polynomial, and $\cP$ is the orthogonal projection from $L^2(d \mu)$ onto the space $\bbP^N$ of polynomials on $[-1,1]$ with degree at most $N$; that is
	\begin{equation}
	\cP \psi = \sum_{\ell=0}^N \psi_\ell p_\ell, \quad \text{where $\psi_\ell =  \int_{-1}^1 p_\ell \psi d \mu$}
	\end{equation}
	for any $\psi \in L^2(d \mu)$.  When expressed in terms of the expansion coefficients $\fNl$ in \eqref{eqn:fN}, \eqref{eqn:spectral} takes the form of a linear, symmetric hyperbolic system of balance laws in $x$ and $t$.  Standard semi-group theory (see for example \cite[Chapter 7]{brezis2010functional} or \cite[Chapter 7.4]{evans1998partial}) implies that this system has a solution in $C^0([0,\infty);[L^2(dx)]^{N+1})$ that is also in $C^1([0,\infty);[L^2(dx)]^{N+1}) \cap C^0( [0,\infty);[H^1(dx)]^{N+1} )$ when the expansion coefficients of $\cP g$ are in $H^1(dx)$.
	
	We refer to $f^N$ as the spectral approximation or $P_N$ solution. A straight-forward calculation shows that this approximation converges like
	\begin{equation}
	\label{eqn:spectral_est}
	\| f(\cdot,\cdot,t) - f^N(\cdot,\cdot,t)  \|_{L^2(d \mu dx)} \leq \frac{C(t)}{N^q},
	\end{equation}
	where $q$ is the number of $L^2$ angular derivatives of $f$ and $\partx f$ and the constant $C$ depends on $t$ but, due to the dissipative structure of $\cL$, does not depend on $\eps$ in a bad way.%
	\footnote{An estimate of the form \eqref{eqn:spectral_est} can be found in \cite{frank2016convergence} when $\eps=O(1)$.  However, a more general argument is needed to show that $C$ can be made independent of $\eps \in [0,1]$.  We give such an argument in the appendix.}  

A natural question for the spectral approximation is whether it provides an improvement over the diffusion approximation when $\eps$ is small.  The goal of the current paper is to derive an error estimate to demonstrate that this is in fact the case.  Specifically, let 
\begin{equation}\label{eqn:legendre_expansion}
f(x,\mu,t) = \sum_{\ell=0}^\infty \fl(x,t) p_\ell(\mu),\qquad \textrm{where } \fl(x,t)  = \int_{-1}^{1} p_\ell(\mu) f(x,\mu,t) d \mu,
\end{equation}
be the spectral expansion of $f$ in $L^2(d \mu)$.  
For small values of $\ell$, the coefficients $\{\fl\}$ correspond to measurable quantities and thus have physical significance.  For example, $\fzero$ is a constant multiple of the particle concentration.  Thus we also derive estimates for the errors in these coefficients, respectively. 

For the purposes of the current paper, we introduce the following assumption.
\begin{ass}\label{ass:1}
	The function $g$ is isotropic; that is, it is independent of $\mu$.  We write it as $g(x,\mu) = \frac{1}{\sqrt{2}}g_0(x)$, where the $\frac{1}{\sqrt{2}}$ is a normalization constant.
\end{ass}
\noindent This assumption is critical for the results in this paper, but will be removed in future work.  With it, our main result is the following:

\begin{thm}\label{thm:error}
	Suppose that $g_0 \in H^1(dx)$.  
	Then there exists an absolute constant $\lambda_1>0$ such that the $L^2$ error of the $P_N$ approximation satisfies
	\beq
	\label{eqn:L2_error}
	\|f - f^N\|_{L^2(d \mu dx)} (t) \leq B(g) e^{ -\frac{\lambda_1t}{\eps^2}} + C(\partx g)\sqrt t e^{-\frac{\lambda_1 t}{\eps^2}} + D(g,N,t) \eps^{N+1},
	\eeq
	where $D(g, N, t)$ is positive and bounded for any $t >0$ and is decreasing exponentially in $t$ for $t$ sufficiently large.
	Moreover, the $L^2$ error for each coefficient 
	satisfies
		\beq
		\label{eqn:moment_error}
		\|\fl - \fNl \|_{L^2(dx)} (t) \le 
		\begin{cases}
			C(\partx g)\sqrt t e^{-\frac{\lambda_1 t}{\eps^2}} + E(g,N,2,t)  \eps^{2N}, 
			& \ell=0,	\\
			C(\partx g)\sqrt t e^{-\frac{\lambda_1 t}{\eps^2}} + E(g,N,\ell,t) \eps^{2N+2-\ell}, 
			& 1 \leq \ell \leq N,
		\end{cases}
		\eeq
		where $E(g,N,\ell,t)$ is positive and bounded for any $t >0$ and is monotonically decreasing with respect to $t$.
\end{thm}

\begin{rem}
	A formal statement of the $\eps$-dependent scaling in \eqref{eqn:moment_error}, based on a Chapman-Enskog expansion, can be found in \cite{hauck2009temporal}.  In \cite{larsen1996asymptotic}, formal asymptotic results for the $SP_N$ equations, which are equivalent to the spectral approximation of \eqref{eqn:transport_simple} in the current setting, predict a similar scaling, at least for the coefficient $f_1$. 
\end{rem}

\begin{rem}
	In our proofs, we use $\lambda_1 = 1/45$  (cf. \eqref{eqn:def_lbd1} in Section \ref{sec:moment}).  We do not believe this value is optimal; nor have we made any effort to optimize it.
\end{rem}

\begin{rem}
	All of the rates in \eqref{eqn:L2_error} and \eqref{eqn:moment_error} are observed in the numerical tests in Section \ref{sec:numerics}.  
	Moreover, we observe these rates numerically even when $\partial_x g \notin L^2(dx)$.  Further discussion of this point is given in Section \ref{sec:numerics}.
\end{rem}

Theorem \ref{thm:error} has important practical consequences for the discretization of \eqref{eqn:transport_simple_1} in transition regimes, when $\eps$ is small, but not small enough to invoke the diffusion approximation.  Indeed, for a fully discrete scheme in space, time, and angle, it is important to balance errors with respect to each variable.   While not crucial for the solution of \eqref{eqn:transport_simple_1}, the efficiency gained from proper balancing of errors is essential for more general kinetic problems, for which the distribution function depends on six phase-space variables, plus time.  Theorem \ref{thm:error} justifies the use of fewer spectral modes than the standard estimate \eqref{eqn:spectral_est} in transition regimes.  The first statement of the theorem says that after an initial layer, the approximation of the transport solution is accurate up to $\mO(\epsilon^{N+1})$.   The second statement on the individual coefficients, which is much stronger, plays an even more important role, since it is the low-order coefficients that correspond to physically meaningful quantities.  However, for more realistic applications, these estimates will ultimately need to be extended beyond the current idealized setting.

The remainder of this paper is dedicated to the proof of Theorem \ref{thm:error} and the presentation of supporting numerical results.  
Preliminary notation and an introduction of the modified energy are given in Section \ref{sec:outline}.  
Details of proofs are provided in Section \ref{sec:proof}.
In Section \ref{sec:numerics}, we present some numerical tests to validate the convergence rates in theory.  
The benefit of increasing the number of moments $N$ is discussed in Section \ref{sec:numerics2}.  
Conclusions and future work are discussed in Section \ref{sec:conclusion}.

%% file: sections/outline.tex
\section{Preliminaries} 
\label{sec:outline}

In this section, we provide some preliminaries.  We first set the notation, and then introduce the modified energy approach borrowed from \cite{Dolbeault:2015fj}.

\subsection{Setup and Notation}

The proof of Theorem \ref{thm:error} relies on estimates of expansion coefficients for functions in $L^2(d \mu dx)$. 

\begin{defn}[Legendre and Legendre-Fourier expansion]
	\label{defn:expansions}
For any $u \in L^2(d \mu dx)$, the Legendre expansion of $u$ is
\beq
\label{eqn:L-expansion}
u(x,\mu) = \sum\limits_{\ell=0}^\infty u_\ell(x) p_\ell(\mu) , 
\qquad  u_\ell(x) = \int_{-1}^{1} u(x,\mu) p_\ell (\mu)d \mu,
\eeq
and the
Legendre-Fourier expansion is
\beq
\label{eqn:LF-expansion}
u(x,\mu) = \frac{1}{\sqrt{2\pi}} \sum\limits_{\ell=0}^\infty \sum\limits_{k=-\infty}^\infty u_{\ell,k} p_\ell(\mu)  e^{ikx}, 
\qquad u_{\ell,k} = \frac{1}{\sqrt{2\pi}} \int_{-\pi}^{\pi} u_\ell(x) e^{-ikx} dx .
\eeq
The coefficients ${u_\ell}$ and $u_{\ell,k}$ will be referred to as the Legendre and Legendre-Fourier coefficients.
\end{defn}

\begin{rem}
	The definition of $u_\ell$ in \eqref{eqn:L-expansion} is consistent with the use of $g_0$ in Assumption \ref{ass:1} and the definition of $f_\ell$ in \eqref{eqn:legendre_expansion}.
\end{rem}

\label{subsec:setup}
We begin by decomposing the error $e^N:= f - f^N$ into the sum of two components:
\begin{equation}
\eta = f - \cP f = \sum_{\ell=N+1}^\infty \etal(x,t) p_\ell(\mu)
\quad \text{and} \quad
\xi = \cP f - f^N = \sum_{\ell=0}^N \xil(x,t) p_\ell(\mu), \label{eqn:def_el}
\end{equation} 
where
\beq
\label{eqn:eta_xi_ell}
\etal(x,t) = \fl (x,t)
\quand
\xil(x,t) = \fl(x,t) - \fNl(x,t).
\eeq
These components are orthogonal with respect to the $L^2(d \mu)$ inner product, i.e., $\int_{-1}^{1} \eta \xi d \mu = 0$.

Equations for the expansion coefficients $\{\fl\}_{\ell=0}^\infty$ are derived using the three-term recurrence relation for the Legendre polynomials:
\beq \label{eqn:pl_recurrence}
\mu p_\ell(\mu) = a_\ell p_{\ell+1}(\mu) + a_{\ell-1}p_{\ell-1}(\mu),
\eeq
where
\beq \label{eqn:as}
\frac12 < a_\ell = \frac{\ell+1}{\sqrt{(2\ell+1)(2\ell+3)}} \leq \frac{1}{\sqrt{3}} .
\eeq
By taking the $L^2(d \mu)$ inner product of \eqref{eqn:transport_simple_1} with $p_\ell$, $\ell=0,\dots,\infty$, and invoking \eqref{eqn:pl_recurrence}, one arrives at an infinite system of equations for the expansion coefficients $\{ \fl \}_{\ell=0}^\infty$:
\beq \label{eqn:moments}
\left\{
\begin{array}{ll}
	\eps\partt \fzero + a_0 \partx \fone = 0,
	& \ell = 0,\\
	\eps\partt \fl + a_\ell \partx \flpo + a_{\ell-1} \partx \flmo + \frac{1}{\eps} \fl = 0,
	& \ell \ge 1.
\end{array}
\right.
\eeq 
When applied to \eqref{eqn:spectral}, the same procedure yields a similar set of equations for the coefficients $\{\fNl\}_{\ell=0}^N$:
\beq  \label{eqn:PN_moments}
\left\{
\begin{array}{ll}
	\eps\partt \fNzero + a_0 \partx \fNone  = 0,
	& \ell = 0,\\
	\eps\partt \fNl + a_\ell \partx \fNlpo + a_{\ell-1} \partx \fNlmo + \frac{1}{\eps} \fNl = 0,
	& 1 \le \ell \le N-1,\\
	\eps\partt \fNN  + a_{N-1} \partx \fNNmo + \frac{1}{\eps} \fNN = 0,
	& \ell = N,
\end{array}
\right.
\eeq
with initial condition $\fNl(\cdot,0) = \fl(\cdot,0)$, for $\ell = 0 ,\dots, N$.  We refer to this system as the $P_N$ system.
The equations in \eqref{eqn:PN_moments} differ in form from the first $N+1$ equations of \eqref{eqn:moments} only when $\ell = N$; it is this difference that is the origin of the 
error between $Pf$ and $f^N$.  Subtracting \eqref{eqn:PN_moments} from \eqref{eqn:moments} yields a system equations 
for $\{\xi\}_{\ell=0}^N$, with an additional source term in the last equation:
\beq  \label{eqn:PN_error}
\left\{
\begin{array}{ll}
	\eps\partt \eNzero + a_0 \partx \eNone  = 0,
	& \ell = 0,\\
	\eps\partt \el + a_\ell \partx \eNlpo + a_{\ell-1} \partx \eNlmo + \frac{1}{\eps} \el = 0,
	& 1 \le \ell \le N-1,\\
	\eps\partt \eNN  + a_{N-1} \partx \eNNmo + \frac{1}{\eps} \eNN = -a_N \partx \fNpo,
	& \ell = N.
\end{array}
\right.
\eeq
with initial condition $\el (\cdot,0) = 0$, for $\ell = 0, \dots, N$.

By taking Fourier transforms in $x$, we can write \eqref{eqn:moments} in terms of the Legendre-Fourier coefficients of $f$.
\beq \label{eqn:moments_k}
\left\{
\begin{array}{ll}
	\eps\partt \fzerok + a_0 ik \fonek  =0,
	& \ell = 0,\\
	\eps\partt \flk + a_\ell ik \flpok + a_{\ell-1} ik \flmok + \frac{1}{\eps} \flk = 0,
	& \ell \geq 1. \qquad \qquad
\end{array}
\right.
\eeq 
Similarly, the Legendre-Fourier coefficients of $f^N$ satisfy
\beq \label{eqn:PN_k}
\left\{
\begin{array}{ll}
	\eps\partt \fNzerok + a_0  ik  \fNonek = 0,
	& \ell = 0,\\
	\eps\partt \fNlk + a_\ell  ik  \fNlpok + a_{\ell-1}  ik  \fNlmok + \frac{1}{\eps} \fNlk = 0	,
	& 1 \le \ell \le N-1,\\
	\eps\partt \fNNk  + a_{N-1}  ik  \fNNmok + \frac{1}{\eps} \fNNk = 0	,
	& \ell = N,
\end{array}
\right.
\eeq 
and the Legendre-Fourier coefficients of $\xi$ satisfy
\beq \label{eqn:PN_error_k}
\left\{
\begin{array}{ll}
	\eps\partt \ezerok + a_0 ik \eonek = 0,
	& \ell = 0,\\
	\eps\partt \elk + a_\ell ik \elpok + a_{\ell-1} ik \elmok + \frac{1}{\eps} \elk = 0,
	& 1 \le \ell \le N-1,\\
	\eps\partt \eNk  + a_{N-1} ik \eNmok + \frac{1}{\eps} \eNk = -a_N ik \fNpok,
	& \ell = N.
\end{array}
\right.
\eeq 
It turns out the behavior of the Legendre-Fourier coefficients $f_{\ell,k}$ and $\xi_{\ell,k}$ depends on the wave number $k$, with the long-time behavior being dominated by the low frequency parts. We therefore separate the coefficients into high and low frequency terms. 

\begin{defn}[High and low frequency parts]
	\label{defn:frequencies}
	Let $\epsilon > 0$ be given and let $u \in L^2(d \mu dx)$ have Legendre and Legendre-Fourier coefficients as defined in Definition \ref{defn:expansions}.
	Then $u_\ell$ can be decomposed into a high frequency part $\ulhigh$ and a low frequency part $\ullow$, given by
	\beq
	\label{eqn:high_low_def_ell}
	\ulhigh(x) := \sum\limits_{|k|\eps >\frac12} u_{\ell,k} e^{ikx},
	\quad \text{and} \quad
	\ullow(x) :=  \sum\limits_{|k|\eps \leq \frac12} u_{\ell,k} e^{ikx},
	\eeq
	respectively.  Similarly, $u$ can be decomposed into a high frequency part $\uhigh$ and a low frequency part $\ulow$, given by
	\beq
	\label{eqn:high_low_def}
	\uhigh(x,\mu) := \sum\limits_{\ell=0}^\infty \ulhigh(x) p_\ell(\mu),
	\quad \text{and} \quad
	\ulow(x,\mu) := \sum\limits_{\ell=0}^\infty \ullow(x) p_\ell(\mu),
	\eeq
	respectively.
	
\end{defn}

\subsection{Modified energy method}
\label{sec:energy}
One may conclude from \eqref{eqn:dissipation} that solutions of \eqref{eqn:transport_simple} dissipate the energy functional $ \cH \colon L^2(d \mu dx) \to \bbR$, given by
\begin{equation}
\cH(u) =  \| u \|^2_{L^2(d \mu dx)}.
\end{equation}
A key tool in the proof of Theorem \ref{thm:error} is the spectral decomposition of $\cH$.

\begin{defn}
	\label{defn:energy_frequency}
	Given $\epsilon > 0$ and any $u \in L^2(d \mu dx)$ with Legendre-Fourier expansion in \eqref{eqn:LF-expansion}, let
	\beq\label{eqn:Hkj}
	\Hjk(u) := \frac{1}{2}\sum_{\ell=j}^\infty |\ulk|^2.
	\eeq
\end{defn} 
\noindent A direct consequence of Definitions \ref{defn:frequencies} and \ref{defn:energy_frequency} is that 
\beq
\label{eqn:u_low_high_L2}
\sum\limits_{|k|\eps > \frac12} \Hzerok(u) = \frac12 \| \uhigh\|^2_{L^2(d \mu dx)}
\quand
\sum\limits_{|k|\eps \leq \frac12} \Hzerok(u) = \frac12 \| \ulow\|^2_{L^2(d \mu dx)}.
\eeq

Since $f$ satisfies \eqref{eqn:moments_k}, it follows that
\beq \label{eqn:partt_H0_complex}
\partt \Hzerok(f) + \frac{2}{\eps^2} \Honek(f) + ik \sum_{\ell=0}^{\infty} a_\ell \left( \flk^* \flpok + \flpok^* \flk \right)= 0.
\eeq
For real-valued $f$, the real part of \eqref{eqn:partt_H0_complex} gives
\beq \label{eqn:partt_H0}
\partt \Hzerok(f) + \frac{2}{\eps^2} \Honek(f) = 0.
\eeq
Thus $\Hzerok(f)$ is a non-increasing function of time.  
However, this is not enough to prove that $\Hzerok(f)$ decays to zero or how.    In a similar calculation, \eqref{eqn:PN_error_k} implies
\beq \label{eqn:partt_H0_error}
\partt \Hzerok(\xi) + \frac{2}{\eps^2} \Honek(\xi) 
	\leq \frac{a_N |k|}{\eps}|\eNk| |\fNpok|
	\leq \frac{1}{2\eps^2}|\eNk|^2  +  \frac{k^2}{6} |\fNpok|^2, 
\eeq
where the third expression is a direct consequence of Young's inequality and the bound on $a_N$ from \eqref{eqn:as}.

In order to estimate the decay rate of the $\Hzerok(f)$ or $\Hzerok(\xi)$, the energy needs to be modified.  
Thus following \cite{Dolbeault:2015fj}, we modify the energy by adding a compensating function.  

\begin{defn}[Compensating function] \label{defn:compensating function}
	Given any $u \in L^2(d \mu dx)$ with Legendre-Fourier expansion in \eqref{eqn:LF-expansion},  a compensating function for $\Hzerok(u)$ in \eqref{eqn:Hkj} is a real-valued function 
\beq \label{eqn:hgamma}
\hgammak(u) = -\frac{\gamma}{4a_0}\im (\uzerok \uoneks),
\eeq 
where $\gamma \in \bbR$ is a positive scalar parameter to be determined and $a_0$ is the constant defined in \eqref{eqn:as}.  
\end{defn}

The role of the compensating function is elucidated by the following lemma
\begin{lem}\label{thm:compensating function}
	Let $u \in L^2(d \mu dx)$ have Legendre-Fourier coefficients that satisfy
	\beq  \label{eqn:two-moment}
	\left\{
	\begin{array}{ll}
		\eps\partt \uzerok + a_0 ik \uonek  =0, \\
		\eps\partt \uonek + a_1 ik \utwok + a_0 ik \uzerok + \frac{1}{\eps} \uonek = 0.
	\end{array}
	\right.
	\eeq
 Then 
	\beq
	\label{eqn:H_bound_gamma}
	(1-\frac{\gamma}{2}) \Hzerok(u) \le (\Hzerok + \hgammak)(u) \le  (1+\frac{\gamma}{2}) \Hzerok(u);
	\eeq
	and, for positive $k$, the time derivative is bounded by
	\beq \label{eqn:partt_h}
	\partt \hgammak(u) 
	\leq  - \gamma \left( \frac{k}{16\eps}|\uzerok|^2  - \left(\frac{ k}{4\eps}+\frac{3}{8\eps^3k}\right)|\uonek|^2  - \frac{ k}{5\eps}|\utwok|^2 \right).
	\eeq
\end{lem}	

\begin{proof}
	From the definition of $\hgammak$ in \eqref{eqn:hgamma} and the fact that $a_0 = 1/\sqrt{3}$, it follows that
		\beq
		|\hgammak(u)|   
		\leq \frac{\gamma}{2} |\uzerok| |\uonek| 
		\leq \frac{\gamma}{2} \left(\frac12|\uzerok|^2 + \frac12 |\uonek|^2\right)
		\leq \frac{\gamma}{2} \Hzerok(u),
		\eeq
	which immediately implies \eqref{eqn:H_bound_gamma}.
	To derive \eqref{eqn:partt_h}, we differentiate \eqref{eqn:hgamma} in time and use \eqref{eqn:two-moment} to conclude that
	\beq 
	\label{eqn:h_t}
	\partt \hgammak(u) = 
	\frac{\gamma}{4 } \left( 
	-\frac{ k}{\eps}|\uzerok|^2  + \frac{ k}{\eps}|\uonek|^2 + 
	\frac{1}{a_0\eps^2}\im(\uzerok \uoneks) -  \frac{a_1 k}{a_0\eps}\re(\uzerok \utwoks) 
	\right).
	\eeq
	Using Young's inequality, we compute bounds for the last two terms in \eqref{eqn:h_t}:
	\beq
	\frac{1}{a_0\eps^2}\im(\uzerok \uoneks) \leq \frac{ k}{2\eps} |\uzerok|^2  + \frac{1}{2a_0^2\eps^3k}|\uonek|^2 
	= \frac{ k}{2\eps} |\uzerok|^2  + \frac{3}{2\eps^3k}|\uonek|^2
	\eeq
	and
	\beq
	-\frac{a_1 k}{a_0\eps}\re(\uzerok \utwoks) \leq \frac{ k}{4\eps} |\uzerok|^2  + \frac{a_1^2 k}{a_0^2\eps}|\utwok|^2 
	= \frac{k}{4\eps} |\uzerok|^2  + \frac{4 k}{5\eps}|\utwok|^2.
	\eeq
	These bounds, when substituted into \eqref{eqn:h_t}, give \eqref{eqn:partt_h}.
\end{proof}

%% file: sections/proof.tex
\section{Proofs} 
\label{sec:proof}

This section is dedicated to the proof of Theorem \ref{thm:error}, which proceeds in 4 steps.  First, in Section \ref{sec:moment},  we determine bounds on the coefficients $\flk$ for $\ell=0, \dots,\infty$ and $k =-\infty, \dots, \infty$.  Second, in Section \ref{sec:error_expansion}, we use the bounds on $\flk$ to estimate $\eta$.  Third, in Section \ref{sec:energy_error}, we use the bound on $\fNpok$ to estimate $\xi$.  Fourth,  in Section \ref{sec:error}, we compute finer estimates on $\xillow$ for $\ell=0, \dots,N$.  In Section \ref{sec:proof_main}, the results of these four steps are combined to prove Theorem \ref{thm:error}. More specifically, the first three steps are used to establish the spectral error in \eqref{eqn:PN_moments}, while the last is required to establish the moment errors given in \eqref{eqn:PN_error}. 

In many cases, the proofs below rely on the decomposition of functions into high- and low-frequency components, as prescribed in Definition \ref{defn:frequencies}.  Since we consider only real-valued functions $u \in L^2(d \mu dx)$, $\ulks = \ulnk$.  Therefore $|\ulk| = |\ulnk|$, which means it is sufficient to consider only non-negative components of the Fourier spectrum, i.e., wave numbers $k \geq0$.

\input{sections/proof_moment}

\input{sections/proof_error_expansion}

\input{sections/modified_energy_xi}

\input{sections/proof_error_new}

\input{sections/proof_main_theorem}

%% file: sections/proof_moment.tex
\subsection{Bounding the coefficients of $f$} 
\label{sec:moment}
	
In this section, we first use the method of modified energy to bound $\Hzerok(f)$ in Lemma \ref{thm:energy_f}.  With such bounds and method of induction, we find bounds on $\flk$ in Lemma \ref{thm:f}.

\input{sections/modified_energy_f}

\begin{lem}\label{thm:f}
Let $g_0 \in L^2(dx)$ be given.
For $|k|\eps > 1/2$,
\beq \label{eqn:est_flk_high}
|\flk|(t) \le \sqrt{12 \Hzerok(g)}\, e^{ -\frac{\lambda_1 t}{\eps^2}}, \quad \ell = 0,1,2,\dots.
\eeq  
As a result,
		\beq
		\label{eqn:est_fl_high}
			\|\flhigh\|_{L^2(dx)}(t) \leq \|\Fhigh\|_{L^2(d \mu dx)} (t) \leq \sqrt{6} \, \|\ghigh\|_{L^2( d \mu dx)} e^{-\frac{\lambda_1 t}{\eps^2}} ,\quad \ell = 0,1,2,\dots,
		\eeq		
For $|k|\eps \leq 1/2$,
		\beq
			 |\flk| (t) \leq C_\ell^k \eps^\ell k^\ell e^{-\lambda_2 k^2 t},  \quad \ell = 0,1,2,\dots,
			 \label{eqn:est_flk_low}
	 \eeq
	 with 
	 \beq
			 C_\ell^k (g) =  \sqrt{12 \Hzerok(g)} A^\ell \quand A = \frac{2}{\sqrt3 (1-\lambda_2/4)} \simeq 1.2.
			 \label{eqn:clk}
	 \eeq
	As a result, 
		\beq
		\label{eqn:est_fl_low}
			\|\fllow\|_{L^2(dx)} (t) \leq F(g,\ell,t) \eps^{\ell},  \quad \ell = 0,1,2,\dots
		\eeq
			where
		\begin{align}
			F(g,\ell,t) &= \left[2 \max_{k >0}(C_\ell^k)^2 \sum_{k > 0}  k^{2\ell} e^{-2\lambda_2 k^2 t} + \delta_{\ell,0}|g_{0,0}|^2 \right]^\frac12 \nonumber \\
			&= \left[24 \max\limits_{k >0} \Hzerok(g) \sum\limits_{k > 0} (Ak)^{2\ell} e^{-2\lambda_2 k^2 t}+ \delta_{\ell,0}|g_{0,0}|^2 \right]^\frac12  
		\label{eqn:C_glt}
		\end{align}
		is positive, bounded for any $t>0$, independent of $k$ or $\eps$, and monotonically decreasing with respect to $t$.  
		As a result,
		\beq
			\|\fl\|_{L^2(dx)} (t) \leq  \sqrt{6}  \|\ghigh\|_{L^2(d \mu dx)} e^{-\frac{\lambda_1 t}{\eps^2}}+ F(g,\ell,t) \eps^{\ell}.
		\eeq
\end{lem}

\begin{proof}
We again consider high and low frequencies separately.  
\begin{itemize}
	\item[(i)] {\bf High frequency.}
	\input{sections/proof_moment_outside}
	\item[(ii)] {\bf Low frequency.}
	\input{sections/proof_moment_inside}

\end{itemize}

\end{proof}

\begin{rem}
	While $F(g,\ell,t)$ is independent of $\eps$, it depends on $\ell$.  A more careful examination of this dependence is provided in Section \ref{sec:numerics2}.
\end{rem}

\begin{rem}
	The assumption that $g$ is isotropic is critical to the proof above.  More specifically, it is needed in order to ignore the contribution of the initial condition in the first line of \eqref{eqn:induction_rhs_1}.  If $\gnk$ is not zero, then the estimates are quite different and the proofs are much more complicated.  We leave the analysis for anisotropic initial conditions to future work.
\end{rem}

\begin{rem}\label{thm:remark_fNl}
	The proof above also works for the coefficients $\fNl$ of the $P_N$ system. Hence the estimates on $\fl$ in Lemma \ref{thm:f} also apply to $\fNl$.  In Section \ref{sec:numerics}, we use $\fNl$ as a proxy for $\fl$ in order to verify these estimates numerically.
\end{rem}

%% file: sections/modified_energy_f.tex
\begin{lem}\label{thm:energy_f} For any $g_0 \in L^2(dx)$,
	
	\begin{empheq}[left= \Hzerok(f) (t) \leq
		\empheqlbrace]{alignat=2}
		& 6 e^{-\frac{2\lambda_1 t}{\eps^2}} \Hzerok(g), 
		&\quad |k|\eps > \frac12, 
		\label{eqn:est_Hk_high_f} \\
		& 6  e^{-2\lambda_2 k^2 t}  \Hzerok(g),
		&\quad  |k|\eps \leq \frac12,
		\label{eqn:est_Hk_low_f} 
	\end{empheq}		
	with $\lambda_1 = \frac{1}{45}$ (cf. \eqref{eqn:def_lbd1}) and $\lambda_2 = \frac{4}{45}$ (cf. \eqref{eqn:def_lbd2}).    
\end{lem}

\begin{proof}
We set $u = f$ in \eqref{eqn:partt_h}, add the result to \eqref{eqn:partt_H0}, and use the fact that $\Honek(f) = \Hthreek(f) +  \frac12 |\fonek|^2  + \frac12 |\ftwok|^2$.  This gives
\beq
\label{eqn:partt_Hh}
\partt \left( \Hzerok(f) + \hgammak(f) \right) 
	\leq - \frac{2}{\eps^2} \Hthreek(f) - \sum_{\ell=0}^2 c_{\gamma,\ell} |f_{\ell,k}|^2,
\eeq
where
\beq \label{eqn:def_c_gamma}
c_{\gamma,0} = \frac{\gamma k}{16\eps},
\qquad
c_{\gamma,1} = \frac{1}{\eps^2} -\frac{\gamma k}{4\eps} - \frac{3\gamma}{8\eps^3k},
\qquad
c_{\gamma,2} = \frac{1}{\eps^2}-\frac{\gamma k}{5\eps} .
\eeq
We next separate the frequency spectrum into high-frequency terms, when $k \eps > 1/2$, and low-frequency terms, when $ 0 \leq k \eps \leq 1/2$.   The choice of $\gamma$ and the subsequent estimates will depend on which part of the spectrum is being considered.

\begin{itemize}
	\item[(i)] {\bf High frequency.} For $k\eps > 1/2$, we set
	\beq\label{eqn:gamma_1}
	\gamma = \gamma^{\rm{high}} := \frac{16}{29}\frac{1}{k\eps} < \frac{32}{29}
	\eeq
	so that 
	\beq
	c_{\gamma,0} = \frac{1}{29\eps^2},
	\qquad
	c_{\gamma,1} = \left( \frac{1}{\eps^2} -\frac{4}{29\eps^2} - \frac{6}{29 k^2 \eps^4}\right) 
				> \frac{1}{29\eps^2},
	\qquad
	c_{\gamma,2} = \left(\frac{1}{\eps^2}-\frac{16}{145\eps^2} \right) 	> \frac{1}{29\eps^2}.
	\eeq
	By substituting these bounds into \eqref{eqn:partt_Hh}, we find that
	\beq
	\label{eqn:partt_Hh_high}
	\partt  \left( \Hzerok(f) + \hgammak(f) \right)  
		\leq - \frac{2}{\eps^2} \Hthreek(f) - \frac{1}{29\eps^2}\sum_{\ell=0}^2  |f_{\ell,k}|^2
		 \leq  - \frac{2}{29\eps^2} \Hzerok(f) 
		 \leq  - \frac{2}{45\eps^2}   \left( \Hzerok(f) + \hgammak(f) \right),
	\eeq
	where the last inequality uses the upper bound on $\Hzerok(f) + \hgammak(f)$ in \eqref{eqn:H_bound_gamma} and the upper bound on $\gamma^{\rm{high}}$ in \eqref{eqn:gamma_1}.
	We integrate the inequality in \eqref{eqn:partt_Hh_high} and apply the bounds in \eqref{eqn:H_bound_gamma}, using the fact that $\frac13 < \frac{13}{29} < 1-\frac{\gamma}{2} < 1+\frac{\gamma}{2} < \frac{45}{29} < 2$.     
	This gives
	\beq 
	\frac13 \Hzerok(f)(t) <  \left( \Hzerok(f) + \hgammak(f) \right) (t) \leq e^{ -\frac{2\lambda_1 t}{\eps^2}} \left( \Hzerok(g) + \hgammak(g) \right)  <  2 e^{ -\frac{2\lambda_1 t}{\eps^2}}\Hzerok(g),
	\quad \text{$\lambda_1 = \frac{1}{45}$,}
	\label{eqn:def_lbd1} 
	\eeq  
	from which \eqref{eqn:est_Hk_high_f} follows.

	\item[(ii)] {\bf Low frequency.} For $k=0$, the result in \eqref{eqn:est_Hk_low_f} follows trivially from \eqref{eqn:two-moment}.  For $0< k\eps \leq \frac12$, we let 
	\beq\label{eqn:gamma_2}
	\gamma = \gamma^{\rm{low}} :=\frac{64}{29} k\eps \leq \frac{32}{29}
	\eeq
		so that 
	\beq
	c_{\gamma,0} = \frac{4k^2}{29},
	\qquad
	c_{\gamma,1} = \left( \frac{1}{\eps^2} -\frac{16k^2}{29} - \frac{24}{29 \eps^2}\right) 
	\geq \frac{4k^2}{29},
	\qquad
	c_{\gamma,2} = \left(\frac{1}{\eps^2}-\frac{64}{145\eps^2} \right) 	> \frac{4k^2}{29}.
	\eeq

	By substituting these bounds into \eqref{eqn:partt_Hh}, we find that
	\beq
	\label{eqn:partt_Hh_low}
	\partt  \left( \Hzerok(f) + \hgammak(f) \right) 
	\leq - \frac{2}{\eps^2} \Hthreek(f) - \frac{4}{29}k^2\sum_{\ell=0}^2  |f_{\ell,k}|^2
	\leq  -\frac{8}{29} k^2 \Hzerok(f) \leq- \frac{8}{45}k^2   \left( \Hzerok(f) + \hgammak(f) \right), 
	\eeq
	where the last inequality uses the upper bound on $H$ in \eqref{eqn:H_bound_gamma} and the upper bound on $\gamma^{\rm{low}}$ in \eqref{eqn:gamma_2}. We integrate the inequality in \eqref{eqn:partt_Hh_low} and apply the bounds in \eqref{eqn:H_bound_gamma}, using the fact that $\frac13 < \frac{13}{29} \leq 1-\frac{\gamma}{2} < 1+\frac{\gamma}{2} \leq \frac{45}{29} < 2$.
    This gives
	\beq
	\frac13 \Hzerok(f)(t) <  \left( \Hzerok(f) + \hgammak(f) \right) (t) \leq e^{-2\lambda_2 k^2t}\left( \Hzerok(g) + \hgammak(g) \right) < 2 e^{-2\lambda_2 k^2t} \Hzerok(g),
	\quad \text{$\lambda_2 = \frac{4}{45}$},
	\label{eqn:def_lbd2} 
	\eeq
	from which \eqref{eqn:est_Hk_low_f} follows.  
\end{itemize}

\end{proof}

%% file: sections/proof_moment_outside.tex
For $|k| \eps >1/2$, the  definition of $\Hzerok$ in  \eqref{eqn:Hkj}, along with bound in \eqref{eqn:est_Hk_high_f}, implies that 
\beq
\label{eqn:est_flk_high_sq}
|\flk|^2 (t)
	\le 2 \Hzerok(f)  (t)
	\leq 12 \Hzerok(g) e^{ -\frac{2\lambda_1 t}{\eps^2}}. 
\eeq
Taking square roots gives \eqref{eqn:est_flk_high}.
We sum \eqref{eqn:est_flk_high_sq} over all $k$ such that $|k|\eps > 1/2$ and use the definition of $\flhigh$ in \eqref{eqn:high_low_def_ell} and the expression for $\Hzerok$ in \eqref{eqn:u_low_high_L2} to conclude that 

\beq
\|\flhigh\|^2_{L^2(dx)}(t) 
	= \sum\limits_{|k|\eps >\frac12}|\flk|^2(t)
	\leq \sum\limits_{|k|\eps >\frac12} 12 \Hzerok(g) e^{ -\frac{2\lambda_1 t}{\eps^2}} 
	=  6 \|\ghigh\|^2_{L^2(d \mu dx )} e^{-\frac{2\lambda_1 t}{\eps^2}}.
\eeq
Taking square roots gives \eqref{eqn:est_fl_high}.

%% file: sections/proof_moment_inside.tex
To establish \eqref{eqn:est_flk_low} for $0 \leq k\eps \leq 1/2$, we consider three cases, the first two of which are rather specific. 
\begin{itemize}
	\item Case 1: $\ell=0$, $k=0$.  In this case, direct inspection of \eqref{eqn:moments_k} shows that $\fzerozero(t) = g_{0,0}$ is constant w.r.t. $t$. The assumption that $g$ is isotropic implies that $\glzero(t) = 0$ for $\ell \geq 1$ and all $t \geq 0$. Hence $\Hzerozero(g) = \frac{1}{2}|g_{0,0}|^2$, whereby $C^0_0 = \sqrt{6} |g_{0,0}|$.  Thus the bound in \eqref{eqn:est_flk_low} is satisfied. 
	\item Case 2: $\ell \geq 1$ , $k=0$.  In this case, \eqref{eqn:moments_k} implies that $\flzero(t) = e^{-\frac{t}{\eps^2}}\flzero(0)$.  Hence, with the isotropic assumption on $g$, $\flzero(t) = 0$. Thus the bound in \eqref{eqn:est_flk_low} holds. 
	\item Case 3: $\ell \geq 0$, $0 < k \eps \leq 1/2$.  In this case, we actually prove the stronger statement
	\beq \label{eqn:est_flk_low_strong}
	|\fnk|(t) \le C_\ell^k \eps^\ell k^\ell e^{-\lambda_2 k^2 t},
	\qquad \ell  \geq 0, \quad n \geq \ell,
	\eeq
	with $C_\ell^k$ defined in \eqref{eqn:clk}.  The result in \eqref{eqn:est_flk_low} then follows by setting $n=\ell$ in \eqref{eqn:est_flk_low_strong}. 
	We proceed by induction on $\ell$.  According to the definition of $\Hzerok$ in \eqref{eqn:Hkj} and the bound in \eqref{eqn:est_Hk_low_f},
	\beq \label{ineq:H0}
	|\fnk|^2(t) \leq 2 \Hzerok(f)(t) \le 12 \Hzerok(g) \, e^{-2 \lambda_2 k^2 t} \,, \qquad \lambda_2 = \frac{4}{45}, \qquad n \geq 0.
	\eeq
	Taking square roots in \eqref{ineq:H0} recovers \eqref{eqn:est_flk_low_strong} for the case $\ell=0$.  Next, assume that \eqref{eqn:est_flk_low_strong} holds for $\ell = \ell_*$ for some $\ell_* \geq 0$ fixed. 
	Using \eqref{eqn:moments_k}, the estimate  \eqref{eqn:est_flk_low_strong} with $\ell = \ell_*$, and the fact that $a_n \leq 1/\sqrt{3}$ for all $n \geq 0$, we arrive at the following estimate for $|\fnk|$ for all $n \ge \ell_*+1$:
	\beq
	\label{eq:partt_fnk}
	\partt |\fnk| + \frac{1}{\eps^2} |\fnk| 
	\leq  \frac{k}{\eps} \left(a_n |\fnpok| + a_{n-1} |\fnmok|\right) 
	\leq \frac{2k}{\sqrt{3}\eps} \left(C_{\ell_*}^k \eps^{\ell_*} k^{\ell_*} e^{-\lambda_2 k^2 t}\right).  
	\eeq
	Thus integration of \eqref{eq:partt_fnk} in time (with an integrating factor on the left-hand side) gives 
	\begin{eqnarray}
	|\fnk|(t)  
	&\leq& e^{-\frac{t}{\eps^2}} |\gnk| + \frac{2k}{\sqrt{3}\eps}  \int_0^t e^{-\frac{t-s}{\eps^2}} \left( C_{\ell_*}^k \eps^{\ell_*} k^{\ell_*} e^{-\lambda_2 k^2 s} \right)ds \label{eqn:integrate} \nonumber \\
	&=& \frac{2k}{\sqrt{3}\eps}  \int_0^t e^{-\frac{t-s}{\eps^2}} \left( C_{\ell_*}^k \eps^{\ell_*} k^{\ell_*} e^{-\lambda_2 k^2 s} \right)ds \nonumber \\
	&=& \frac{2}{\sqrt{3}}\frac{C_{\ell_*}^k}{1-\lambda_2 \eps^2 k} \eps^{{\ell_*}+1} k^{{\ell_*}+1} \left(e^{-\lambda_2 k^2 t}  - e^{-\frac{t}{\eps^2}}\right) \nonumber \\
	&\leq& A C_{\ell_*}^k \eps^{{\ell_*}+1} k^{{\ell_*}+1} e^{-\lambda_2 k^2 t} \nonumber \\
	&=& C_{{\ell_*}+1}^k \eps^{{\ell_*}+1} k^{{\ell_*}+1} e^{-\lambda_2 k^2 t},
	\label{eqn:induction_rhs_1}
	\end{eqnarray}
	where $C_{{\ell_*}+1}^k$ is given in \eqref{eqn:clk} and we have again used the fact that $\gnk = 0$ for $n \geq 1$.  This proves \eqref{eqn:est_flk_low_strong} and hence \eqref{eqn:est_flk_low}.
\end{itemize}

To show \eqref{eqn:est_fl_low}, we sum \eqref{eqn:est_flk_low} over all low frequency values of $k$:
\begin{eqnarray}
\|\Flow_\ell\|^2_{L^2(dx)}(t) 
&=& \sum_{|k|\eps \le \frac12} |\flk|^2(t) 
\leq |\flzero|^2(t)  + 2 \sum_{0 < k\eps \le \frac12} |\flk|^2(t)   \nonumber \\
&\leq& \delta_{\ell,0}|g_{0,0}|^2  + 2 \sum_{0 < k\eps \le \frac12} \left\{ ( C_\ell^k)^2 \eps^{2\ell} k^{2\ell} e^{-2\lambda_2 k^2 t} \right\} 
\nonumber \\
&\leq& \delta_{\ell,0}|g_{0,0}|^2  + 2 \eps^{2\ell}  \max_{0 < k\eps \le \frac12}(C_\ell^k)^2 \sum_{0 < k\eps \le \frac12}  k^{2\ell} e^{-2\lambda_2 k^2 t} 
\nonumber \\
&\leq&  F(g,\ell,t)^2 \eps^{2\ell},
\label{eqn:fl_low_estimate}
\end{eqnarray}
where $F(g,\ell,t)$ is given in \eqref{eqn:C_glt}.  
 This proves \eqref{eqn:est_fl_low}.

%% file: sections/proof_error_expansion.tex
\subsection{Estimating  $\eta$} 
\label{sec:error_expansion}
In this section, we use the bounds on $\fl$ to bound $\eta$.
\begin{lem}\label{thm:eta} 
	Let $g_0 \in L^2(dx)$ be given.
	\beq \label{eqn:est_eta_high}
	\|\etahigh\| _{L^2(d \mu dx)} (t) 
	\leq  \sqrt{6}  \|\ghigh\|_{L^2(d \mu dx)} e^{-\frac{\lambda_1 t}{\eps^2}},
	\eeq
	and
	\beq
	\label{eqn:est_eta_low}
	\|\etalow\| _{L^2(d \mu dx)} (t) 
	\leq \sqrt2 F(g,N+1,t) \eps^{N+1},
	\eeq
	where $F$ is given in \eqref{eqn:C_glt} with $\ell = N+1$.
	As a result,
	\beq
	\label{eqn:est_eta}
	\|\eta\| _{L^2(d \mu dx)} (t) 
	\leq  \sqrt{6} \|\ghigh\|_{L^2(d \mu dx)} e^{-\frac{\lambda_1 t}{\eps^2}}
	+ \sqrt2 F(g,N+1,t) \eps^{N+1}.
	\eeq	
\end{lem}

\begin{proof}
We prove \eqref{eqn:est_eta_high} and \eqref{eqn:est_eta_low}, which combine to give  \eqref{eqn:est_eta}.
\begin{itemize}
	\item[(i)] {\bf High frequencies.} 
We first recall the high-frequency definitions in \eqref{eqn:high_low_def_ell} and \eqref{eqn:high_low_def} and the definition of $\eta$ in \eqref{eqn:def_el}.  We then apply \eqref{eqn:est_fl_high}.  This gives
\beq
\|\etahigh\| _{L^2(d \mu dx)} (t) 
\leq \|\Fhigh\|_{L^2(d \mu dx)} (t)
\leq  \sqrt{6}  \|\ghigh\|_{L^2(d \mu dx)} e^{-\frac{\lambda_1 t}{\eps^2}}.
\eeq

	\item[(ii)] {\bf Low frequencies.} 
We recall the low-frequency definitions in \eqref{eqn:high_low_def_ell} and \eqref{eqn:high_low_def} and then apply the bound in \eqref{eqn:est_fl_low}.  This gives 
\beq
\|\etalow\| _{L^2(d \mu dx)}^2 (t)
=  \sum_{\ell=N+1}^\infty \|\Flow_\ell\|^2_{L^2(dx)}(t)
\leq  \sum_{\ell=N+1}^\infty  \left[F(g,\ell,t) \eps^\ell\right]^2 
\eeq
Using the definition of $F(g,\ell,t)$ in \eqref{eqn:C_glt}, we have 
\begin{eqnarray}
\sum_{\ell=N+1}^\infty  \left[F(g,\ell,t) \eps^\ell\right]^2 
& = &  \sum_{\ell=N+1}^\infty \left(24 \max\limits_{k>0} \Hzerok(g) \sum\limits_{k > 0}(Ak)^{2\ell} e^{-2\lambda_2 k^2 t} \eps^{2\ell} 
\right)\\
& = &  \eps^{2(N+1)}   24 \max\limits_{k>0} \Hzerok(g)  \sum\limits_{k > 0} (Ak)^{2(N+1)} e^{-2\lambda_2 k^2 t}  \sum_{\ell=0}^\infty  (Ak\eps)^{2\ell} \\
&= &   \left[F(g,N+1,t) \right]^2 \eps^{2(N+1)} \sum_{\ell=0}^\infty  (Ak\eps)^{2\ell} .
\end{eqnarray}
Recall that $0 < A < 1.2$.  Therefore, $(Ak\eps) < 0.6 < 1$ and  
\beq
\sum_{\ell=0}^\infty  (Ak\eps)^{2\ell} = \frac{1}{1-(Ak\eps)^2} < 2. 
\eeq
\end{itemize}

\end{proof}

%% file: sections/modified_energy_xi.tex
\subsection{Estimating $\xi$}	 
\label{sec:energy_error}

With bounds on $\fNpok$ in \eqref{eqn:est_flk_high}, we use method of modified energy to bound  $\Hzerok(\xi)$ and then estimate $\xi$.

\begin{lem}\label{thm:energy_xi} Let $g_0 \in H^1(dx)$, then
	
	\begin{empheq}[left= \Hzerok(\xi) (t) \leq
		\empheqlbrace]{alignat=2}
		&6t e^{-\frac{2\lambda_1 t}{\eps^2}} \Hzerok(\partx g),
		&\quad  |k|\eps > \frac12, 
		\label{eqn:est_Hk_high_xi} \\
		&\frac{t}{2} [C_{N+1}^k]^2 k^{2(N+2)} e^{-2\lambda_2 k^2 t} \eps^{2(N+1)} ,
		&\quad  |k|\eps \leq \frac12,
		\label{eqn:est_Hk_low_xi}
	\end{empheq}	
	with $\lambda_1 = \frac{1}{45}$, $\lambda_2 = \frac{4}{45}$  and $C_{N+1}^k$ defined in \eqref{eqn:clk}.  Hence
	\beq
	\label{eqn:est_Hhigh_xi}
	\| \xihigh\|_{L^2(d \mu dx)} (t) \leq \sqrt{6t}  e^{-\frac{\lambda_1 t}{\eps^2}} \| \partx \ghigh\|_{L^2(d \mu dx)} 
	\eeq 	
	and 
	\beq
	\| \xilow\|_{L^2(d \mu dx)}(t) \leq \frac{\sqrt{t} }{A} F(g,N+2,t) \eps^{N+1} ,
	\label{eqn:est_hlow_xi}
	\eeq
	where $F$ is given in \eqref{eqn:C_glt} with $\ell = N+2$.  	As a result,
	\beq
	\label{eqn:est_xi}
	\|\xi\|_{L^2(d \mu dx)} (t)
	\leq \sqrt{6t} \|\partx g\|_{L^2(d \mu dx)} e^{-\frac{\lambda_1 t}{\eps^2}} + \frac{\sqrt t}{ A} F(g,N+2,t) \eps^{N+1}.
	\eeq
	
\end{lem}

\begin{proof}
The proof relies on the same calculations as Lemma \ref{thm:energy_f}, but must incorporate the presence of a source term in the energy equation. (Compare \eqref{eqn:partt_H0} to \eqref{eqn:partt_H0_error}.)  We set $u = \xi$ in \eqref{eqn:partt_h}, add the result to \eqref{eqn:partt_H0_error}, and use the fact that $\Honek(\xi) = \Hthreek(\xi) +  \frac12 |\eonek|^2  + \frac12 |\etwok|^2$.  This gives 
\begin{align}
\label{eqn:partt_Hh_error}
\partt  \left( \Hzerok(\xi) + \hgammak(\xi) \right)
&\leq - \frac{2}{\eps^2} \Hthreek(\xi) - \sum_{\ell=0}^2 c_{\gamma,\ell} |\xi_{\ell,k}|^2
+ \frac{1}{2\eps^2}|\eNk|^2  +  \frac{k^2}{6} |\fNpok|^2 \nonumber \\
&\leq - \frac{1}{\eps^2} \Hthreek(\xi) - \sum_{\ell=0}^2 c_{\gamma,\ell} |\xi_{\ell,k}|^2
 +  \frac{k^2}{6} |\fNpok|^2,
\end{align}
where the coefficients $c_{\gamma,0}$, $c_{\gamma,1}$, and $c_{\gamma,2}$ are defined in \eqref{eqn:def_c_gamma} and the term $\frac{1}{2\eps^2}|\eNk|^2 $ in the first line has been absorbed by $-\frac{2}{\eps^2} \Hthreek(\xi)$.%
	\footnote{The cost of combining these two terms is that the coefficient of $\Hthreek(\xi)$ in \eqref{eqn:partt_Hh_error} is only half the coefficient of $\Hthreek(f)$ in \eqref{eqn:partt_Hh}.  However, the bound with respect to these coefficients is very loose.  Hence the estimates in the proof of Lemma \ref{thm:energy_f} follow, except for the source term.
	}	
	  As in the proof of Lemma \ref{thm:energy_f}, we separate the frequency spectrum of $\xi$ into high frequency and low-frequency parts, and choose $\gamma$ appropriately in each case.

\begin{itemize}
	\item[(i)] {\bf High frequency.} For $k\eps > 1/2$, we set $\gamma = \gamma^{\rm{high}}$ (defined in \eqref{eqn:gamma_1}) into \eqref{eqn:partt_Hh_error} and repeat the arguments in part (i) of the proof of Lemma \ref{thm:energy_f}.
	This gives
	\begin{eqnarray}
	\partt \left( \Hzerok(\xi) + \hgammak(\xi) \right) 
	\leq -\frac{2\lambda_1}{\eps^2}(\Hzerok(\xi)  + \hgammak(\xi) )  +  \frac{k^2}{6} |\fNpok|^2.
	 \label{eqn:pt_error_bound_high}
	\end{eqnarray}
	We integrate \eqref{eqn:pt_error_bound_high} in time.  Using \eqref{eqn:est_flk_high} to evaluate $|\fNpok|$ and the fact that $\left(\Hzerok(\xi) + \hgammak(\xi)\right)(0)= 0$, we find that
	
	\beq
	\left(\Hzerok(\xi) + \hgammak(\xi)\right)(t)
		\leq \frac{k^2}{6}  \int_0^t e^{-\frac{2\lambda_1 (t-s)}{\eps^2}} |\fNpok|^2(s) ds
			\leq 2 t e^{-\frac{2\lambda_1 t}{\eps^2}} k^2 \Hzerok(g).
	\eeq

	Applying the left bound from \eqref{eqn:H_bound_gamma}, we find 
	\beq \label{eqn:H0_e_partx}
	\Hzerok(\xi) (t) 
	\leq   6t e^{-\frac{2\lambda_1 t}{\eps^2}} k^2 \Hzerok(g)
	= 6t e^{-\frac{2\lambda_1 t}{\eps^2}} \Hzerok(\partx g),
	\eeq
	which is  \eqref{eqn:est_Hk_high_xi}. 
	Then \eqref{eqn:est_Hhigh_xi} is recovered by summing over all high frequency values of $k$.

	\item[(ii)] {\bf Low frequency.} When $k=0$, \eqref{eqn:PN_error_k} implies that $\elk =0$ (since the initial condition is zero by definition).  For $0< k\eps \leq 1/2$, we set $\gamma = \gamma^{\rm{low}}$ (defined in \eqref{eqn:gamma_2}) into \eqref{eqn:partt_Hh_error} and repeat the arguments in part (ii) of the proof of Lemma \ref{thm:energy_f}.
	This gives
		\beq \label{eqn:parrt_Hxi}
		\partt \left( \Hzerok(\xi) + \hgammak(\xi) \right)
		\leq   - 2\lambda_2 k^2 \left( \Hzerok(\xi) + \hgammak(\xi) \right) +  \frac{k^2}{6} |\fNpok|^2.
		\eeq
		We integrate \eqref{eqn:parrt_Hxi} in time, using the fact $\left(\Hzerok(\xi) + \hgammak(\xi)\right)(0)= 0$ and the estimate in \eqref{eqn:est_flk_low} for $|\fNpok|$.  This gives
		\begin{eqnarray}
		 \label{eqn:enery_error} 
		\left(\Hzerok(\xi) + \hgammak(\xi)\right)(t)  &\leq&  \frac{k^2}{6} e^{-2\lambda_2 k^2t}  \int_0^t \! e^{2\lambda_2 k^2s} |\fNpok|^2 \, ds 
		\nonumber \\
		&\leq&  \frac{t}{6}(C_{N+1}^k)^2 k^{2(N+2)}  e^{-2\lambda_2 k^2t} \eps^{2(N+1)}. 
		\end{eqnarray}
		To arrive at \eqref{eqn:est_Hk_low_xi} from \eqref{eqn:enery_error}, we use the fact that $\Hzerok(\xi) \leq 3(\Hzerok(\xi) + \hgammak(\xi))$ (cf. \eqref{eqn:H_bound_gamma}).  Then to establish \eqref{eqn:est_hlow_xi}, we sum \eqref{eqn:est_Hk_low_xi} over the low frequency values of $k$:
		\begin{eqnarray}
		\| \xilow\|^2_{L^2(d \mu dx)}(t)
		& \leq & \sum\limits_{|k|\eps \le \frac12}\frac{12t}{A^2} (Ak)^{2(N+2)} e^{-2\lambda_2 k^2 t}  \Hzerok( g)  \eps^{2(N+1)} \nonumber\\\
		& \leq & \frac{24t}{A^2} \eps^{2(N+1)} \max_{0 < k\eps \leq \frac12} \Hzerok( g)   \sum\limits_{0<k\eps \le \frac12} (Ak)^{2(N+2)} e^{-2\lambda_2 k^2 t}  \nonumber\\\
		& \leq & \frac{t}{A^2} \eps^{2(N+1)} \left[F(g,N+2,t) \right]^2,
		\end{eqnarray}
		with $F$ defined in \eqref{eqn:C_glt}.
\end{itemize}
\end{proof}

%% file: sections/proof_error_new.tex
\subsection{Finer estimate on $\xillow$} 
\label{sec:error}
In Lemma \ref{thm:energy_xi}, we proved an $\eps$-dependent estimate for $\xilow$.  In this section, we combine the method of induction with a reduced version of the modified energy to refine the estimate for $\xillow$.
\begin{lem} \label{thm:xi_l_1st}
	Let $g_0 \in L^2(dx)$ be given.  
	For $|k|\eps \leq \frac12$, 
		\begin{empheq}[left=	|\elk| (t)  \leq 
			\empheqlbrace]{alignat=2}
			&\tilde{C}_{N-1}^k(t) \eps^{2N} k^{3N}e^{-\lambda_2 k^2 t},
			&\quad  \ell=0,
			\label{eqn:est_xi_lk_low_0} \\
			&\tilde{C}_{N+1-\ell}^k(t) \eps^{2N+2-\ell} k^{3N+4-2\ell} e^{-\lambda_2 k^2 t}, 
			&\quad  1 \leq \ell \leq N, 
			\label{eqn:est_xi_lk_low_1}
		\end{empheq}	
	where 
	\beq \label{eqn:cnk_hat}
	\tilde{C}_{N+1-\ell}^k(t)= M(t) \tilde{C}_{N-\ell}^k(t), \quand \tilde{C}_1^k(t) = \max\{1, t^{1/2}\} \hat{C}(t)  C^k_{N+1},
	\eeq
	with 
	\beq \label{eqn:def_M_Chat}
	M(t) = \frac{2\max\{1, t^{1/2}\}}{\sqrt{3}(1-\lambda_2/4)} \quand \hat{C}(t) = \frac{2t^{1/2} + k^{-1}}{\sqrt3 (1-\lambda_2/4)},
	\eeq
	and $C^k_{N+1}$ defined in \eqref{eqn:clk}.   
\end{lem}

\begin{proof}

\input{sections/proof_error_inside_new}

\end{proof}

The coefficients $\tilde{C}_{N-1}^k(t)$ and $\tilde{C}_{N+1-\ell}^k(t)$ in the estimates for $\elk(t)$ can be replaced by some time independent coefficients, at the cost of a reduced decay rate in the error.   
\begin{lem}\label{thm:xi_l}
	Let $g_0 \in L^2(dx)$ be given.  
	For $|k|\eps \leq \frac12$, 
	\begin{empheq}[left=	|\elk| (t)  \leq 
		\empheqlbrace]{alignat=2}
		& \bar{C}(N,2) C^k_{N+1} \eps^{2N} k^{3N} e^{-\frac{\lambda_2}{2} k^2 t}, 
		&\quad  \ell=0,
		\label{eqn:est_xi_lk_low_0_new} \\
		& \bar{C}(N,\ell) C^k_{N+1} \eps^{2N+2-\ell} k^{3N+4-2\ell} e^{-\frac{\lambda_2}{2} k^2 t}, 
		&\quad  1 \leq \ell \leq N,
		\label{eqn:est_xi_lk_low_1_new}
	\end{empheq}	
where
	\beq
		\bar{C}(N,\ell) = 2 A^{N-\ell+1}  \left( \frac{N-\ell+2}{\lambda_2}\right)^\frac{N-\ell+2}{2} e^{-\frac{N-\ell}{2}+\frac{\lambda_2}{2}-1}
	\eeq
and $C^k_{N+1}$ is defined in \eqref{eqn:clk}.   
Hence 
	\begin{empheq}[left=	\|\xillow\|_{L^2(dx)} (t)  \leq 
		\empheqlbrace]{alignat=2}
		&E(g,N,2,t) \eps^{2N},
		&\quad \ell=0,
		\label{eqn:est_xi_low_0} \\
		&E(g,N,\ell,t) \eps^{2N+2-\ell}, 
		&\quad 1 \leq \ell \leq N, 
		\label{eqn:est_xi_low_1}
	\end{empheq}	
where
	\beq \label{eqn:E_def}
		E(g,N,\ell,t) = \bar{C}(N,\ell)   A^{-2N-3+2\ell}  F(g,3N+4-2\ell, t/2)
	\eeq
and  $F$ is defined in \eqref{eqn:C_glt}.

\end{lem}

\begin{proof}
		The strategy is simple: use part of the exponentially decaying term in \eqref{eqn:est_xi_lk_low_0} and \eqref{eqn:est_xi_lk_low_1} to control powers of $t$ in the other coefficients.
		Since $\max\{1, t^{1/2}\} \leq (t+1)^{1/2}$, \eqref{eqn:cnk_hat} and \eqref{eqn:def_M_Chat} imply the following bound: 
		\begin{align}
		\tilde{C}_{N+1-\ell}^k(t) e^{-\lambda_2 k^2 t}
		& \leq A^{N-\ell+1}   (t+1)^{(N-\ell+1)/2}  \left( t^{1/2} + \frac{1}{2k} \right) C^k_{N+1}  e^{-\lambda_2 k^2 t} \nonumber \\
		& \leq  2 A^{N-\ell+1}   (t+1)^{(N-\ell+2)/2} e^{-\frac{\lambda_2}{2} t} C^k_{N+1} e^{-\frac{\lambda_2}{2} k^2 t}, \quad k >1,
		\end{align}
		The product $2 A^{N-\ell+1}  (t+1)^{(N-\ell+2)/2}   e^{-\frac{\lambda_2}{2} t}$ takes its maximum value $\bar{C}(N,\ell) $ at $t = (N-\ell+2)/\lambda_2 - 1$.  This proves \eqref{eqn:est_xi_lk_low_0_new} and \eqref{eqn:est_xi_lk_low_1_new}.

		To establish \eqref{eqn:est_xi_low_0} and \eqref{eqn:est_xi_low_1}, we sum \eqref{eqn:est_xi_lk_low_0_new} and \eqref{eqn:est_xi_lk_low_1_new}, respectively, over all low frequency values of $k$.  
		For example, summing \eqref{eqn:est_xi_lk_low_1_new} gives
		\begin{eqnarray}
		\|\elow_\ell\|^2_{L^2(dx)}(t) 
		&=& \sum_{|k|\eps \le \frac12} |\elk|^2(t) 
		\le 2\sum_{0 < k\eps \le \frac12} \left( \bar{C}(N,\ell) C^k_{N+1}  \eps^{2N+2-\ell} k^{3N+4-2\ell} e^{-\frac{\lambda_2}{2} k^2 t}\right)^2  \nonumber\\
		&\leq &  (\bar{C}(N,\ell))^2 \eps^{2(2N+2-\ell)}  2 \max_{0 < k\eps \leq \frac12} ( C^k_{N+1})^2 \sum_{0 < k\eps \le \frac12}  k^{2(3N+4-2\ell)} e^{-\lambda_2 k^2 t} \nonumber\\
		&=&  (\bar{C}(N,\ell)  )^2 \eps^{2(2N+2-\ell)} 24  \max_{0 < k\eps \leq \frac12} \Hzerok(g) A^{2(N+1)}
		\sum_{0 < k\eps \le \frac12}  k^{2(3N+4-2\ell)} e^{-\lambda_2 k^2 t}  
		,\label{eqn:ell_low_estimate}
		\end{eqnarray}
		which yields \eqref{eqn:est_xi_low_1}.  Then \eqref{eqn:est_xi_low_0} is derived similarly.
		
\end{proof}

\begin{rem}
One could easily prove bounds of the form \eqref{eqn:est_xi_low_0} and \eqref{eqn:est_xi_low_1} by using \eqref{eqn:est_xi_lk_low_0} and \eqref{eqn:est_xi_lk_low_1} directly, with the coefficient $E$ replaced by $\tilde{E}(g,N,\ell,t) = \max\{1, t^{1/2}\} \hat{C}(t) M(t)^{N-\ell} A^{-2N-3+2\ell}  F(g,3N+4-2\ell,t)$.  Although $\tilde{E}$ decays more quickly for large $t$ (due to the difference in the third argument in $F$), the time-dependent factor $\max\{1, t^{1/2}\} \hat{C}(t) M(t)^{N-\ell}$ makes the analysis of $\tilde{E}$ as a function of $N$ more difficult.
 We instead use $E$ in \eqref{eqn:E_def} because it is easier to bound its growth with respect to $N$.  This fact will be useful in Section \ref{sec:constant_math}.
\end{rem}

%% file: sections/proof_error_inside_new.tex
For $k=0$, \eqref{eqn:est_xi_lk_low_0} and \eqref{eqn:est_xi_lk_low_1} hold, since direct inspection of \eqref{eqn:PN_error_k} shows that $\elzero(t)=0$ for all $\ell = 0, 1, \cdots, N$.  
We therefore consider only $0 < k \le \frac{1}{2\eps}$.  In this case, we establish \eqref{eqn:est_xi_lk_low_1} and \eqref{eqn:est_xi_lk_low_0} by proving the stronger statements:
\begin{itemize}
	\item[(1)] for $2 \leq \ell \leq N$,
	\beq
	\label{eqn:est_induction_back_1_2}
	|\enk| (t)  \leq 
	\tilde{C}_{N+1-\ell}^k(t) \eps^{2N+2-\ell} k^{3N+4-2\ell} e^{-\lambda_2 k^2 t}, 
	\quad 0 \leq n \leq \ell;
	\eeq
	\item[(2)] for $\ell = 1$,
	\beq
	\label{eqn:est_induction_back_0_2}
	|\eonek| (t)  \leq 
	\tilde{C}_{N}^k(t) \eps^{2N+1} k^{3N+1} e^{-\lambda_2 k^2 t}.
	\eeq
\end{itemize}
When $\ell = 0$, \eqref{eqn:est_xi_lk_low_0} is recovered by setting $\ell = 2$ and $n=0$ in \eqref{eqn:est_induction_back_1_2}.
When $\ell = 1$, \eqref{eqn:est_xi_lk_low_1} is recovered by \eqref{eqn:est_induction_back_0_2}.
When $2 \leq \ell \leq N$, \eqref{eqn:est_xi_lk_low_1}  is recovered by setting $n=\ell$ in \eqref{eqn:est_induction_back_1_2}.

\begin{itemize}
	\item[(1)] To prove \eqref{eqn:est_induction_back_1_2}, we use the method of induction, starting with $N$ and working backward. 
	\begin{itemize}
		\item[(a)]\label{eqn:induction_back_1}
		For the initial step in the induction, we need to show that for $\ell = N$,
		\beq
		\label{eqn:est_elk_step1}
		|\enk|(t) \le \tilde{C}_1^k(t) \eps^{N+2} k^{N+4}e^{-\lambda_2 k^2 t},\quad n = 0, 1, \dots, N.
		\eeq
		We prove \eqref{eqn:est_elk_step1} in two sub-steps.
		\begin{itemize}
			\item[(i)] The first sub-step is to show \eqref{eqn:est_elk_step1} for $n = 1, \dots, N$.  We combine the last two equations of (\ref{eqn:PN_error_k}). This gives
			\beq \label{eqn:ineq_enk}
			\partt |\enk| +  \frac{1}{\eps^2} |\enk| 
			\le \frac{k}{\eps} \left( a_{n-1} |\enmok| + (1-\delta_{n,N}) a_n |\enpok| + \delta_{n,N} a_N |\fNpok|  \right),
			\eeq  
			for $1 \le n \le N$. It follows from \eqref{eqn:est_Hk_low_xi} that
			\beq
			\label{ineq:H0_e}
			|\elk|(t) \le t^{1/2} C^k_{N+1} \eps^{N+1} k^{N+2} e^{-\lambda_2 k^2 t}, \quad \ell = 0, 1, \cdots, N, \quad 0 < k \le \frac{1}{2\eps},
			\eeq
			where $C^k_{N+1}$ is defined in \eqref{eqn:clk}.  We use \eqref{ineq:H0_e} to estimate $\enmok$ and $\enpok$ and \eqref{eqn:est_flk_low} to estimate $\fNpok$.  Then \eqref{eqn:ineq_enk} reduces to 
			\beq
			\partt |\enk| +  \frac{1}{\eps^2} |\enk| 
			\le \frac{1}{\sqrt 3} \left( 2t^{1/2} + \frac{1}{k} \right) C^k_{N+1} \eps^N k^{N+3}e^{-\lambda_2 k^2 t}, \quad 1 \leq n \leq N.
			\eeq
			We then integrate in time, using the zero initial condition for $\enk$ to find 
			\beq \label{eqn:est_xilk_intermediate}
			|\enk|(t) \le \hat{C}(t) C^k_{N+1}  \eps^{N+2} k^{N+3}e^{-\lambda_2 k^2 t}, \quad 1 \leq n \leq N,
			\eeq
			with $\hat{C}(t)$ defined in \eqref{eqn:def_M_Chat}.  Since $\hat{C} C^k_{N+1} \leq \tilde{C}_1^k $, \eqref{eqn:est_xilk_intermediate} verifies \eqref{eqn:est_elk_step1} for $n = 1, \dots, N$. \footnote{Note that power of $k$ in \eqref{eqn:est_xilk_intermediate} is $N+3$, which is actually better than the estimate in \eqref{eqn:est_elk_step1}. However in the second substep, an additional power of $k$ is needed (cf. \eqref{eqn:est_elk_update}) in order to gain an additional factor of $\eps$ in the estimate for $\ezerok$. }
			\item[(ii)] The second sub-step is to use the result of (i) to show that \eqref{eqn:est_elk_step1} holds for $n = 0$.  Since using \eqref{eqn:PN_error_k} directly will result in order reduction by one power of $\eps$, we instead consider the smaller system for $\left\{ \enk \right\}_{n=0}^{N-1}$ and treat $\eNk$ as a source term:
			
			\beq
			\left\{
			\begin{array}{ll}
				\eps\partt \ezerok + a_0 ik \eonek = 0,
				& n = 0;\\
				\eps\partt \enk + a_n ik \enpok + a_{n-1} ik \enmok + \frac{1}{\eps} \enk = 0,
				& 1 \le n \le N-2;\\
				\eps\partt \eNmok  + a_{N-2} ik \eNmtk + \frac{1}{\eps} \eNmok = -a_{N-1} ik \eNk,
				& n = N-1.
			\end{array}
			\right.
			\eeq 
			
			We then repeat the arguments used to establish \eqref{eqn:est_Hk_low_xi} using the estimate for $\eNk$ in \eqref{eqn:est_xilk_intermediate} instead of the estimate for $\fNpok$.  This procedure  requires the introduction of a new functional
			\beq
			\label{def:Hjik} 
			\Hjik \colon u \mapsto \frac{1}{2}\sum_{n=j}^i |u_{n,k}|^2, \quad u \in L^2(d \mu dx),
			\eeq
			which is defined such that $\HjNk = \Hjk $ (cf. \eqref{eqn:Hkj}).  With $\HzeroNmok(\xi)$ and compensating function $\hgammak(\xi)$, defined in \eqref{eqn:hgamma}, one can first derive a differential inequality analogous to \eqref{eqn:partt_Hh_error} and then follow the arguments in part (ii) of the proof of Lemma \ref{thm:energy_xi}.  The result is 
			\beq
				\HzeroNmok(\xi)(t) \leq \frac{t}{2} \hat{C}(t)^2 (C_{N+1}^k)^2 \eps^{2(N+2)} k^{2(N+4)}e^{-2 \lambda_2 k^2 t},
			\eeq
			and then
			\beq
			\label{eqn:est_elk_update}
			|\enk|(t) \leq t^{1/2} \hat{C}(t) C_{N+1}^k\eps^{N+2} k^{N+4}e^{-\lambda_2 k^2 t}, \quad 0 \leq n \leq N-1.
			\eeq
			As compared to \eqref{eqn:est_Hk_low_xi}, the extra powers of $\eps$ and $k$ in \eqref{eqn:est_elk_update} come from the higher powers in the estimate for $\eNk$ in \eqref{eqn:est_xilk_intermediate} when compared to the estimate for $\fNpok$ in \eqref{eqn:est_flk_low}. Since $t^{1/2} \hat{C} C^k_{N+1} \leq \tilde{C}_1^k $, \eqref{eqn:est_elk_update} verifies \eqref{eqn:est_elk_step1} for $n = 0$. 
		\end{itemize}

		\item[(b)] For the next step of the induction, we assume that for some $3 \leq \ell_* \leq N$ fixed, \eqref{eqn:est_induction_back_1_2} holds for $\ell = \ell_*$:
		\beq \label{eqn:est_xink}
		|\enk| (t)  \leq 
		\tilde{C}_{N+1-{\ell_*}}^k(t) \eps^{2N+2-{\ell_*}} k^{3N+4-2{\ell_*}} e^{-\lambda_2 k^2 t}, \quad 0 \leq n \leq \ell_*.
		\eeq
		
		We would like to show 
		\beq
		\label{eqn:est_induction_result}
		|\enk|(t) 
		\leq   \tilde{C}_{N+1-({\ell_*}-1)}^k(t) \eps^{2N+2-({\ell_*}-1)} k^{3N+4-2({\ell_*}-1)}e^{-\lambda_2 k^2 t}, \quad  0 \leq n \leq {\ell_*}-1.
		\eeq

		\begin{itemize}
			\item[(i)] We first show \eqref{eqn:est_induction_result} for $1 \leq n \leq \ell_* - 1$.  Using \eqref{eqn:PN_error_k}, the estimate \eqref{eqn:est_xink} for $\enmok$ and $\enpok$ and the fact that $a_{n-1} \leq 1/\sqrt{3}$ and $a_n \leq 1/\sqrt{3}$, we derive the following estimate for $|\enk|$ for $1 \leq n \leq \ell_* - 1$:
			\begin{align}
			\partt |\enk| +  \frac{1}{\eps^2} |\enk| 
			&\leq \frac{k}{\eps} \left( a_{n-1} |\enmok| + a_n |\enpok| \right) \nonumber \\
			&\leq \frac{2}{\sqrt{3}}\tilde{C}_{N+1-{\ell_*}}^k(t) \eps^{2N+1-{\ell_*}} k^{3N+5-2{\ell_*}}e^{-\lambda_2 k^2 t}.
			\label{eqn:partt_enk}
			\end{align}
			Integration of \eqref{eqn:partt_enk} in time gives
			\beq
			|\enk|(t) 
			\leq \frac{2}{\sqrt{3}(1-\lambda_2/4)} \tilde{C}_{N+1-{\ell_*}}^k(t) \eps^{2N+3-{\ell_*}} k^{3N+5-2{\ell_*}}e^{-\lambda_2 k^2 t},  \label{eqn:est_enk_update}
			\eeq
			which recovers \eqref{eqn:est_induction_result}, using the definition of $\tilde{C}_{N+2-{\ell_*}}^k $ in  \eqref{eqn:cnk_hat}. 
			\item[(ii)] We next show \eqref{eqn:est_induction_result} for $n = 0$, repeating the argument from step (a)(ii).
			We consider the smaller system for $\left\{ \enk \right\}_{n=0}^{\ell_*-2}$ and treat $\elstarmok$ as a source term:
			\beq
			\left\{
			\begin{array}{ll}
				\eps\partt \ezerok + a_0 ik \eonek = 0,
				& n = 0;\\
				\eps\partt \enk + a_n ik \enpok + a_{n-1} ik \enmok + \frac{1}{\eps} \enk = 0,
				& 1 \le n \le  \ell_*-3;\\
				\eps\partt \elstarmtk  + a_{\ell_*-3} ik \elstarmthk + \frac{1}{\eps} \elstarmtk = -a_{\ell_*-2} ik \elstarmok,
				& n = \ell_*-2.
			\end{array}
			\right.
			\eeq 
			Using the energy $\Hzerolsmtwok(\xi)$ defined in \eqref{def:Hjik} and compensating function $\hgammak(\xi)$, defined in \eqref{eqn:hgamma},  we have
			\beq
			\Hzerolsmtwok(\xi) (t) \leq  \frac{2t}{\left(\sqrt{3}(1-\lambda_2/4)\right)^2} \left(\tilde{C}_{N+1-{\ell_*}}^k(t)\right)^2 \eps^{2(2N+3-{\ell_*})} k^{2(3N+6-2{\ell_*})}e^{-2\lambda_2 k^2 t},
			\eeq
			and then 
			\beq
			\label{eqn:est_ezerok_update}
			|\enk|(t) \le \frac{2t^{1/2}}{\sqrt{3}(1-\lambda_2/4)} \tilde{C}_{N+1-{\ell_*}}^k(t) \eps^{2N+3-{\ell_*}} k^{3N+6-2{\ell_*}}e^{-\lambda_2 k^2 t}, \quad 0 \leq n \leq \ell_*-2.
			\eeq
			This estimate is analogous to \eqref{eqn:est_elk_update}.
		\end{itemize}
	\end{itemize}

	\item[(2)] 
	To prove \eqref{eqn:est_induction_back_0_2}, one just need to repeat the argument in (b)(i) with $\ell_* = 2$. 
\end{itemize}

%% file: sections/proof_main_theorem.tex
\subsection{Proof of Theorem \ref{thm:error}} 
\label{sec:proof_main}

We first prove the error estimate for $f^N$. Since $f - f^N = \eta + \xi$, we simply apply the triangle inequality and combine the estimates for $\eta$ in \eqref{eqn:est_eta} and $\xi$ in \eqref{eqn:est_xi}, and get 
\begin{align}
	\|f - f^N\|_{L^2(d \mu dx)} (t) \leq &
	\sqrt{6} \|\ghigh\|_{L^2(d \mu dx)} e^{-\frac{\lambda_1 t}{\eps^2}} + \sqrt2 F(g,N+1,t) \eps^{N+1} \nonumber \\
	& + \sqrt{6t} \|\partx g\|_{L^2(d \mu dx)} e^{-\frac{\lambda_1 t}{\eps^2}} +
	\frac{\sqrt t}{ A} F(g,N+2,t) \eps^{N+1}.
\end{align}
This establishes \eqref{eqn:L2_error} with constants
	\beq
	B(g) =  \sqrt{6}  \|g\|_{L^2(d \mu dx)},
	\quad
	C(\partx g) = \sqrt{6} \|\partx g\|_{L^2(d \mu dx)},
	\quad
	\label{eqn:D_gnt}
	D(g,N,t) = \sqrt{2} F(g,N+1,t) + \frac{\sqrt t}{ A} F(g,N+2,t).
	\eeq
We next prove the error estimate for $\fNl$.   
Since $\fl - \fNl = \el = \xilhigh + \xillow$, we combine the estimate $\eqref{eqn:est_Hhigh_xi}$ with \eqref{eqn:est_xi_low_0} and \eqref{eqn:est_xi_low_1}. 
After some additional trivial bounds ($\| \xilhigh\|_{L^2(dx)} \leq \| \xihigh\|_{L^2(d \mu dx)}$ and $\| \partx \ghigh\|_{L^2(d \mu dx)} \leq  \|\partx g\|_{L^2(d \mu dx)}$), we arrive at
	\begin{empheq}[left=	\| \fl - \fNl \| _{L^2(dx)}  (t)  \leq 
	\empheqlbrace]{alignat=2}
	& \sqrt{6t}  e^{-\frac{\lambda_1 t}{\eps^2}} \| \partx g\|_{L^2(d \mu dx)} 
	+ E(g,N,2,t) \eps^{2N},
	&\quad \ell=0,\\
	&\sqrt{6t}  e^{-\frac{\lambda_1 t}{\eps^2}} \| \partx g\|_{L^2(d \mu dx)}
	+ E(g,N,\ell,t) \eps^{2N+2-\ell}, 
	&\quad 1 \leq \ell \leq N,
	\end{empheq}
with $E$ defined in \eqref{eqn:E_def}. This establishes \eqref{eqn:moment_error} and completes the proof.

%% file: sections/numerics.tex
\section{Numerical Examples} 
\label{sec:numerics}
In this section, we perform numerical tests to demonstrate the theoretical results, by exploring different  values of $\eps$, $N$, and the initial condition $g$.  All calculations are based on the $P_N$ system \eqref{eqn:PN_moments}.  Since the exact solution of $f$ is not readily available, we use $f^N$ with $N=65$ as a reference solution in order to calculate $L^2$ errors, and as discussed in Remark \ref{thm:remark_fNl}, we use $\fNl$ as a proxy for the coefficients $\fl$ of the exact solution in order to test the asymptotic estimates in Lemma \ref{thm:f}.   

For the spatial discretization of \eqref{eqn:PN_moments}, we use a Fourier-Galerkin method, typically with $100$ Fourier modes, although more modes are added as needed to ensure that the spatial error neglected can be neglected.  The method is implemented with a fast Fourier transform (FFT) algorithm.  What remains is an ODE for the Fourier-Galerkin coefficients that can be solved exactly (up to machine precision).  In some situations, the size of the coefficients differs by many orders of magnitude.  Thus in order to rule out the effect of cumulative round-off error that this discrepancy may create, we use the Multiprecision Computing Toolbox for MATLAB by Advanpix LLC. \cite{advanpixmultiprecision} with $250$ digits.

\begin{eg}\normalfont \label{eg:test1}
	We start with the kinetic equation \eqref{eqn:transport_simple} with three different initial conditions:
\beq  \label{eqn:num1_test1_ic}
\left.
\begin{array}{l}
	g(x,\mu) = g^{(1)} (x) = 1 + 1_{[-\frac{\pi}{2},\frac{\pi}{2}]} (x), \\
	g(x,\mu)  = g^{(2)}(x) = 1 + \cos(x) 1_{[-\frac{\pi}{2},\frac{\pi}{2}]} (x), \\
	g(x,\mu)  = g^{(3)}(x) = 1 + \cos(x).
\end{array}
\right.
\eeq	
Simple calculations with Fourier analysis imply  $g^{(1)} \in H^q(dx)$ for $q < \frac12$ and $g^{(2)} \in H^q(dx)$ for $q < \frac32$, while $g^{(3)}$ is a smooth function.  However, $g^{(1)}$ does not satisfy the regularity assumption in
Theorem \ref{thm:error} which is required for the high-frequency bound \eqref{eqn:est_Hk_high_xi} in Lemma \ref{thm:energy_xi}.

We solve the $P_N$ system \eqref{eqn:PN_moments} with $\eps = 2 \cdot 4^{-m}, $ with $ m=1,\dots, 5$, and $N = 1, 2, \cdots, 5$.  
For each $\eps$ and $N$, $L^2(d\mu dx)$ errors with respect to the reference solution are listed in Table \ref{tab:test1_total_errors}.  
The convergence rates of the coefficients $\fNl$ for $P_4$ and $P_5$ are listed in Tables \ref{tab:test1_P4_f} and \ref{tab:test1_P5_f}.  
The convergence rates of the errors $ \el$ for $\fNl$ are listed in Tables \ref{tab:test1_P4_e} and \ref{tab:test1_P5_e}.

\begin{table}[!htbp]
	\centering
	\footnotesize
	\begin{tabular}{|c||c|c|c|c|c|c|c|c|c|c|c|c|}
		\multicolumn{11}{c}{$g^{(1)}$}  \\ \hline
		$\eps$ & $P_1$ error & order & $P_2$ error & order & $P_3$ error & order & $P_4$ error & order & $P_5$ error & order \\ \hline
1/2    & 6.34E-02    &       & 3.27E-02    &       & 2.07E-02    &       & 1.48E-02    &       & 1.17E-02    &       \\
1/8    & 2.60E-03    & 2.30  & 1.71E-04    & 3.79  & 1.97E-05    & 5.02  & 3.39E-06    & 6.05  & 6.39E-07    & 7.08  \\
1/32   & 1.60E-04    & 2.01  & 2.50E-06    & 3.05  & 7.09E-08    & 4.06  & 3.00E-09    & 5.07  & 1.39E-10    & 6.08  \\
1/128  & 9.99E-06    & 2.00  & 3.89E-08    & 3.00  & 2.76E-10    & 4.00  & 2.91E-12    & 5.00  & 3.38E-14    & 6.01  \\
1/512  & 6.24E-07    & 2.00  & 6.07E-10    & 3.00  & 1.08E-12    & 4.00  & 2.84E-15    & 5.00  & 8.24E-18    & 6.00  \\ 
\hline
		\multicolumn{11}{c}{}\\
		\multicolumn{11}{c}{$g^{(2)}$}  \\ \hline
		$\eps$ & $P_1$ error & order & $P_2$ error & order & $P_3$ error & order & $P_4$ error & order & $P_5$ error & order \\ \hline
1/2    & 4.39E-02    &       & 1.59E-02    &       & 6.63E-03    &       & 3.44E-03    &       & 2.21E-03    &       \\
1/8    & 2.41E-03    & 2.09  & 1.76E-04    & 3.25  & 1.88E-05    & 4.23  & 2.27E-06    & 5.28  & 2.90E-07    & 6.45  \\
1/32   & 1.49E-04    & 2.01  & 2.65E-06    & 3.03  & 7.04E-08    & 4.03  & 2.11E-09    & 5.04  & 6.64E-11    & 6.05  \\
1/128  & 9.31E-06    & 2.00  & 4.14E-08    & 3.00  & 2.74E-10    & 4.00  & 2.05E-12    & 5.00  & 1.62E-14    & 6.00  \\
1/512  & 5.82E-07    & 2.00  & 6.46E-10    & 3.00  & 1.07E-12    & 4.00  & 2.00E-15    & 5.00  & 3.94E-18    & 6.00  \\
\hline
		\multicolumn{11}{c}{}\\
		\multicolumn{11}{c}{$g^{(3)}$}  \\ \hline
		$\eps$ & $P_1$ error & order & $P_2$ error & order & $P_3$ error & order & $P_4$ error & order & $P_5$ error & order \\ \hline
1/2    & 5.78E-02    &       & 1.24E-02    &       & 2.37E-03    &       & 3.94E-04    &       & 5.79E-05    &       \\
1/8    & 3.51E-03    & 2.02  & 2.14E-04    & 2.93  & 1.35E-05    & 3.73  & 8.50E-07    & 4.43  & 5.34E-08    & 5.04  \\
1/32   & 2.18E-04    & 2.00  & 3.31E-06    & 3.01  & 5.21E-08    & 4.01  & 8.18E-10    & 5.01  & 1.28E-11    & 6.01  \\
1/128  & 1.36E-05    & 2.00  & 5.17E-08    & 3.00  & 2.03E-10    & 4.00  & 7.98E-13    & 5.00  & 3.13E-15    & 6.00  \\
1/512  & 8.50E-07    & 2.00  & 8.07E-10    & 3.00  & 7.94E-13    & 4.00  & 7.80E-16    & 5.00  & 7.64E-19    & 6.00
\\ \hline
	\end{tabular}
	\caption{Errors and convergence rates for the $P_N$ solutions in Example \ref{eg:test1}.  According to Theorem \ref{thm:error}, the theoretical order of convergence is $N+1$.}

	\label{tab:test1_total_errors}
\end{table}

\begin{table}[!htbp]
	\centering
	\footnotesize
	\begin{tabular}{|c||c|c|c|c|c|c|c|c|c|c|}
		\multicolumn{11}{c}{$g^{(1)}$}  \\ \hline
		$\eps$ & $\fNzero$    & order & $\fNone$    & order & $\fNtwo$    & order 	& $\fNthree$    & order		& $\fNfour$    & order \\ \hline
1/2    & 2.16E+00 &       & 1.43E-01 &       & 3.84E-02 &       & 1.48E-02 &       & 1.35E-02 &       \\
1/8    & 2.16E+00 & 0.00  & 3.32E-02 & 1.06  & 2.19E-03 & 2.07  & 1.62E-04 & 3.25  & 1.96E-05 & 4.72  \\
1/32   & 2.16E+00 & 0.00  & 8.25E-03 & 1.00  & 1.36E-04 & 2.01  & 2.49E-06 & 3.01  & 7.09E-08 & 4.05  \\
1/128  & 2.16E+00 & 0.00  & 2.06E-03 & 1.00  & 8.48E-06 & 2.00  & 3.89E-08 & 3.00  & 2.76E-10 & 4.00  \\
1/512  & 2.16E+00 & 0.00  & 5.16E-04 & 1.00  & 5.30E-07 & 2.00  & 6.07E-10 & 3.00  & 1.08E-12 & 4.00  \\
 \hline
 \multicolumn{11}{c}{}\\
		\multicolumn{11}{c}{$g^{(2)}$}  \\ \hline
		$\eps$ & $\fNzero$    & order & $\fNone$    & order & $\fNtwo$    & order		& $\fNthree$    & order		& $\fNfour$    & order \\ \hline
1/2    & 1.91E+00 &       & 1.21E-01 &       & 3.65E-02 &       & 1.28E-02 &       & 5.90E-03 &       \\
1/8    & 1.90E+00 & 0.00  & 2.73E-02 & 1.07  & 1.99E-03 & 2.10  & 1.73E-04 & 3.10  & 1.88E-05 & 4.15  \\
1/32   & 1.90E+00 & 0.00  & 6.78E-03 & 1.00  & 1.23E-04 & 2.01  & 2.65E-06 & 3.01  & 7.04E-08 & 4.03  \\
1/128  & 1.90E+00 & 0.00  & 1.69E-03 & 1.00  & 7.69E-06 & 2.00  & 4.14E-08 & 3.00  & 2.74E-10 & 4.00  \\
1/512  & 1.90E+00 & 0.00  & 4.23E-04 & 1.00  & 4.81E-07 & 2.00  & 6.46E-10 & 3.00  & 1.07E-12 & 4.00  \\
 \hline
 \multicolumn{11}{c}{}\\
		\multicolumn{11}{c}{$g^{(3)}$}  \\ \hline
		$\eps$ & $\fNzero$    & order & $\fNone$    & order & $\fNtwo$    & order		& $\fNthree$    & order		& $\fNfour$    & order \\ \hline
1/2    & 1.61E+00 &       & 2.24E-01 &       & 5.50E-02 &       & 1.20E-02 &       & 2.36E-03 &       \\
1/8    & 1.59E+00 & 0.01  & 5.20E-02 & 1.05  & 3.36E-03 & 2.02  & 2.13E-04 & 2.91  & 1.35E-05 & 3.72  \\
1/32   & 1.59E+00 & 0.00  & 1.29E-02 & 1.00  & 2.09E-04 & 2.00  & 3.31E-06 & 3.01  & 5.21E-08 & 4.01  \\
1/128  & 1.59E+00 & 0.00  & 3.23E-03 & 1.00  & 1.30E-05 & 2.00  & 5.17E-08 & 3.00  & 2.03E-10 & 4.00  \\
1/512  & 1.59E+00 & 0.00  & 8.08E-04 & 1.00  & 8.15E-07 & 2.00  & 8.07E-10 & 3.00  & 7.94E-13 & 4.00 
\\ \hline
	\end{tabular}
	\caption{Convergence rates of the coefficients $\fNl$ for the $P_4$ solution in Example \ref{eg:test1}. According to Lemma \ref{thm:f}, the theoretical order of convergence is $\ell$.}
	\label{tab:test1_P4_f}
\end{table}

\begin{table}[!htbp]
	\centering
	\footnotesize
	\begin{tabular}{|c||c|c|c|c|c|c|c|c|c|c|c|c|}
		\multicolumn{13}{c}{$g^{(1)}$}  \\ \hline
		$\eps$ & $\fNzero$    & order & $\fNone$    & order & $\fNtwo$    & order 	& $\fNthree$    & order		& $\fNfour$    & order		& $\fNfive$    & order \\ \hline
1/2    & 2.16E+00 &       & 1.43E-01 &       & 3.80E-02 &       & 1.54E-02 &       & 1.00E-02 &       & 8.71E-03 &       \\
1/8    & 2.16E+00 & 0.00  & 3.32E-02 & 1.06  & 2.19E-03 & 2.06  & 1.62E-04 & 3.28  & 1.90E-05 & 4.52  & 3.39E-06 & 5.66  \\
1/32   & 2.16E+00 & 0.00  & 8.25E-03 & 1.00  & 1.36E-04 & 2.01  & 2.49E-06 & 3.01  & 7.08E-08 & 4.03  & 3.00E-09 & 5.07  \\
1/128  & 2.16E+00 & 0.00  & 2.06E-03 & 1.00  & 8.48E-06 & 2.00  & 3.89E-08 & 3.00  & 2.76E-10 & 4.00  & 2.91E-12 & 5.00  \\
1/512  & 2.16E+00 & 0.00  & 5.16E-04 & 1.00  & 5.30E-07 & 2.00  & 6.07E-10 & 3.00  & 1.08E-12 & 4.00  & 2.84E-15 & 5.00  \\
 \hline
 \multicolumn{13}{c}{}\\
		\multicolumn{13}{c}{$g^{(2)}$}  \\ \hline
		$\eps$ & $\fNzero$    & order & $\fNone$    & order & $\fNtwo$    & order		& $\fNthree$    & order		& $\fNfour$    & order		& $\fNfive$    & order \\ \hline
1/2    & 1.91E+00 &       & 1.21E-01 &       & 3.65E-02 &       & 1.30E-02 &       & 4.91E-03 &       & 2.85E-03 &       \\
1/8    & 1.90E+00 & 0.00  & 2.73E-02 & 1.07  & 1.99E-03 & 2.10  & 1.73E-04 & 3.11  & 1.85E-05 & 4.02  & 2.27E-06 & 5.15  \\
1/32   & 1.90E+00 & 0.00  & 6.78E-03 & 1.00  & 1.23E-04 & 2.01  & 2.65E-06 & 3.01  & 7.03E-08 & 4.02  & 2.11E-09 & 5.04  \\
1/128  & 1.90E+00 & 0.00  & 1.69E-03 & 1.00  & 7.69E-06 & 2.00  & 4.14E-08 & 3.00  & 2.74E-10 & 4.00  & 2.05E-12 & 5.00  \\
1/512  & 1.90E+00 & 0.00  & 4.23E-04 & 1.00  & 4.81E-07 & 2.00  & 6.46E-10 & 3.00  & 1.07E-12 & 4.00  & 2.00E-15 & 5.00  \\
\hline
\multicolumn{13}{c}{}\\
		\multicolumn{13}{c}{$g^{(3)}$}  \\ \hline
		$\eps$ & $\fNzero$    & order & $\fNone$    & order & $\fNtwo$    & order		& $\fNthree$    & order		& $\fNfour$    & order		& $\fNfive$    & order \\ \hline
1/2    & 1.61E+00 &       & 2.24E-01 &       & 5.50E-02 &       & 1.20E-02 &       & 2.30E-03 &       & 3.93E-04 &       \\
1/8    & 1.59E+00 & 0.01  & 5.20E-02 & 1.05  & 3.36E-03 & 2.02  & 2.13E-04 & 2.91  & 1.35E-05 & 3.71  & 8.50E-07 & 4.43  \\
1/32   & 1.59E+00 & 0.00  & 1.29E-02 & 1.00  & 2.09E-04 & 2.00  & 3.31E-06 & 3.01  & 5.21E-08 & 4.01  & 8.18E-10 & 5.01  \\
1/128  & 1.59E+00 & 0.00  & 3.23E-03 & 1.00  & 1.30E-05 & 2.00  & 5.17E-08 & 3.00  & 2.03E-10 & 4.00  & 7.98E-13 & 5.00  \\
1/512  & 1.59E+00 & 0.00  & 8.08E-04 & 1.00  & 8.15E-07 & 2.00  & 8.07E-10 & 3.00  & 7.94E-13 & 4.00  & 7.80E-16 & 5.00 
 \\ \hline
	\end{tabular}
	\caption{Convergence rates of the coefficients $\fNl$ for the $P_5$ solution in Example \ref{eg:test1}. According to Lemma \ref{thm:f}, the theoretical order of convergence is $\ell$.}
	\label{tab:test1_P5_f}
\end{table}

\begin{table}[!htbp]
	\centering
	\footnotesize
	\begin{tabular}{|c||c|c|c|c|c|c|c|c|c|c|}
		\multicolumn{11}{c}{$g^{(1)}$}  \\ \hline
		$\eps$ & $\eNzero$    & order & $\eNone$    & order & $\eNtwo$    & order 	& $\eNthree$    & order		& $\eNfour$    & order \\ \hline
		1/2    & 5.93E-03 &       & 3.62E-03 &       & 4.78E-03 &       & 4.71E-03 &       & 5.87E-03 &       \\
		1/8    & 6.71E-08 & 8.22  & 1.35E-08 & 9.02  & 2.62E-08 & 8.74  & 1.18E-07 & 7.64  & 6.17E-07 & 6.61  \\
		1/32   & 9.34E-13 & 8.07  & 4.45E-14 & 9.11  & 3.20E-13 & 8.16  & 6.59E-12 & 7.06  & 1.39E-10 & 6.06  \\
		1/128  & 1.42E-17 & 8.00  & 1.68E-19 & 9.01  & 4.81E-18 & 8.01  & 4.00E-16 & 7.00  & 3.38E-14 & 6.00  \\
		1/512  & 2.16E-22 & 8.00  & 6.42E-25 & 9.00  & 7.33E-23 & 8.00  & 2.44E-20 & 7.00  & 8.25E-18 & 6.00  \\
		\hline
		\multicolumn{11}{c}{}\\
		\multicolumn{11}{c}{$g^{(2)}$}  \\ \hline
		$\eps$ & $\eNzero$    & order & $\eNone$    & order & $\eNtwo$    & order 	& $\eNthree$    & order		& $\eNfour$    & order \\ \hline
		1/2    & 8.37E-04 &       & 4.47E-04 &       & 7.31E-04 &       & 8.59E-04 &       & 1.35E-03 &       \\
		1/8    & 2.28E-08 & 7.58  & 5.15E-09 & 8.20  & 7.69E-09 & 8.27  & 3.89E-08 & 7.21  & 2.85E-07 & 6.11  \\
		1/32   & 3.01E-13 & 8.10  & 1.71E-14 & 9.10  & 8.96E-14 & 8.19  & 2.25E-12 & 7.04  & 6.65E-11 & 6.03  \\
		1/128  & 4.55E-18 & 8.01  & 6.46E-20 & 9.01  & 1.35E-18 & 8.01  & 1.37E-16 & 7.00  & 1.62E-14 & 6.00  \\
		1/512  & 6.94E-23 & 8.00  & 2.46E-25 & 9.00  & 2.05E-23 & 8.00  & 8.34E-21 & 7.00  & 3.95E-18 & 6.00  \\
		\hline
		\multicolumn{11}{c}{}\\
		\multicolumn{11}{c}{$g^{(3)}$}  \\ \hline
		$\eps$ & $\eNzero$    & order & $\eNone$    & order & $\eNtwo$    & order 	& $\eNthree$    & order		& $\eNfour$    & order \\ \hline
		1/2    & 1.62E-08 &       & 9.57E-08 &       & 8.87E-07 &       & 7.49E-06 &       & 5.70E-05 &       \\
		1/8    & 5.52E-11 & 4.10  & 9.91E-12 & 6.62  & 2.13E-10 & 6.01  & 3.36E-09 & 5.56  & 5.32E-08 & 5.03  \\
		1/32   & 9.48E-16 & 7.91  & 3.47E-17 & 9.06  & 3.21E-15 & 8.01  & 2.02E-13 & 7.01  & 1.28E-11 & 6.01  \\
		1/128  & 1.46E-20 & 8.00  & 1.32E-22 & 9.00  & 4.89E-20 & 8.00  & 1.23E-17 & 7.00  & 3.13E-15 & 6.00  \\
		1/512  & 2.22E-25 & 8.00  & 5.02E-28 & 9.00  & 7.46E-25 & 8.00  & 7.53E-22 & 7.00  & 7.65E-19 & 6.00  
		\\ \hline
	\end{tabular}
	\caption{Convergence rates of the error $ \el$ in the coefficients $\fNl$ for the $P_4$ solution in Example \ref{eg:test1}.  According to \eqref{eqn:moment_error} in Theorem \ref{thm:error}, the theoretical order of convergence is $2N=8$ for $\fNzero$ and $2N+2-\ell = 10-\ell$ for $\fNl$, $\ell=1,\dots, 4$. }
	\label{tab:test1_P4_e}
\end{table}

\begin{table}[!htbp]
	\centering
	\footnotesize
	\begin{tabular}{|c||c|c|c|c|c|c|c|c|c|c|c|c|}
		\multicolumn{13}{c}{$g^{(1)}$}  \\ \hline
		$\eps$ & $\eNzero$    & order & $\eNone$    & order & $\eNtwo$    & order 	& $\eNthree$    & order		& $\eNfour$    & order		& $\eNfive$    & order \\ \hline
1/2    & 4.06E-03 &       & 2.96E-03 &       & 3.32E-03 &       & 3.28E-03 &       & 3.62E-03 &       & 4.55E-03 &       \\
1/8    & 4.02E-09 & 9.97  & 1.13E-09 & 10.66 & 1.23E-09 & 10.68 & 4.42E-09 & 9.75  & 2.27E-08 & 8.64  & 1.18E-07 & 7.62  \\
1/32   & 3.13E-15 & 10.15 & 2.10E-16 & 11.18 & 7.85E-16 & 10.29 & 1.52E-14 & 9.08  & 3.12E-13 & 8.07  & 6.56E-12 & 7.07  \\
1/128  & 2.95E-21 & 10.01 & 4.93E-23 & 11.01 & 7.31E-22 & 10.02 & 5.75E-20 & 9.01  & 4.74E-18 & 8.00  & 3.98E-16 & 7.00  \\
1/512  & 2.81E-27 & 10.00 & 1.17E-29 & 11.00 & 6.96E-28 & 10.00 & 2.19E-25 & 9.00  & 7.22E-23 & 8.00  & 2.43E-20 & 7.00  \\
\hline
\multicolumn{13}{c}{}\\
\multicolumn{13}{c}{$g^{(2)}$}  \\ \hline
		$\eps$ & $\eNzero$    & order & $\eNone$    & order & $\eNtwo$    & order 	& $\eNthree$    & order		& $\eNfour$    & order		& $\eNfive$    & order \\ \hline
1/2    & 5.07E-04 &       & 2.72E-04 &       & 4.34E-04 &       & 4.69E-04 &       & 5.59E-04 &       & 8.59E-04 &       \\
1/8    & 1.42E-09 & 9.22  & 3.26E-10 & 9.84  & 4.12E-10 & 10.00 & 1.27E-09 & 9.25  & 6.31E-09 & 8.22  & 3.89E-08 & 7.22  \\
1/32   & 1.15E-15 & 10.12 & 6.76E-17 & 11.10 & 2.54E-16 & 10.32 & 4.28E-15 & 9.09  & 8.73E-14 & 8.07  & 2.23E-12 & 7.04  \\
1/128  & 1.09E-21 & 10.01 & 1.60E-23 & 11.01 & 2.35E-22 & 10.02 & 1.62E-20 & 9.01  & 1.32E-18 & 8.00  & 1.36E-16 & 7.00  \\
1/512  & 1.03E-27 & 10.00 & 3.80E-30 & 11.00 & 2.24E-28 & 10.00 & 6.17E-26 & 9.00  & 2.02E-23 & 8.00  & 8.29E-21 & 7.00  \\
\hline
\multicolumn{13}{c}{}\\
\multicolumn{13}{c}{$g^{(3)}$}  \\ \hline
		$\eps$ & $\eNzero$    & order & $\eNone$    & order & $\eNtwo$    & order 	& $\eNthree$    & order		& $\eNfour$    & order		& $\eNfive$    & order \\ \hline
1/2    & 1.09E-10 &       & 8.15E-10 &       & 9.02E-09 &       & 9.28E-08 &       & 8.74E-07 &       & 7.45E-06 &       \\
1/8    & 2.10E-13 & 4.51  & 3.97E-14 & 7.16  & 8.41E-13 & 6.69  & 1.32E-11 & 6.39  & 2.10E-10 & 6.01  & 3.34E-09 & 5.56  \\
1/32   & 2.33E-19 & 9.89  & 8.53E-21 & 11.07 & 7.89E-19 & 10.01 & 4.98E-17 & 9.01  & 3.16E-15 & 8.01  & 2.01E-13 & 7.01  \\
1/128  & 2.24E-25 & 9.99  & 2.02E-27 & 11.00 & 7.51E-25 & 10.00 & 1.90E-22 & 9.00  & 4.82E-20 & 8.00  & 1.23E-17 & 7.00  \\
1/512  & 2.14E-31 & 10.00 & 4.82E-34 & 11.00 & 7.16E-31 & 10.00 & 7.23E-28 & 9.00  & 7.35E-25 & 8.00  & 7.49E-22 & 7.00 
  \\ \hline
	\end{tabular}
	\caption{Convergence rates of the error $ \el$ in the coefficients $\fNl$ for the $P_5$ solution in Example \ref{eg:test1}.  According to \eqref{eqn:moment_error} in Theorem \ref{thm:error}, the theoretical order of convergence is $2N=10$ for $\fNzero$ and $2N+2-\ell = 12-\ell$ for $\fNl$, $\ell=1,\dots, 5$. }

	\label{tab:test1_P5_e}
\end{table}

\end{eg}
For all three initial conditions, we observe the convergence rates for $f^N$ indicated by \eqref{eqn:L2_error}  in Table \ref{tab:test1_total_errors}, the convergence rates for $\el$ indicated by \eqref{eqn:moment_error} in Tables \ref{tab:test1_P4_e} and \ref{tab:test1_P5_e}, and the covergence rates for $\fNl$ indicated by Remark \ref{thm:remark_fNl} in Tables \ref{tab:test1_P4_f} and \ref{tab:test1_P5_f}.  We observe these rates even for $g^{(1)}$, which is not in $H^1(dx)$ and thus does not satisfy the hypothesis used to prove these estimates.  This discrepancy may be due to the fact that the Fourier-Galerkin method uses a finite number of waves and thus the numerical approximation of $g^{(1)}$ is in $H^1(dx)$.  However, even with $10,000$ Fourier modes, the results do not change.  Thus while $g^{(1)} \in H^1(dx)$ may be a necessary condition, it may be impossible to verify it numerically in this example.

%% file: sections/numerics2.tex
\section{The benefit of increasing $N$} 
\label{sec:numerics2}

The goal of any apriori error estimate is to provide an indication of how the accuracy of an approximation will improve as a given parameter varies.  For example, the spectral estimate in \eqref{eqn:spectral_est}
suggests that $e^N = f - f^N$ behaves like
\beq
\label{eqn:spectral_increase_N}
\frac{\|e^{N+1}\|}{\|e^N\|} \sim \left(\frac{N}{N+1}\right)^q,
\eeq
where $q$ is related to the regularity of $f$.  Thus the gain realized by increasing $N$ to $N+1$ is not expected to be large, especially when $q$ is small.
On the other hand, the estimate in \eqref{eqn:L2_error} of Theorem \ref{thm:error} suggests that by increasing $N$, we gain an additional factor of $\eps$ in the error estimate; that is,
\beq
\label{eqn:spectral_error_increase_N_eps}
\frac{\|e^{N+1}\|}{\|e^N\|} \sim \eps.
\eeq
Similarly, \eqref{eqn:moment_error} of Theorem \ref{thm:error} suggests that we gain an additional factor of $\eps^2$ in the error estimate for $\eNlN := \eNl = \fl - \fNl$; that is 
\beq
\label{eqn:moment_error_increase_N_eps}
\frac{\| \eNpolN \|}{\| \eNlN \|} \sim \eps^2.
\eeq

However, the statements in \eqref{eqn:spectral_error_increase_N_eps} and \eqref{eqn:moment_error_increase_N_eps} are misleading since, unlike the spectral estimate in \eqref{eqn:spectral_est},  the $\eps$-dependent estimates in \eqref{eqn:L2_error} and \eqref{eqn:moment_error} include coefficients that depend on $N$ and $\ell$.   We begin by exploring this question numerically.

\subsection{Numerical experiments}

\begin{eg}\normalfont \label{2eg:test3}
	We return to Example \ref{eg:test1} from Section \ref{sec:numerics} with  initial condition $g^{(2)}$.   
	We again use the $P_{65}$ solution as a reference. We compute $P_N$ solutions with $\eps = 2 \cdot 4^{-m}$, $ m=1,\dots, 5$, and values of $N$ up to $40$, which in practice is quite large.  We examine the solutions at times $t = 0.1, 1, 10$. 
	
	As before the  spatial discretization is a Fourier-Galerkin method that uses fast Fourier transforms (FFT) for implementation.  For most cases, the spatial grid has $100$ points. However, for smaller $t$ or larger $\eps$, gradients in $x$ become larger; in such cases, more points are needed to ensure that the spacial discretization error can be neglected.  Specifically,  for $t = 0.1$ and $\eps = 1/8$, $1000$ points are used; for  $t = 0.1$ and $\eps = 1/2$, $2500$ points are used;  
	for $t = 1$ and $\eps = 1/2$, $1000$ points are used.  
	
	In Figure \ref{fig:2eg_test3_e}, the ratio in \eqref{eqn:spectral_error_increase_N_eps} (normalized by $\eps$) is plotted as a function of $N$.  In Figures \ref{fig:2eg_test3_err_f0}--\ref{fig:2eg_test3_err_f2}, the ratio in \eqref{eqn:moment_error_increase_N_eps} (normalized by $\eps^{2}$) is plotted for $\ell=0,1,2$.  We observe the following trends:
	\begin{itemize}
		\item[1)] Larger values of $t$ lead to smaller error ratios. Numerically, we find that
		for  $1 \leq N \leq 40$,
		\begin{equation}\label{eqn:observed_C}
		\frac{\|e^{N+1}\|}{\|e^{N}\|} \leq G_1(N,t) \eps,
		\quad
		\text{where}
		\quad
		G_1(N,t) \leq \begin{cases}
		13,  & \quad  t =0.1,\\
		4.5, & \quad  t =1,\\
		1.1, & \quad  t =10,
		\end{cases}
		\end{equation} 
		and for $1 \leq N \leq 40$ and $0 \leq \ell \leq 2$,
		\begin{equation} \label{eqn:observed_C_2}
		\frac{\|\eNpolN\| }{\|\eNlN\|} \leq G_2(N,t) \eps^2, 
		\quad
		\text{where}
		\quad
		G_2(N,t) \leq \begin{cases}
		400,  & \quad  t =0.1,\\
		50,  & \quad  t =1, \\
		20, & \quad  t =10.
		\end{cases}
		\end{equation} 

		\item[2)] For fixed $t$, the solution profiles of the normalized error ratios appear to convergence at $\eps$ decreases. 
		
		\item[3)] As $N$ varies, the solution profiles of the normalized error ratios exhibit plateaus with sharp transitions in between. We do not yet understand the origin of this behavior. 
		
	\end{itemize}

	\begin{figure}
		\centering
		\begin{subfigure}[$t=0.1$, $g^{(2)}$]
			{\includegraphics[width=.31\linewidth]{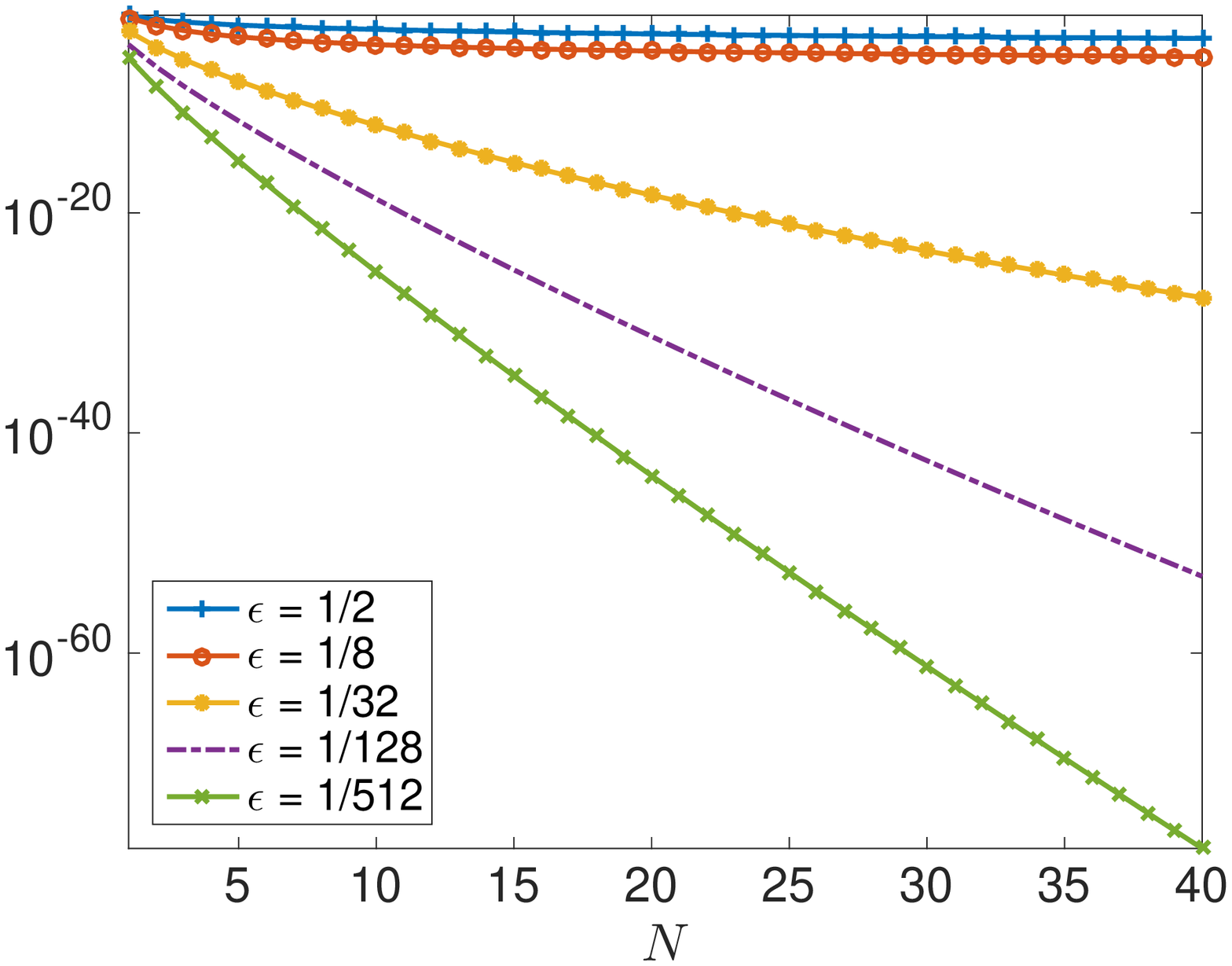}\label{fig:s3_001_fig5}}
		\end{subfigure}%
	\hspace{10pt}
		\begin{subfigure}[$t=1$, $g^{(2)}$]
			{\includegraphics[width=.31\linewidth]{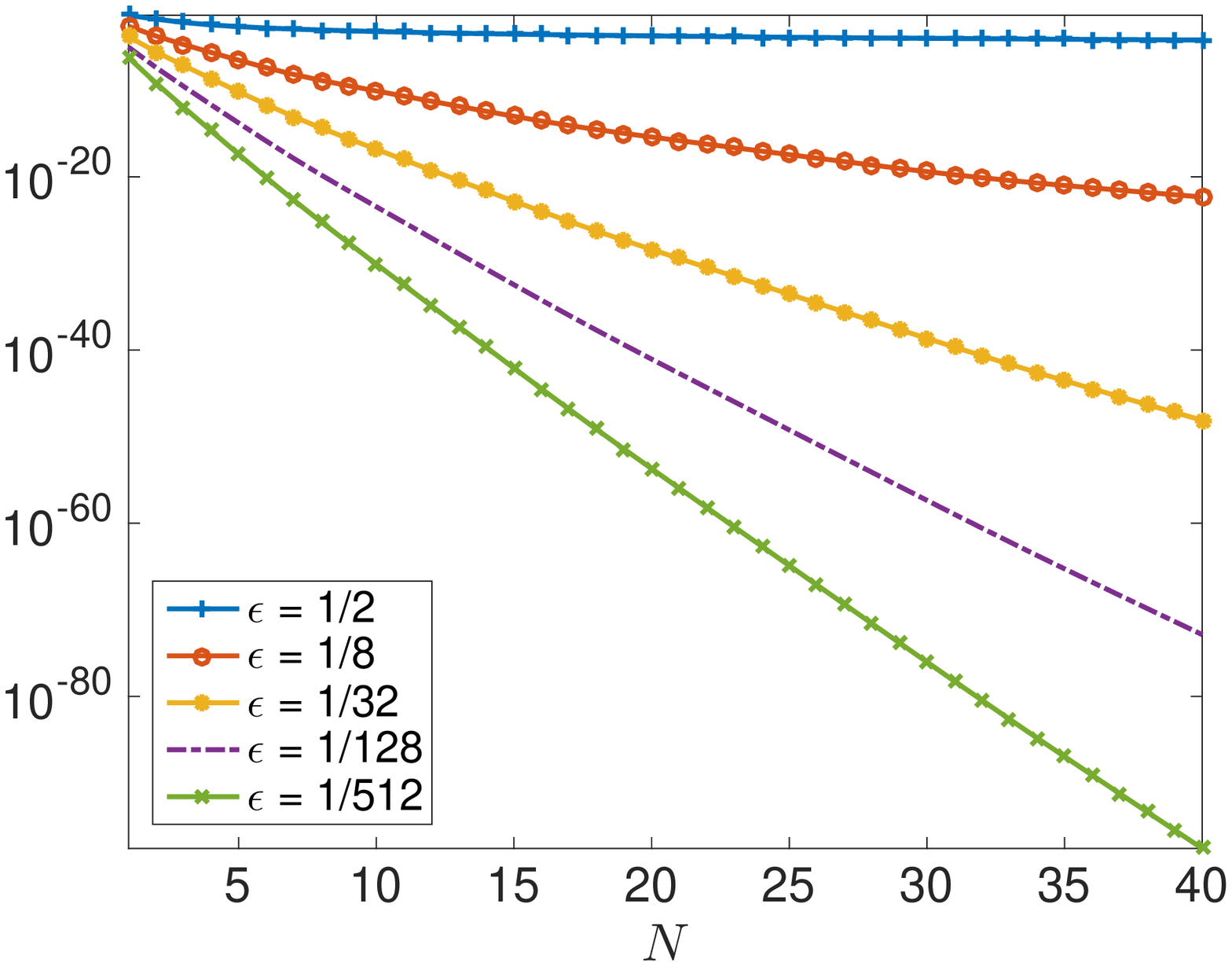}\label{fig:s3_1_fig5}}
		\end{subfigure}%
	\hspace{10pt}
		\begin{subfigure}[$t=10$, $g^{(2)}$]
			{\includegraphics[width=.31\linewidth]{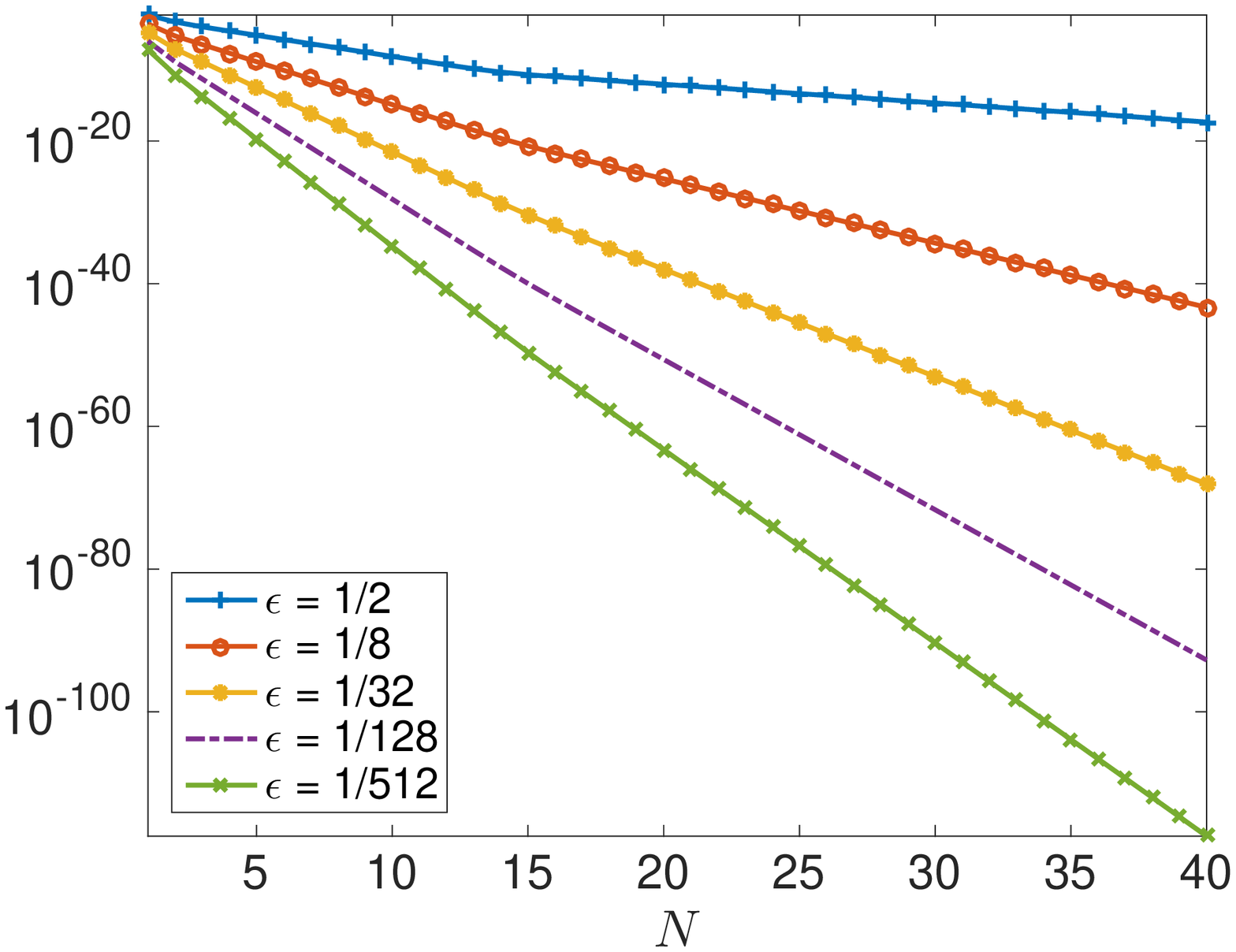}\label{fig:s3_10_fig5}}
		\end{subfigure}%
	\\
	\hspace{4pt}
		\begin{subfigure}[$t=0.1$, $g^{(2)}$]
			{\includegraphics[width=.3\linewidth]{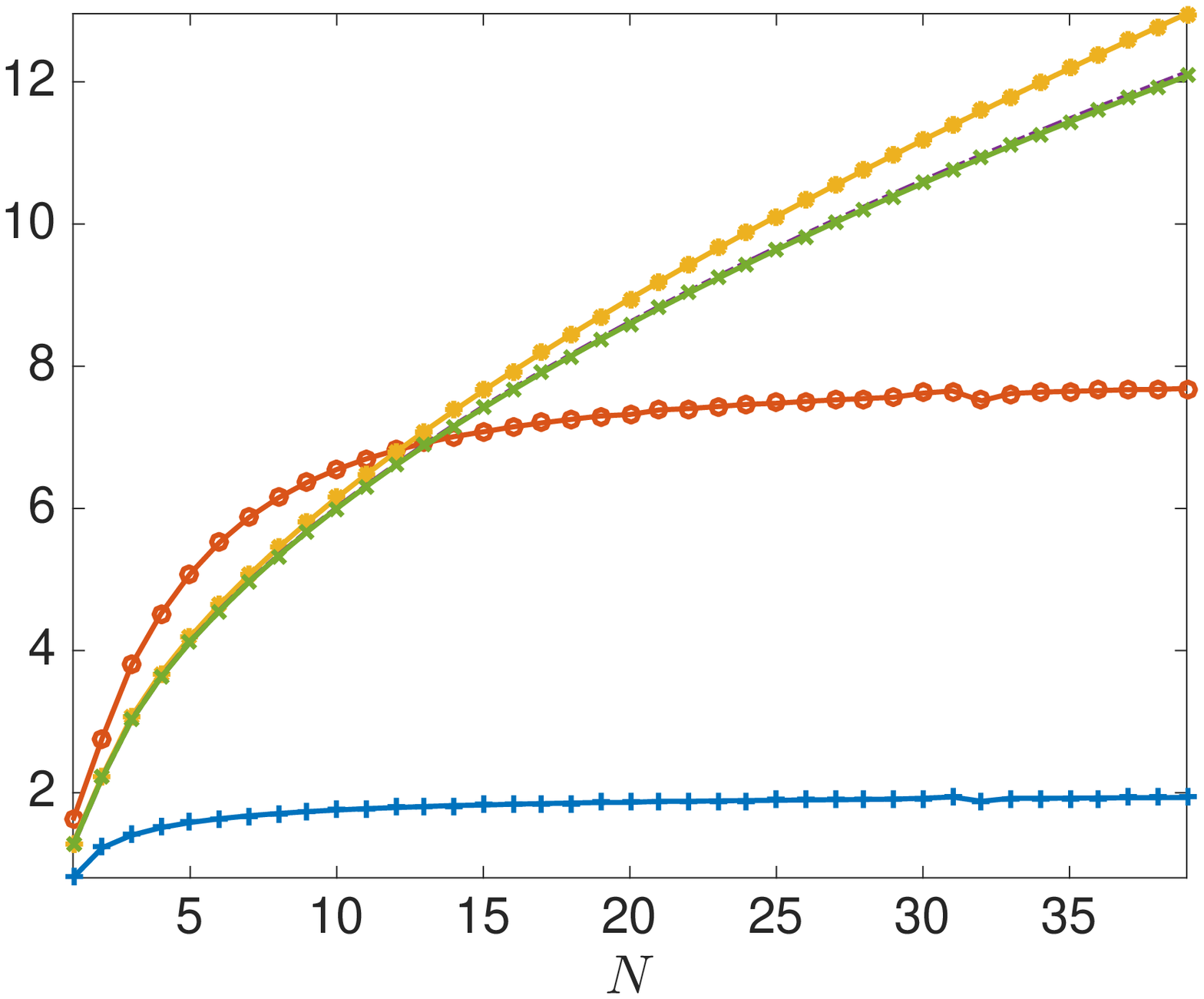}\label{fig:s3_001_fig8}}
		\end{subfigure}%
	\hspace{13pt}
		\begin{subfigure}[$t=1$, $g^{(2)}$]
			{\includegraphics[width=.3\linewidth]{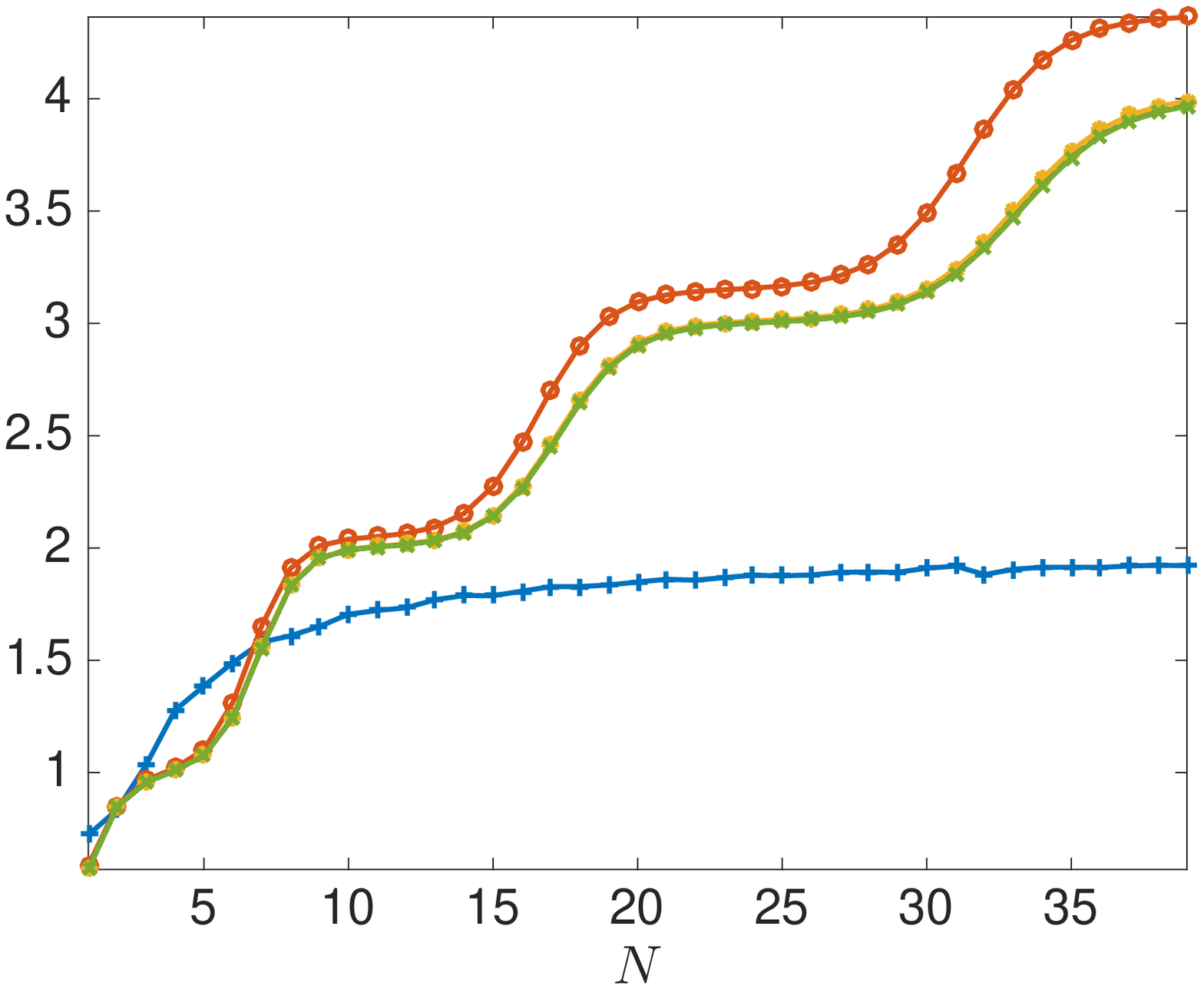}\label{fig:s3_1_fig8}}
		\end{subfigure}%
	\hspace{13pt}
		\begin{subfigure}[$t=10$, $g^{(2)}$]
			{\includegraphics[width=.3\linewidth]{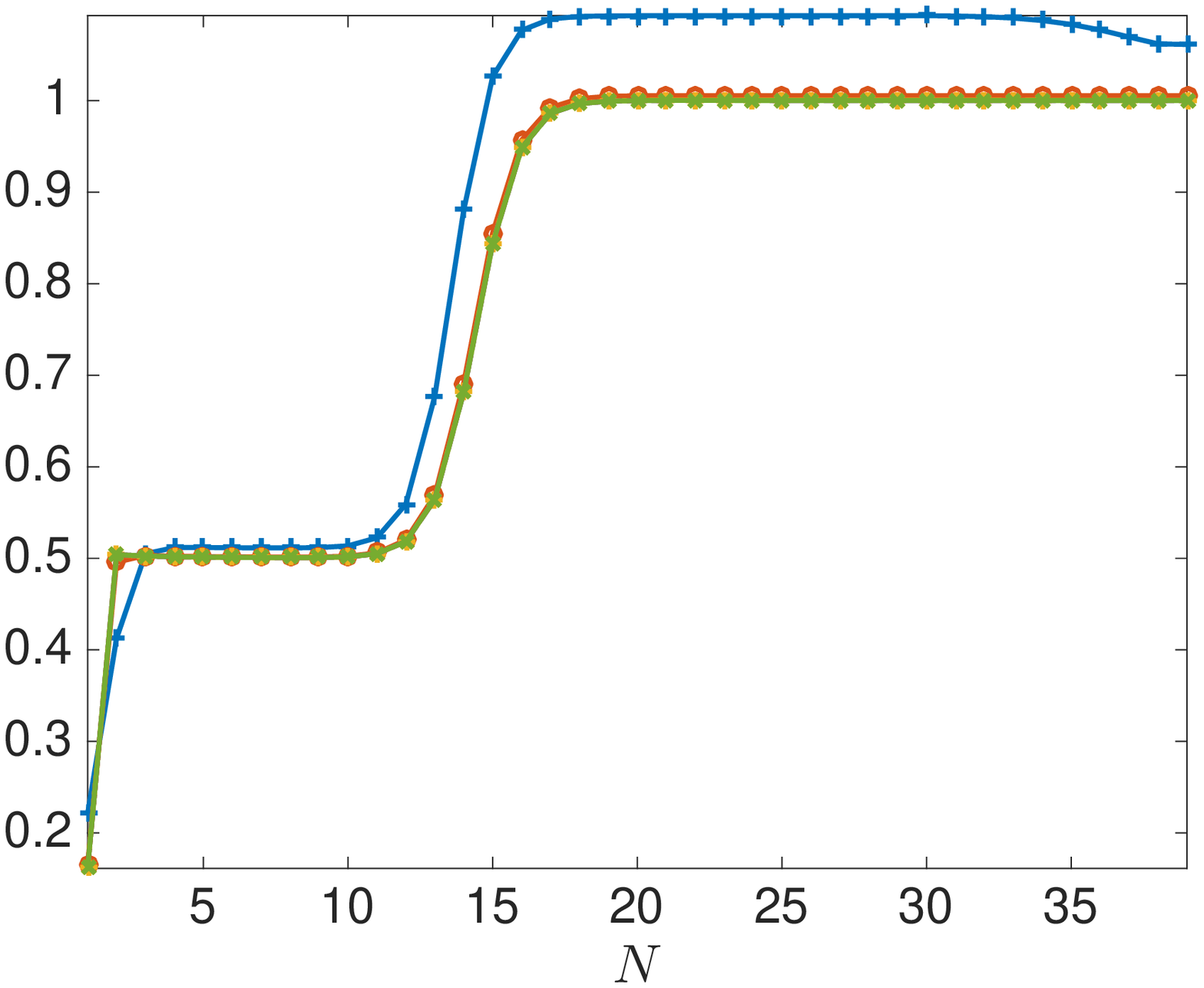}\label{fig:s3_10_fig8}}
		\end{subfigure}%
		\caption{Results from Example \ref{2eg:test3}. Top figures: $\|e^N\|$.  Bottom figures: $\frac{\|e^{N+1}\|}{\|e^N\|\eps}$.}
		\label{fig:2eg_test3_e}
	\end{figure}

	\begin{figure}
		\centering
		\begin{subfigure}[$t=0.1$, $g^{(2)}$]
			{\includegraphics[width=.31\linewidth]{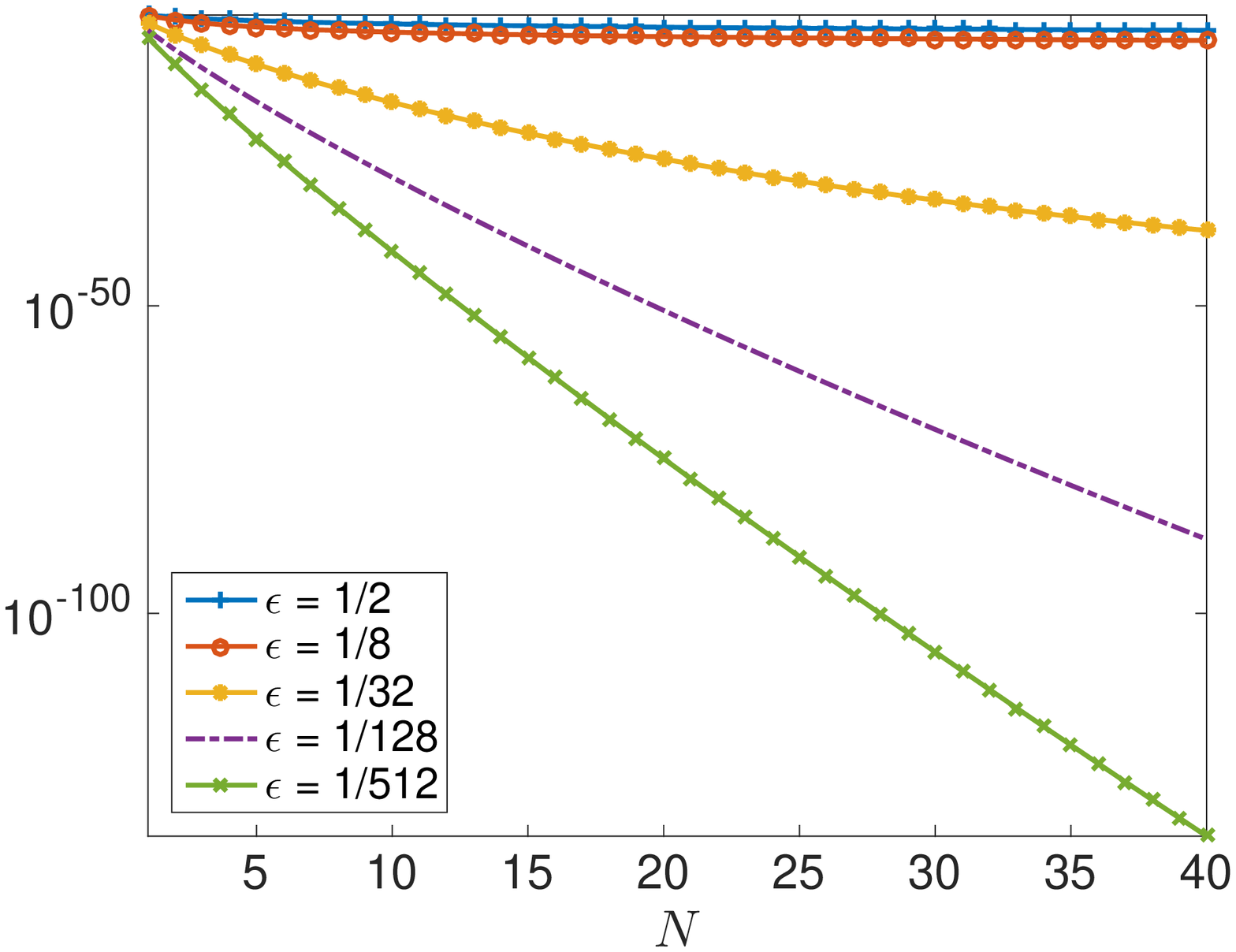}\label{fig:s3_001_fig05}}
		\end{subfigure}%
		\hspace{10pt}
		\begin{subfigure}[$t=1$, $g^{(2)}$]
			{\includegraphics[width=.31\linewidth]{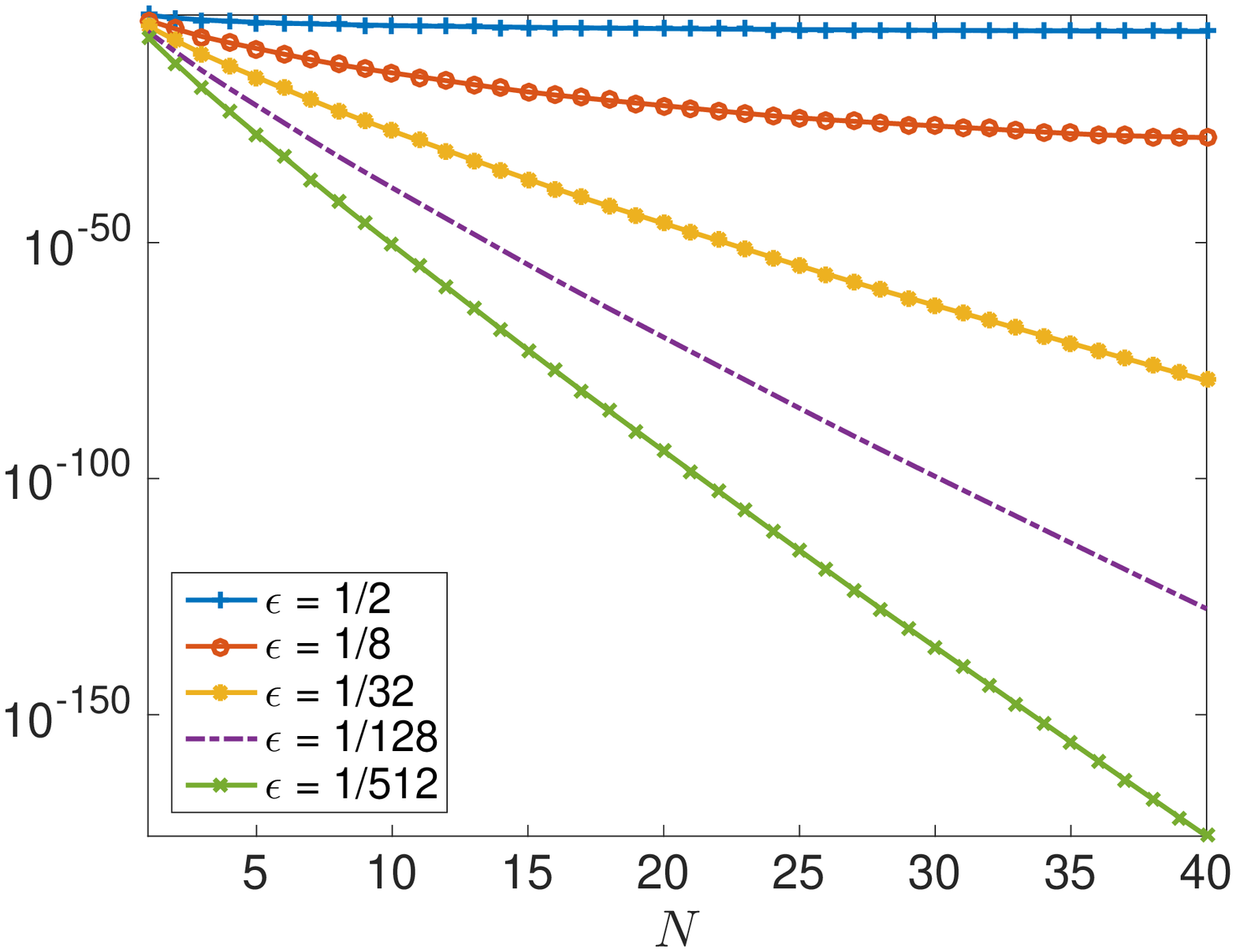}\label{fig:s3_1_fig05}}
		\end{subfigure}%
		\hspace{10pt}
		\begin{subfigure}[$t=10$, $g^{(2)}$]
			{\includegraphics[width=.31\linewidth]{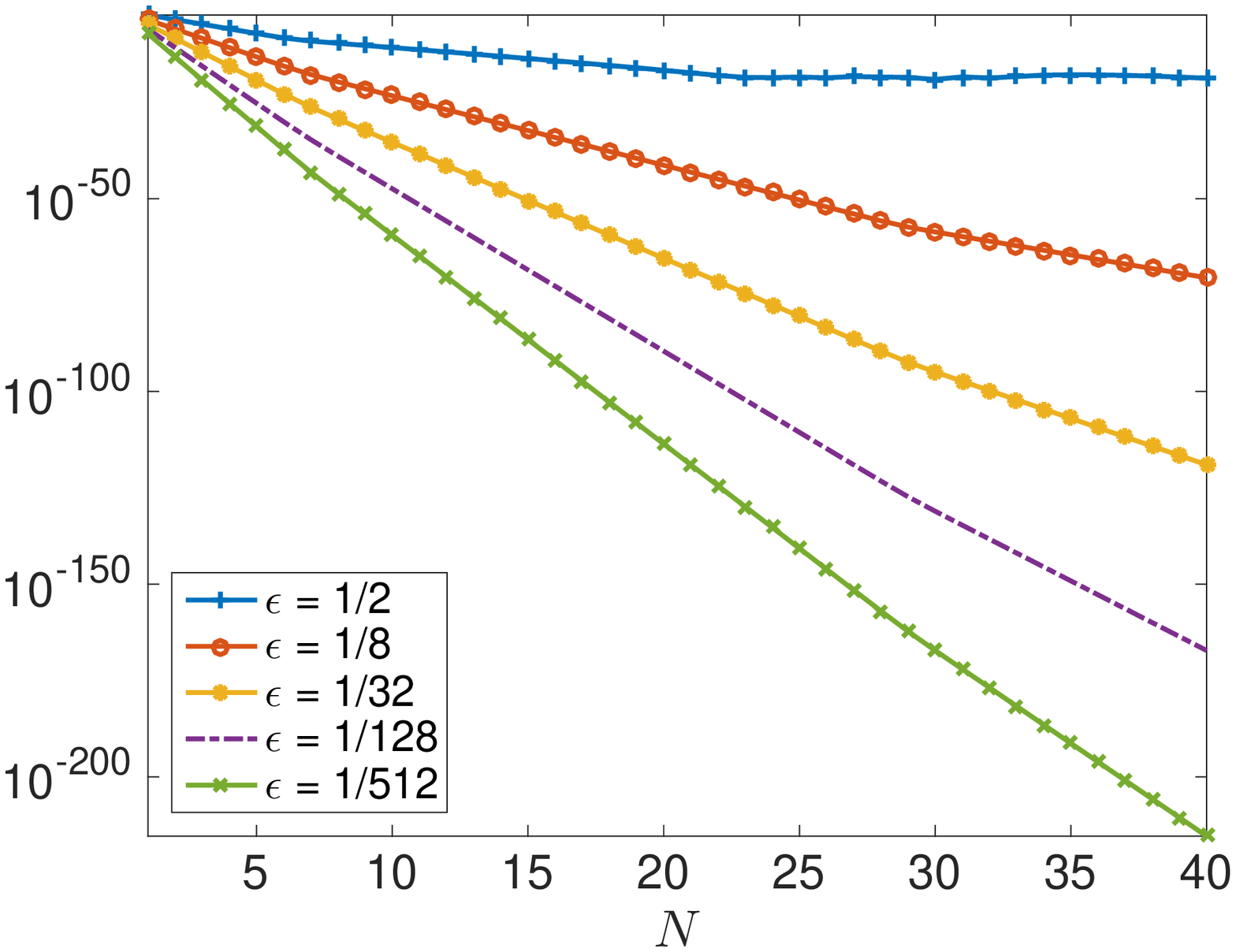}\label{fig:s3_10_fig05}}
		\end{subfigure}%
		\\
		\hspace{4pt}
		\begin{subfigure}[$t=0.1$, $g^{(2)}$]
			{\includegraphics[width=.3\linewidth]{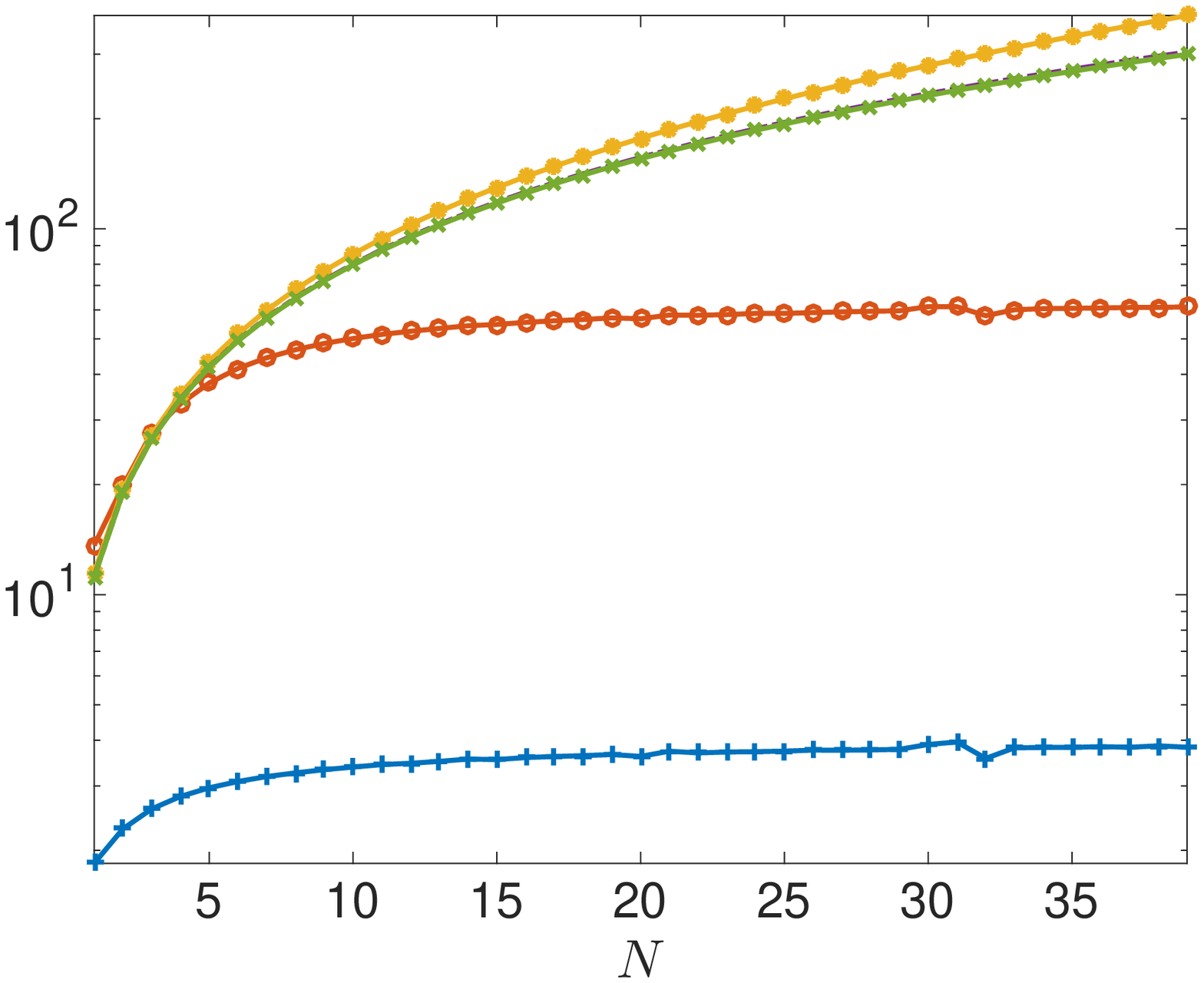}\label{fig:s3_001_fig08}}
		\end{subfigure}%
		\hspace{13pt}
		\begin{subfigure}[$t=1$, $g^{(2)}$]
			{\includegraphics[width=.3\linewidth]{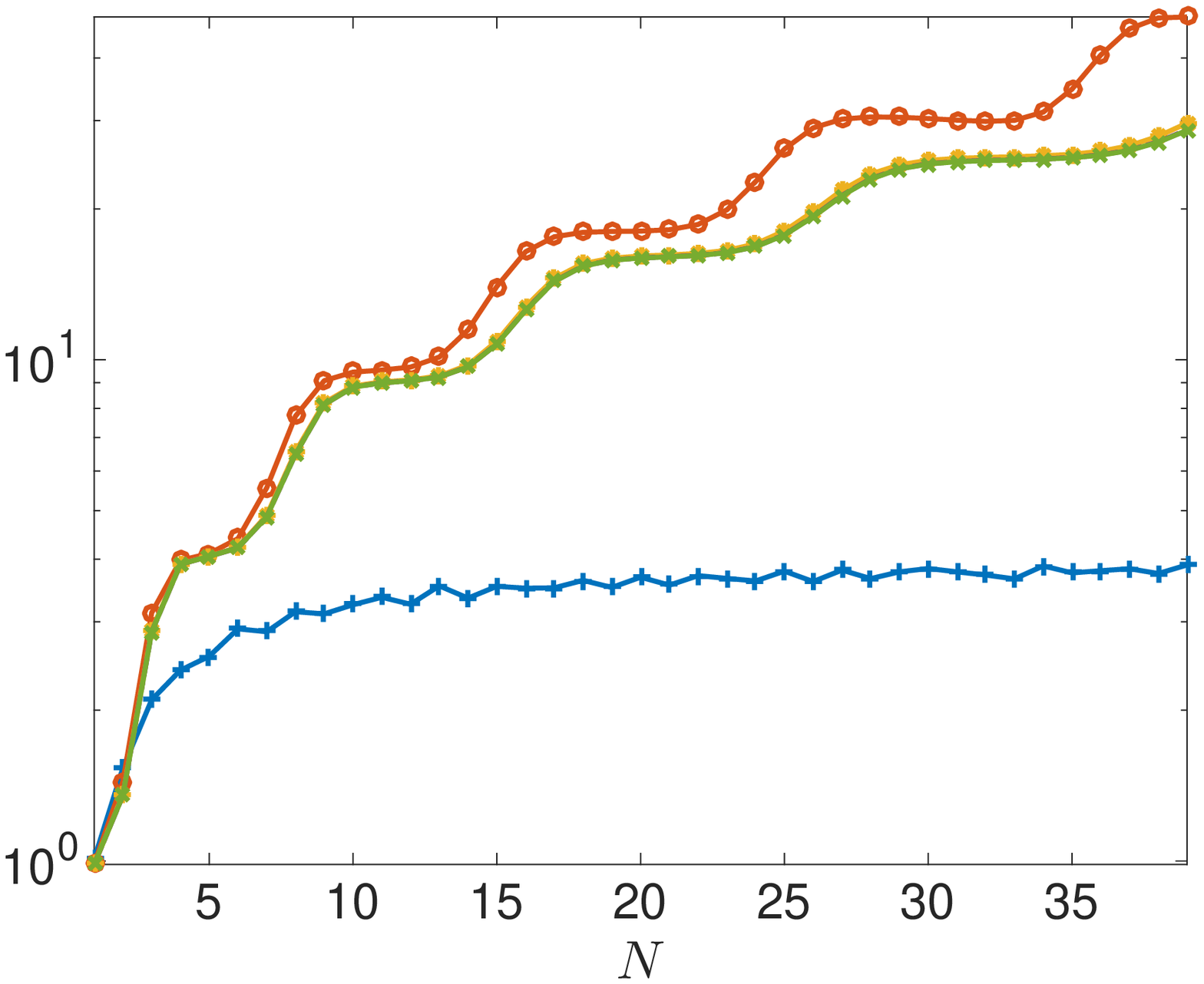}\label{fig:s3_1_fig08}}
		\end{subfigure}%
		\hspace{13pt}
		\begin{subfigure}[$t=10$, $g^{(2)}$]
			{\includegraphics[width=.3\linewidth]{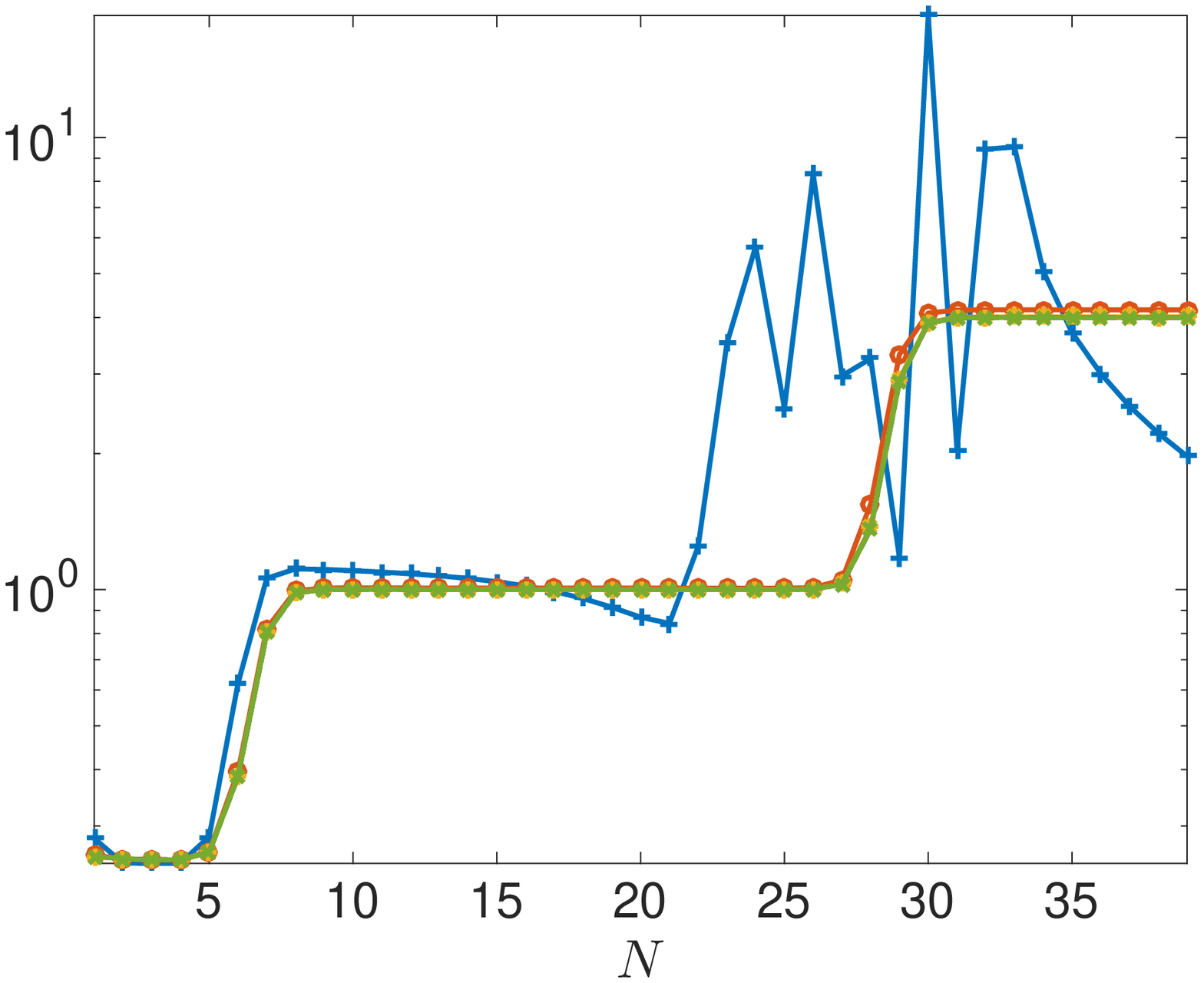}\label{fig:s3_10_fig08}}
		\end{subfigure}%
		\caption{Results from Example \ref{2eg:test3}.  Top figures: $\|\eNzeroN\|$.  Bottom figures: $\frac{\|\eNzeroNpo\|}{\|\eNzeroN\|\eps^2}$.}
		\label{fig:2eg_test3_err_f0}
	\end{figure}
	
	\begin{figure}
		\centering
		\begin{subfigure}[$t=0.1$, $g^{(2)}$]
			{\includegraphics[width=.31\linewidth]{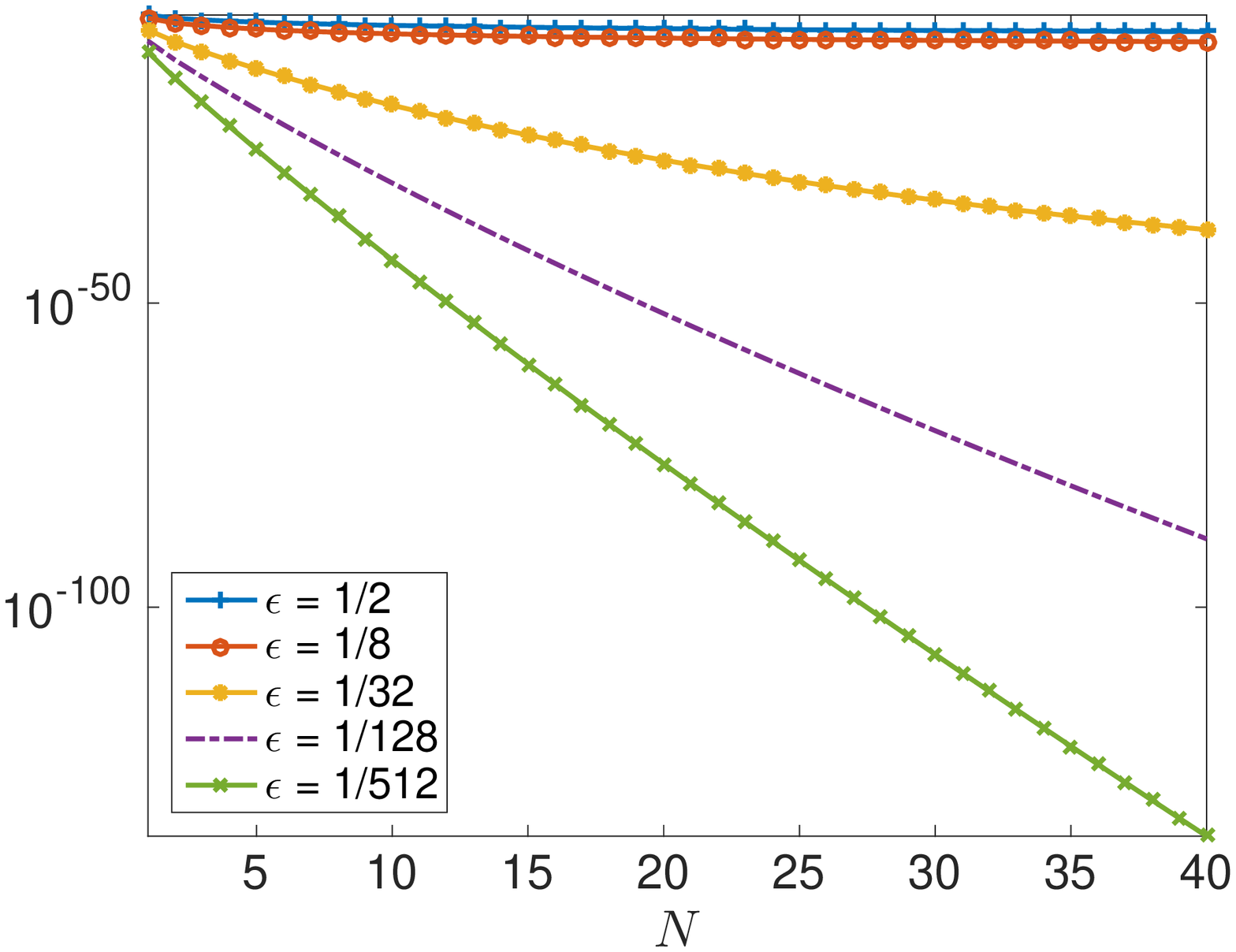}\label{fig:s3_001_fig15}}
		\end{subfigure}%
		\hspace{10pt}
		\begin{subfigure}[$t=1$, $g^{(2)}$]
			{\includegraphics[width=.31\linewidth]{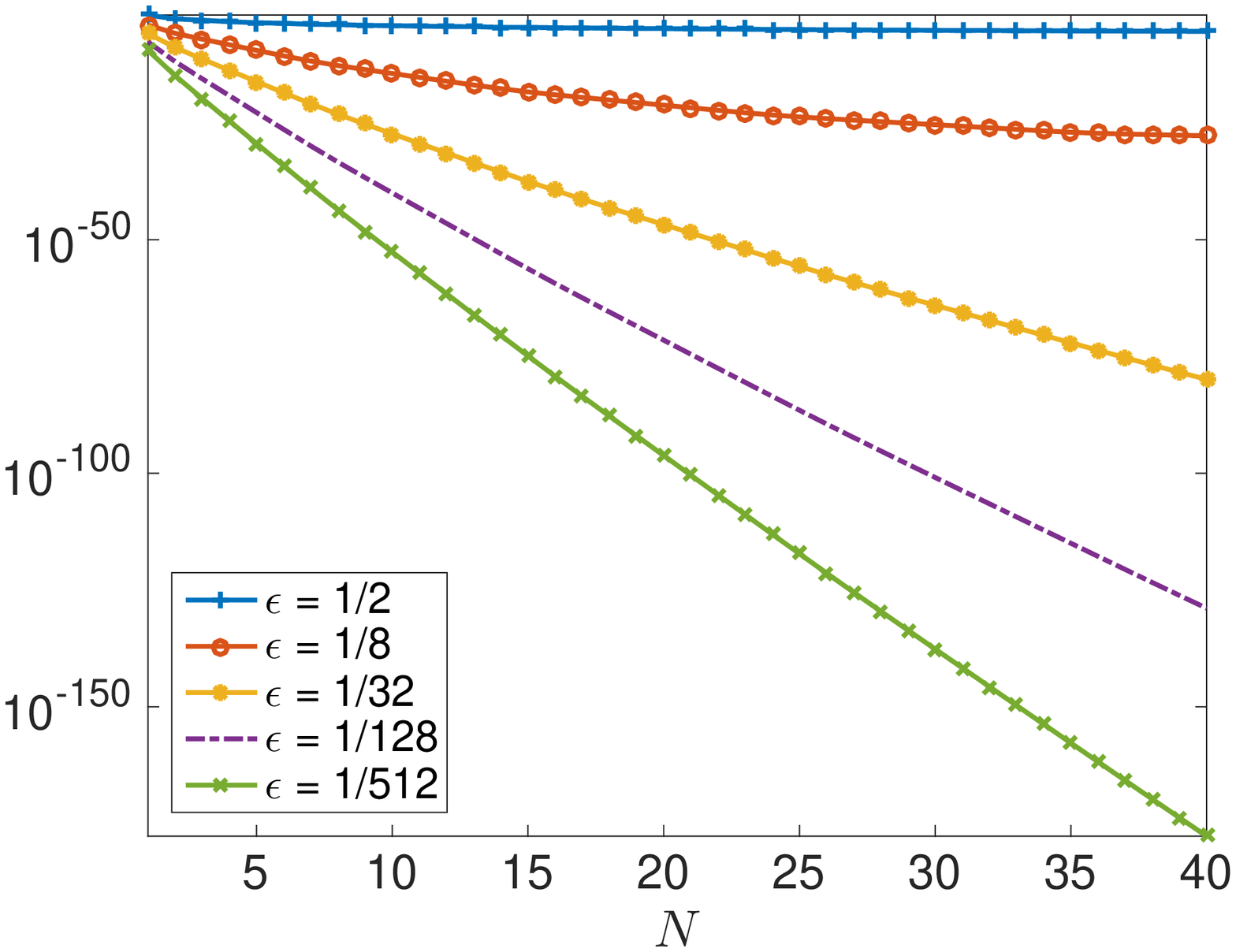}\label{fig:s3_1_fig15}}
		\end{subfigure}%
		\hspace{10pt}
		\begin{subfigure}[$t=10$, $g^{(2)}$]
			{\includegraphics[width=.31\linewidth]{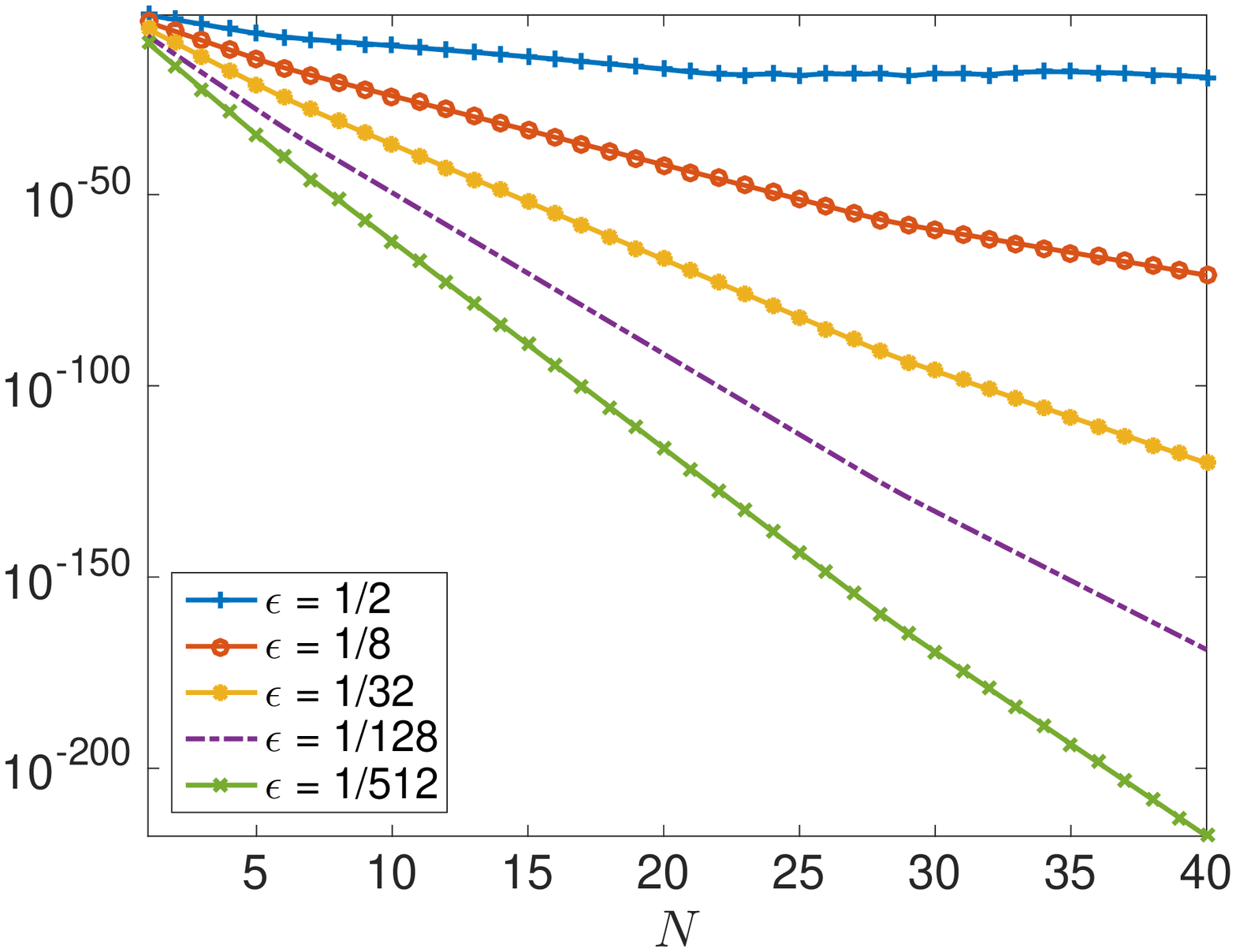}\label{fig:s3_10_fig15}}
		\end{subfigure}%
		\\
		\hspace{4pt}
		\begin{subfigure}[$t=0.1$, $g^{(2)}$]
			{\includegraphics[width=.3\linewidth]{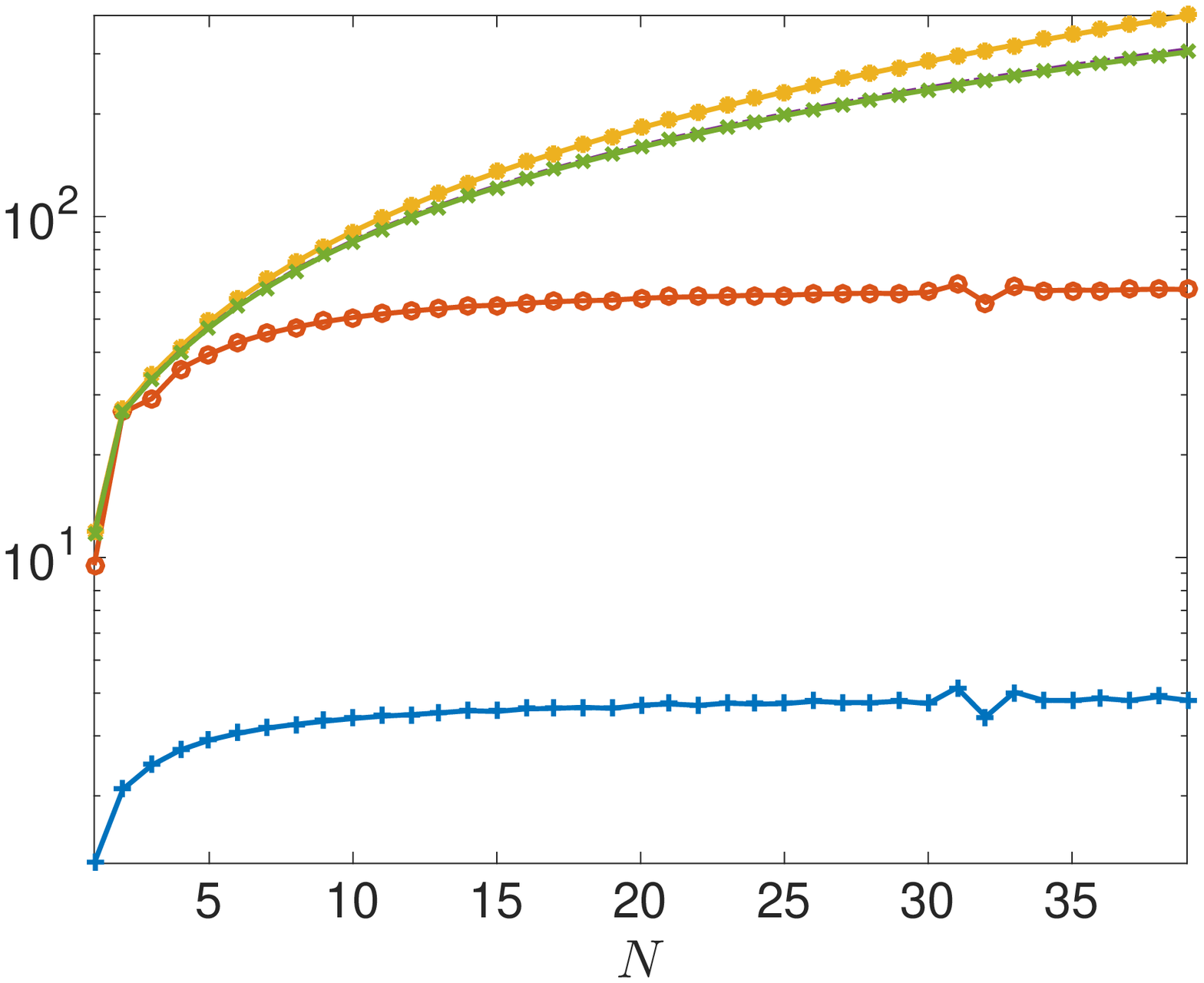}\label{fig:s3_001_fig18}}
		\end{subfigure}%
		\hspace{13pt}
		\begin{subfigure}[$t=1$, $g^{(2)}$]
			{\includegraphics[width=.3\linewidth]{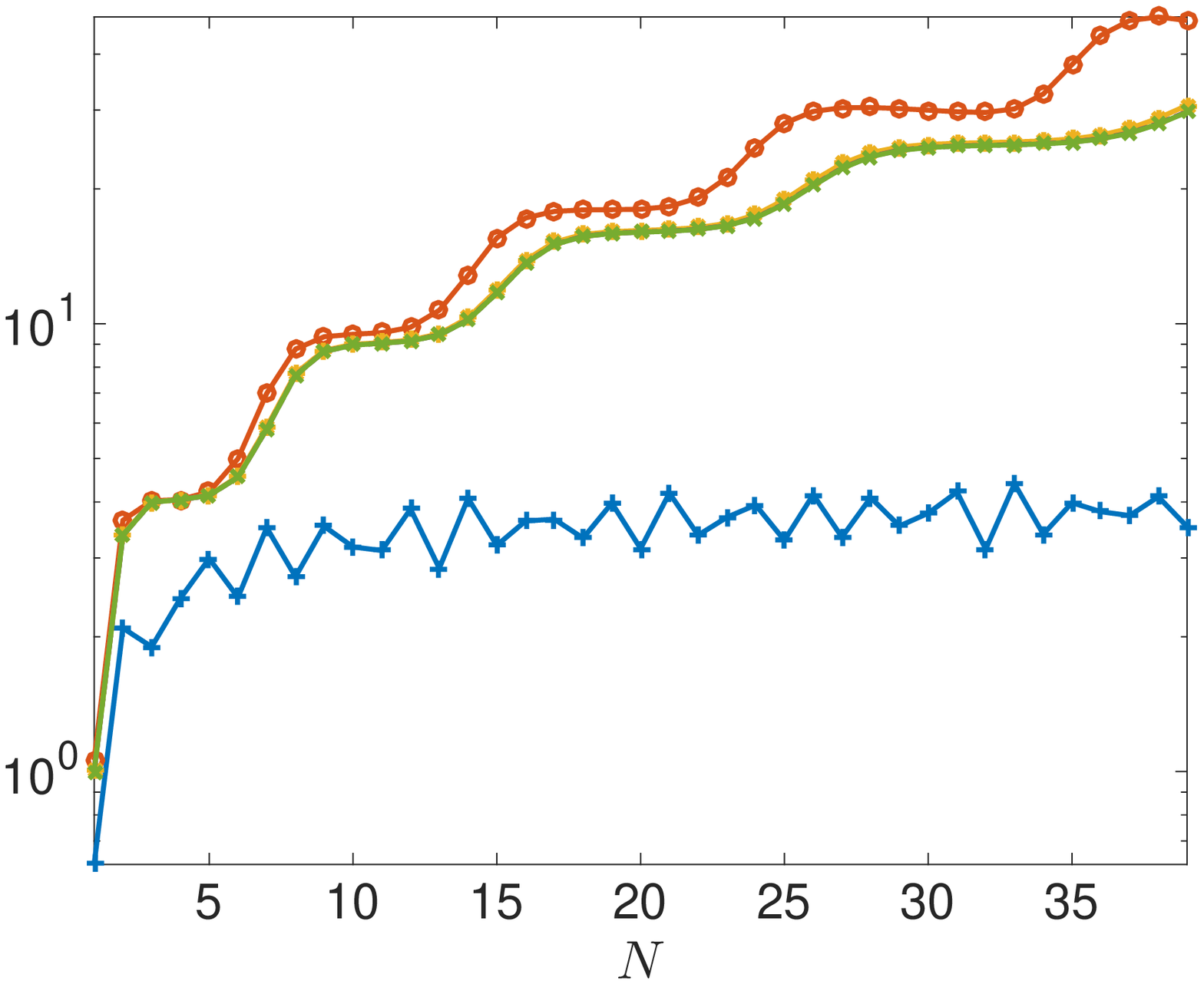}\label{fig:s3_1_fig18}}
		\end{subfigure}%
		\hspace{13pt}
		\begin{subfigure}[$t=10$, $g^{(2)}$]
			{\includegraphics[width=.3\linewidth]{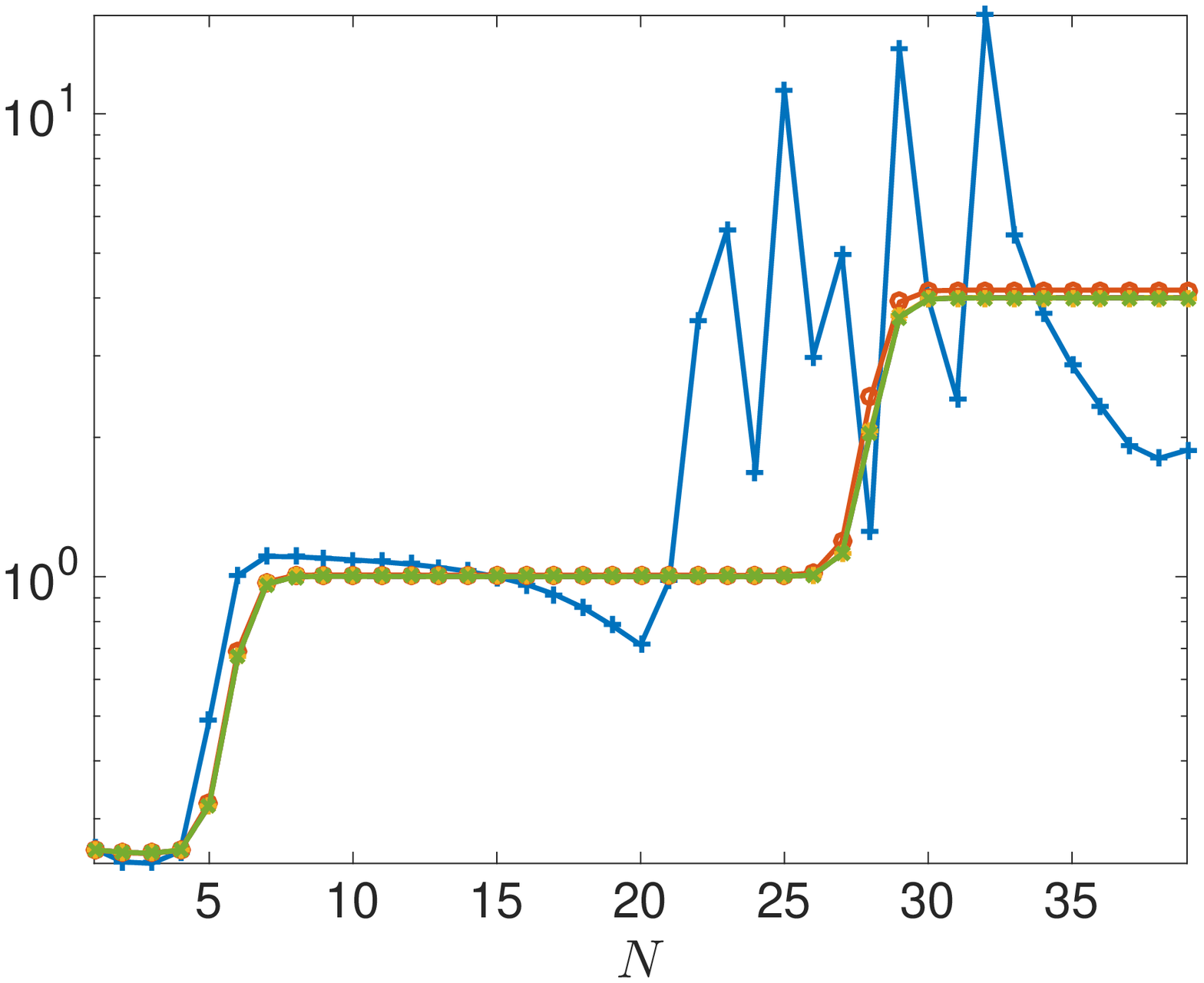}\label{fig:s3_10_fig18}}
		\end{subfigure}%
		\caption{Results from Example \ref{2eg:test3}.  Top figures: $\|\eNoneN\|$.  Bottom figures: $\frac{\|\eNoneNpo\|}{\|\eNoneN\|\eps^2}$.}
		\label{fig:2eg_test3_err_f1}
	\end{figure}
	
	\begin{figure}
		\centering
		\begin{subfigure}[$t=0.1$, $g^{(2)}$]
			{\includegraphics[width=.31\linewidth]{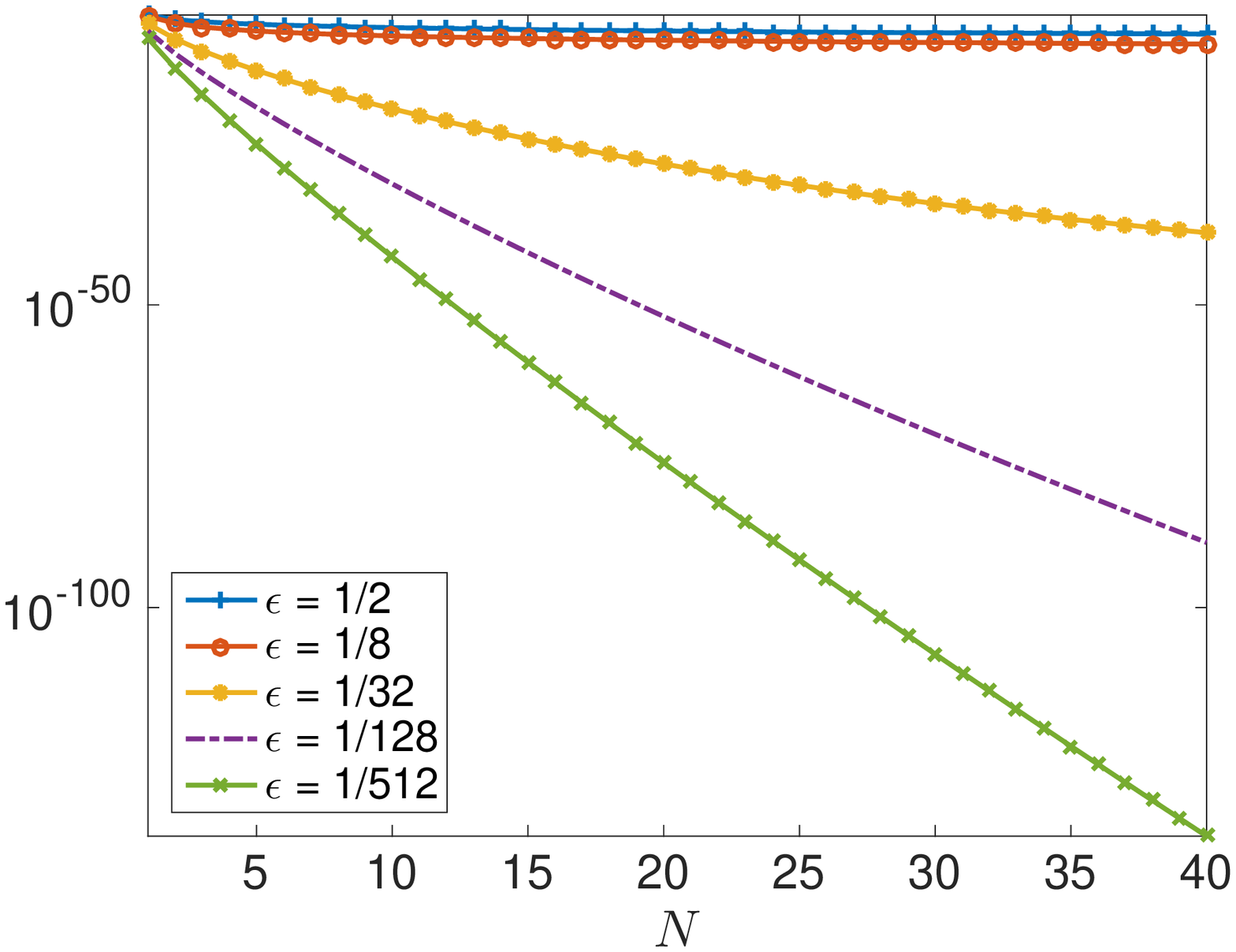}\label{fig:s3_001_fig25}}
		\end{subfigure}%
		\hspace{10pt}
		\begin{subfigure}[$t=1$, $g^{(2)}$]
			{\includegraphics[width=.31\linewidth]{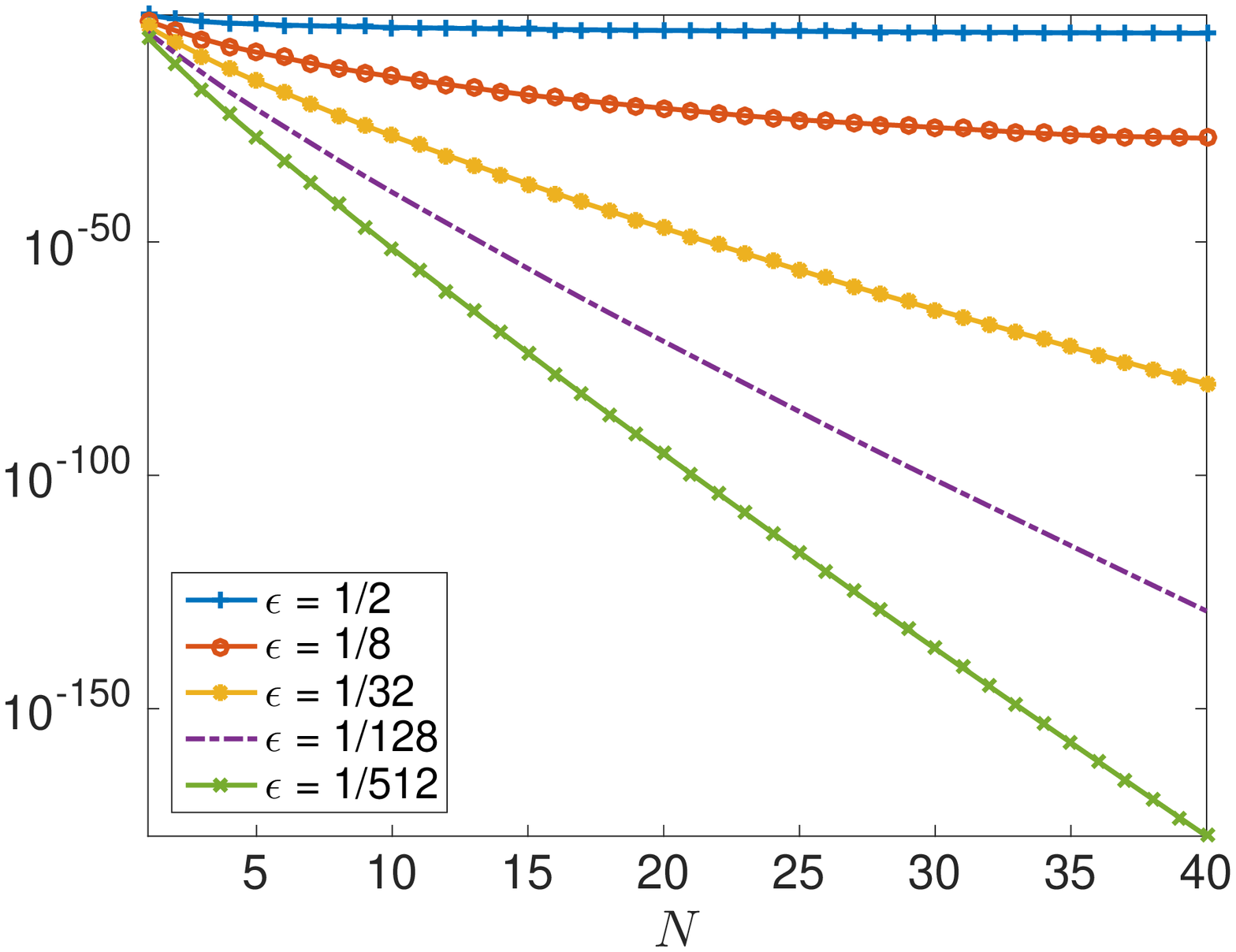}\label{fig:s3_1_fig25}}
		\end{subfigure}%
		\hspace{10pt}
		\begin{subfigure}[$t=10$, $g^{(2)}$]
			{\includegraphics[width=.31\linewidth]{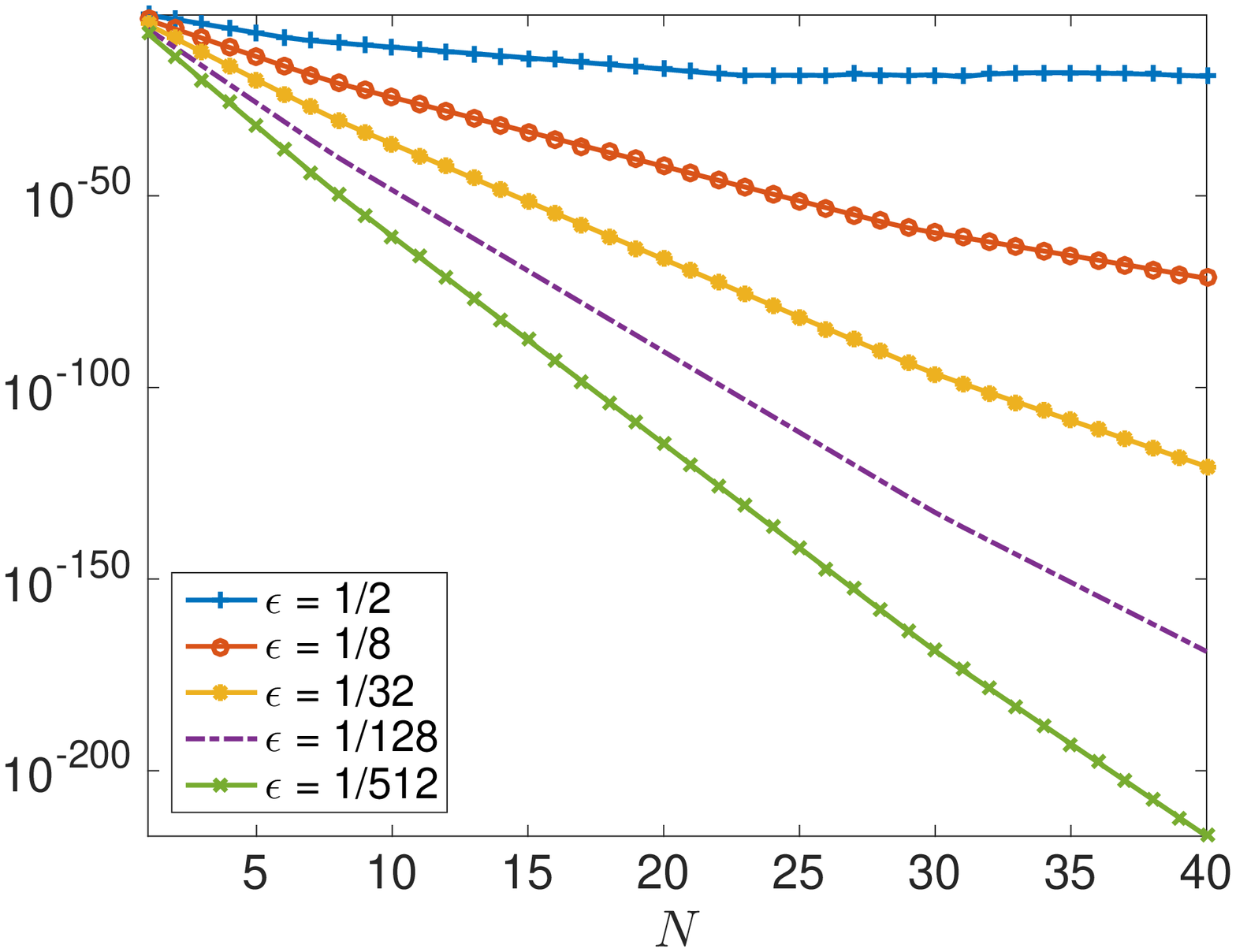}\label{fig:s3_10_fig25}}
		\end{subfigure}%
		\\
		\hspace{4pt}
		\begin{subfigure}[$t=0.1$, $g^{(2)}$]
			{\includegraphics[width=.3\linewidth]{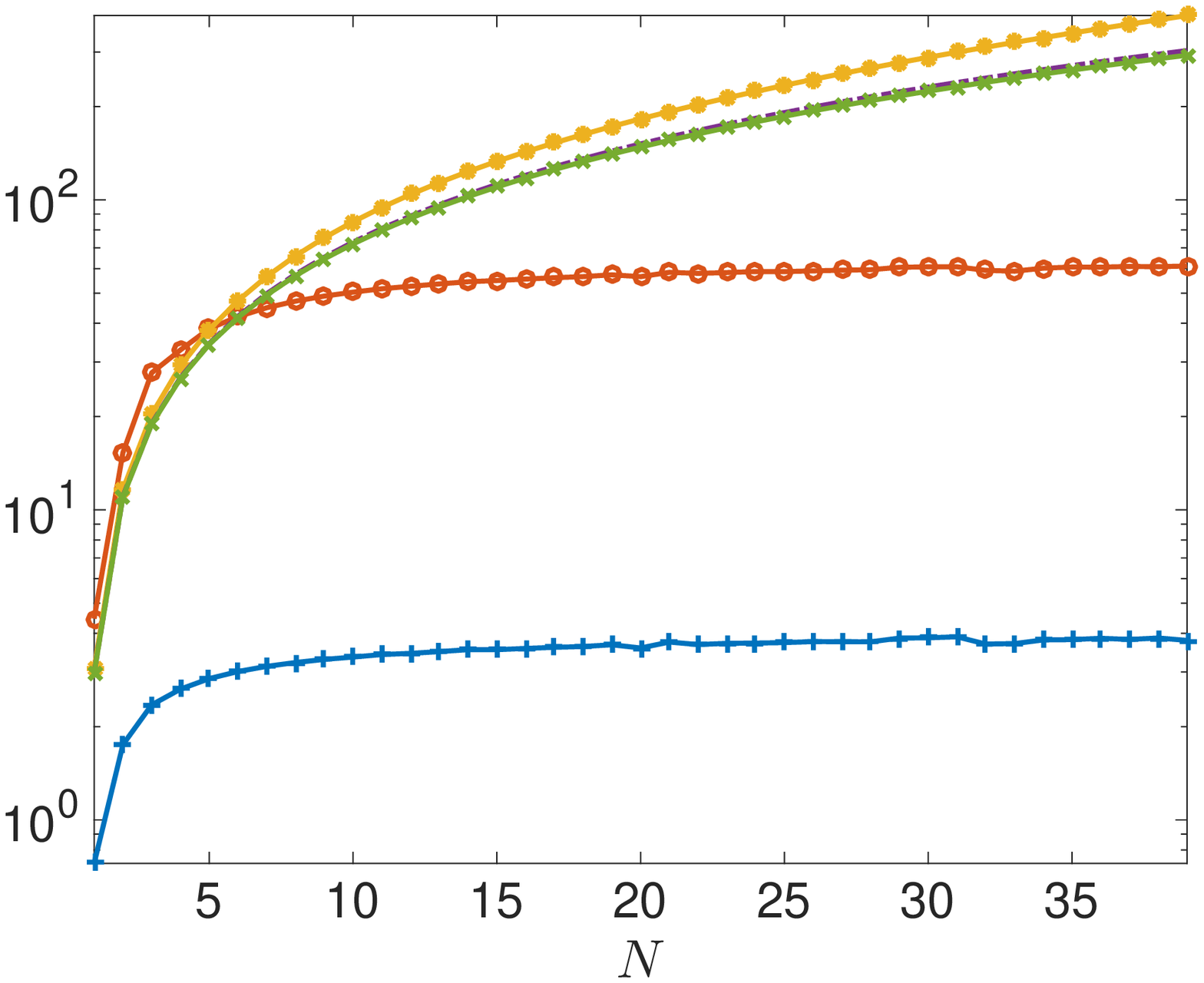}\label{fig:s3_001_fig28}}
		\end{subfigure}%
		\hspace{13pt}
		\begin{subfigure}[$t=1$, $g^{(2)}$]
			{\includegraphics[width=.3\linewidth]{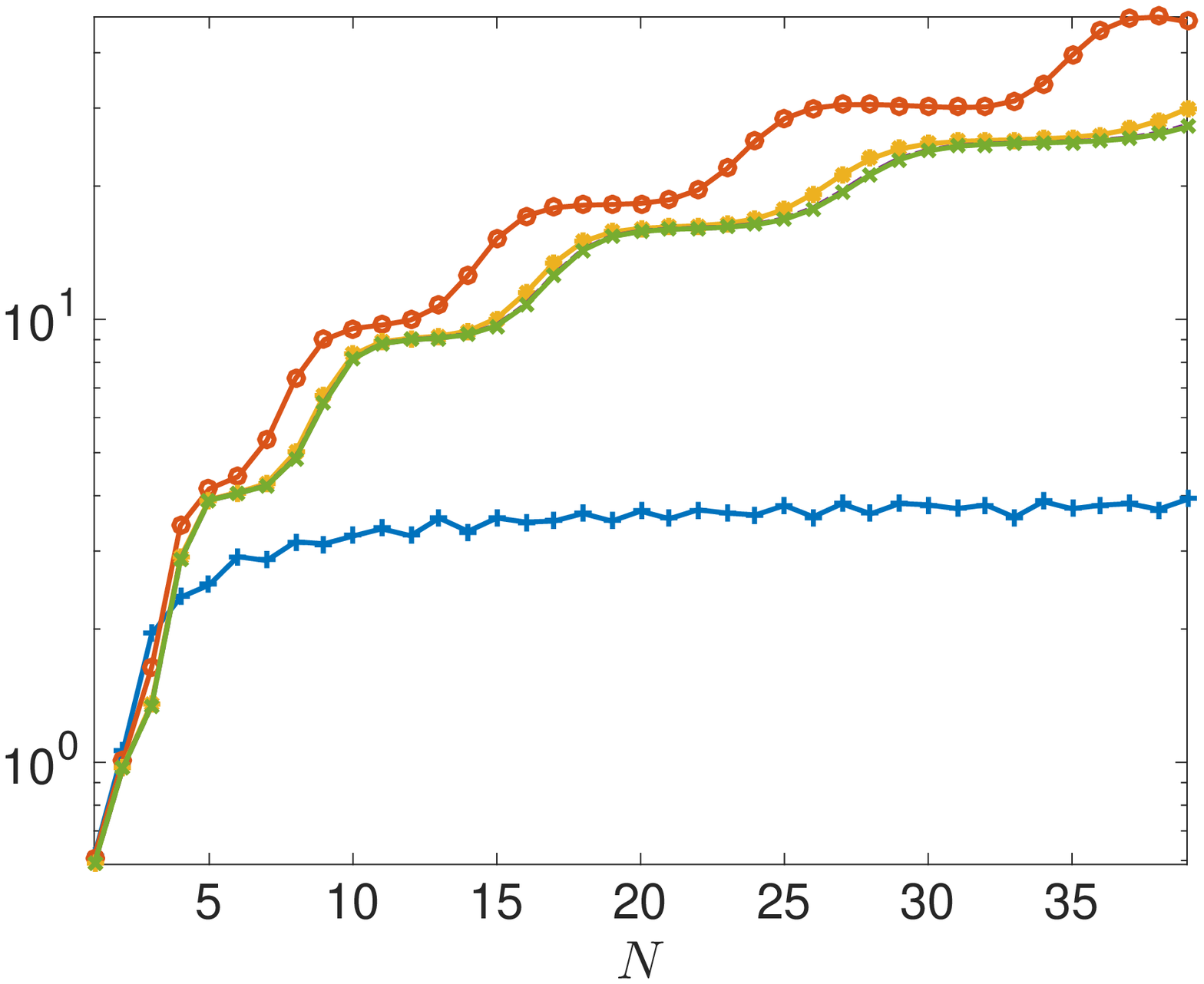}\label{fig:s3_1_fig28}}
		\end{subfigure}%
		\hspace{13pt}
		\begin{subfigure}[$t=10$, $g^{(2)}$]
			{\includegraphics[width=.3\linewidth]{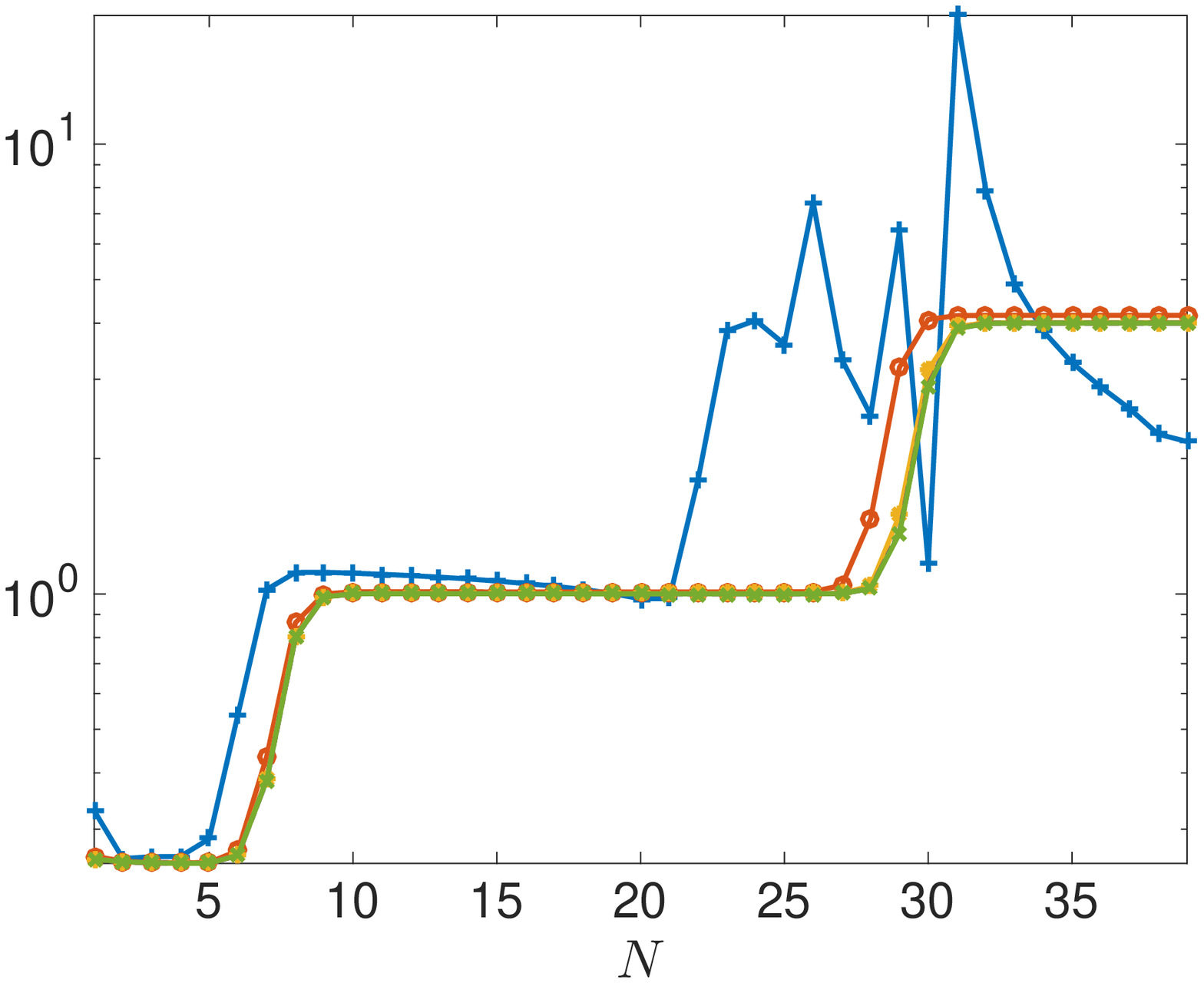}\label{fig:s3_10_fig28}}
		\end{subfigure}%
		\caption{Results from Example \ref{2eg:test3}.  Top figures: $\|\eNtwoN\|$.  Bottom figures: $\frac{\|\eNtwoNpo\|}{\|\eNtwoN\|\eps^2}$.}
		\label{fig:2eg_test3_err_f2}
	\end{figure}

\end{eg}

\begin{eg}\normalfont \label{2eg:test4}
	We repeat the previous test, this time using the initial condition $g^{(3)}$ from Example \ref{eg:test1}.  
	Because $g^{(3)}$ is smooth, $20$ grid points are sufficient to ensure that the spatial error in the Fourier-Galerkin discretization is negligible.
	For large values of $N$, the errors are so small that 300 digits are used.
	
	\begin{figure}
		\centering
		\begin{subfigure}[$t=0.1$, $g^{(3)}$]
			{\includegraphics[width=.31\linewidth]{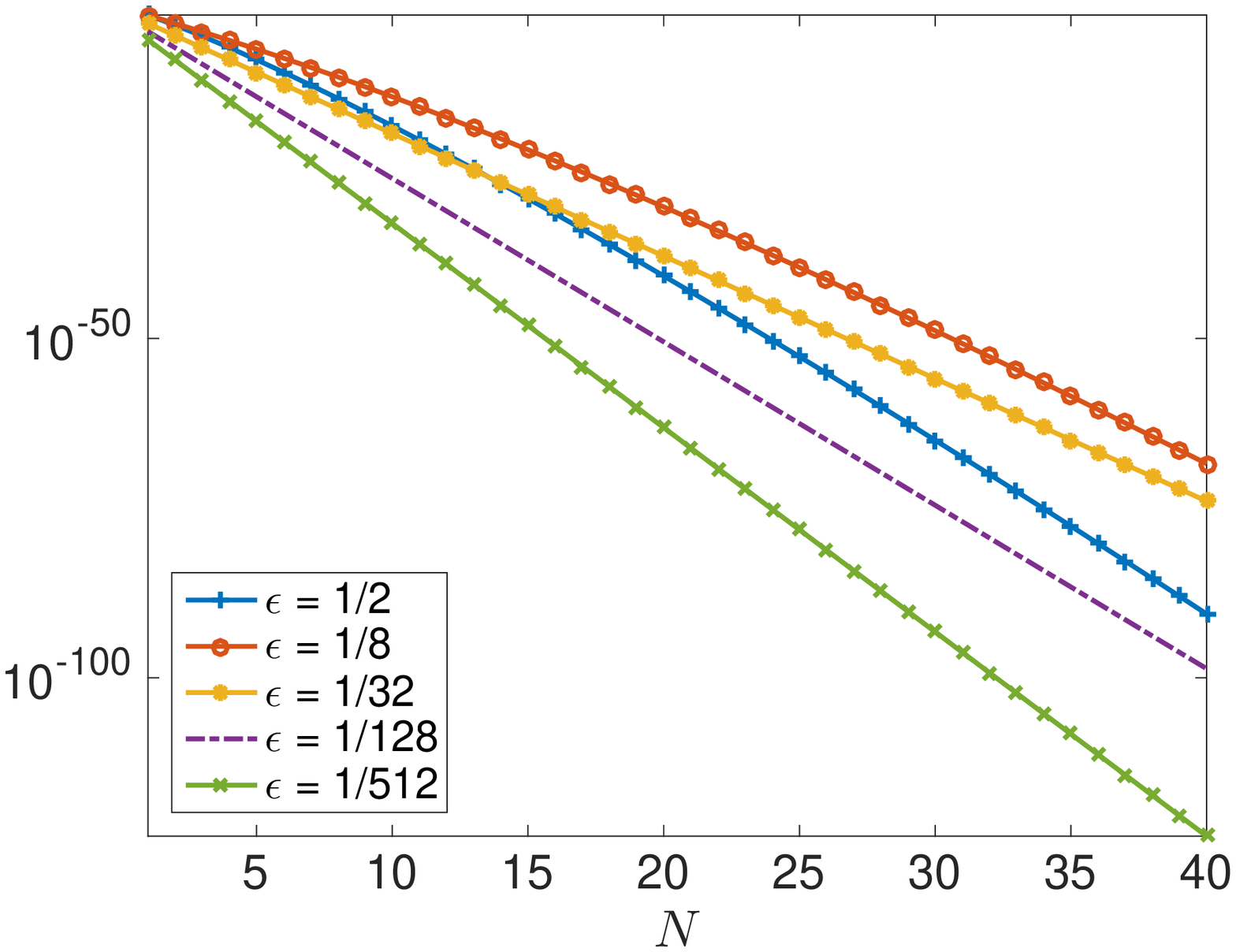}\label{fig:s3_2_001_fig5}}
		\end{subfigure}%
		\hspace{10pt}
		\begin{subfigure}[$t=1$, $g^{(3)}$]
			{\includegraphics[width=.31\linewidth]{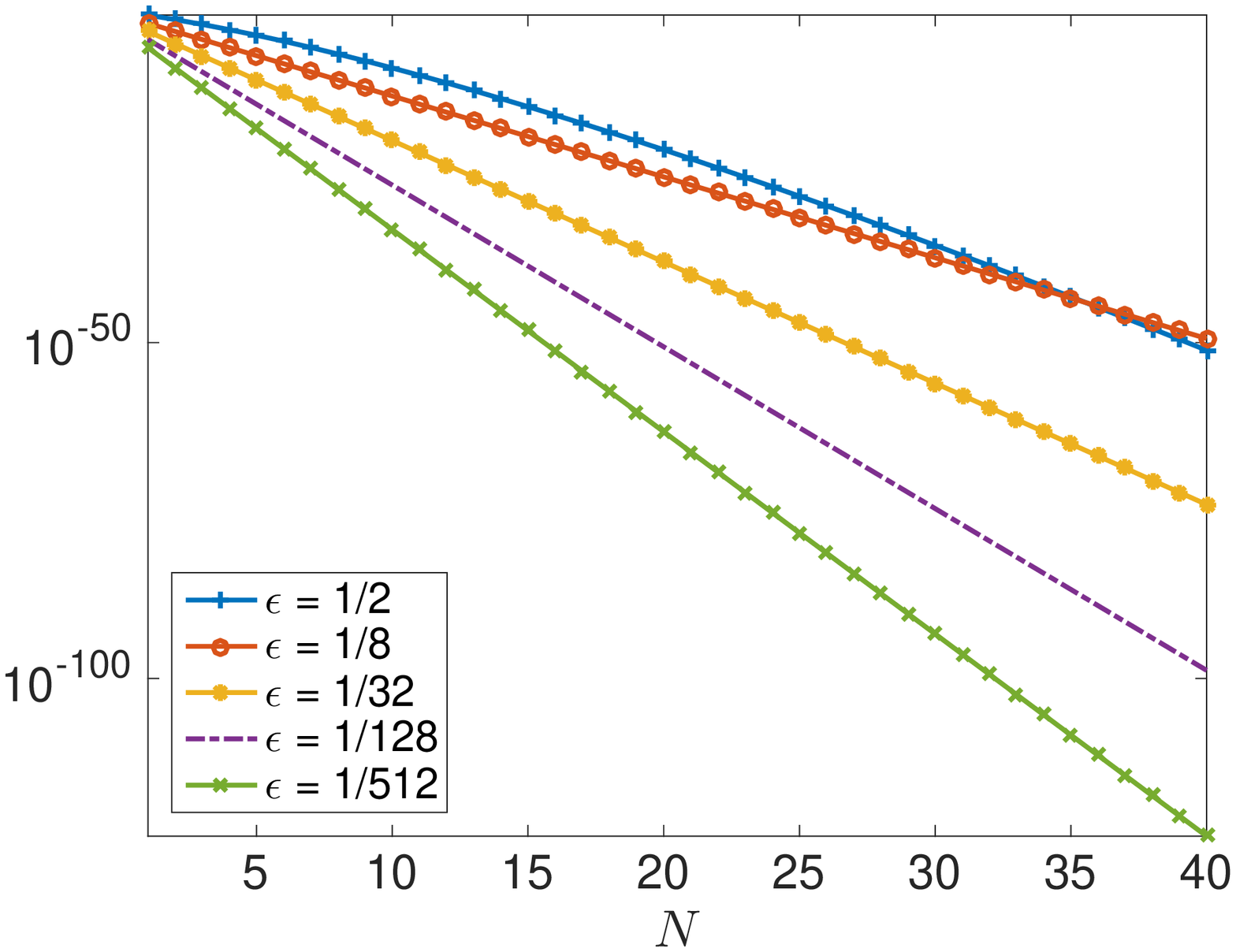}\label{fig:s3_2_1_fig5}}
		\end{subfigure}%
		\hspace{10pt}
		\begin{subfigure}[$t=10$, $g^{(3)}$]
			{\includegraphics[width=.31\linewidth]{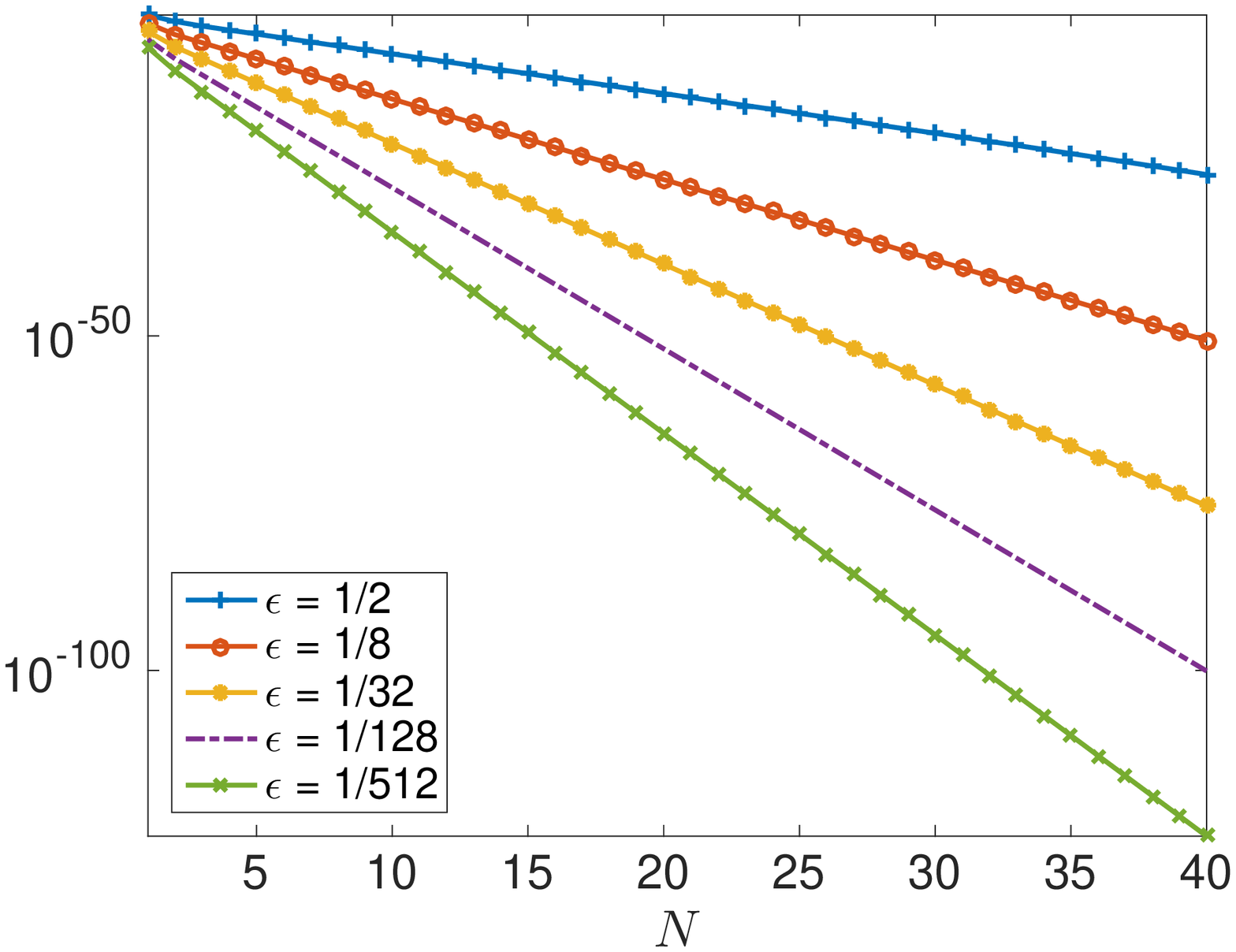}\label{fig:s3_2_10_fig5}}
		\end{subfigure}%
		\\
		\hspace{4pt}
		\begin{subfigure}[$t=0.1$, $g^{(3)}$]
			{\includegraphics[width=.3\linewidth]{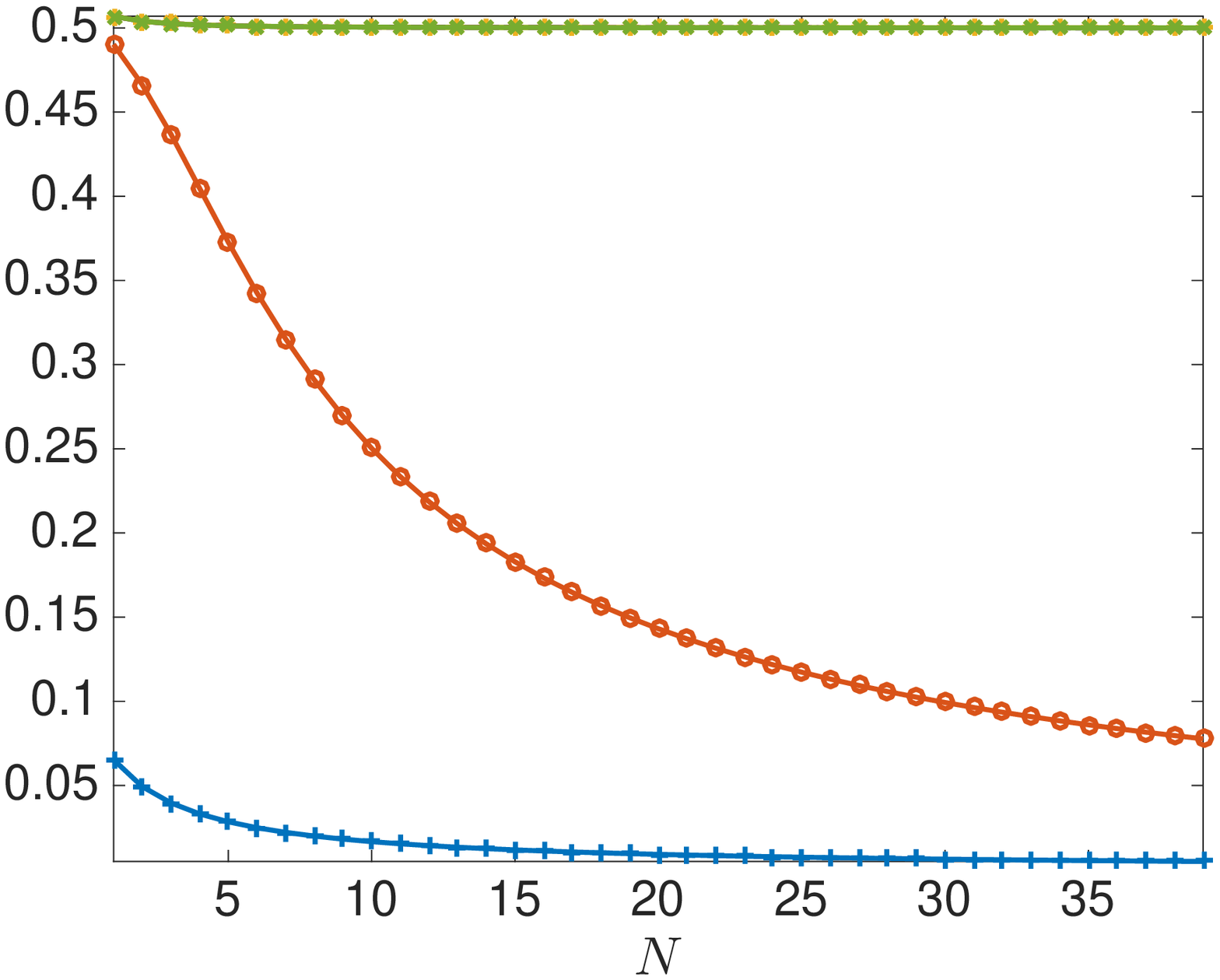}\label{fig:s3_2_001_fig8}}
		\end{subfigure}%
		\hspace{13pt}
		\begin{subfigure}[$t=1$, $g^{(3)}$]
			{\includegraphics[width=.3\linewidth]{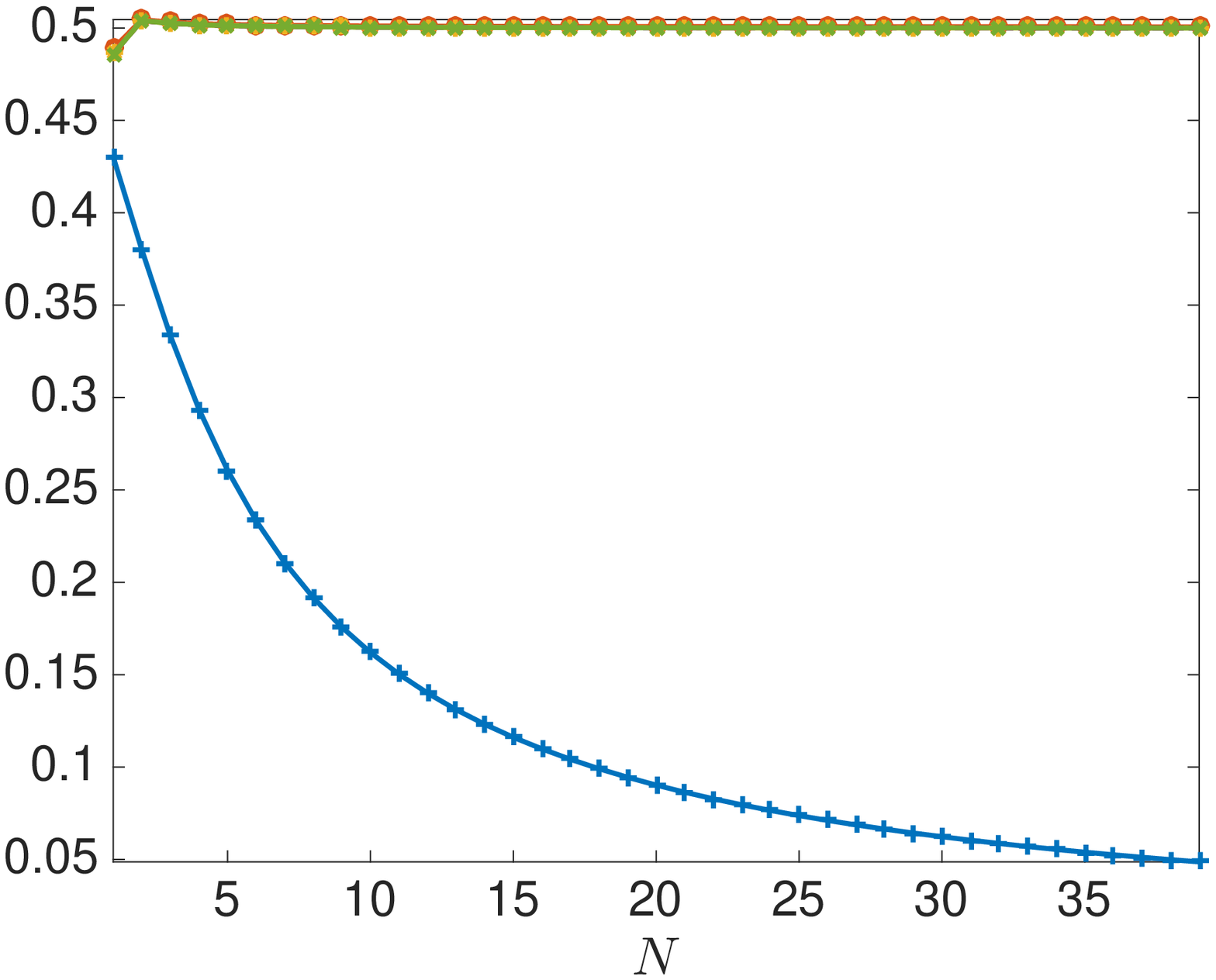}\label{fig:s3_2_1_fig8}}
		\end{subfigure}%
		\hspace{13pt}
		\begin{subfigure}[$t=10$, $g^{(3)}$]
			{\includegraphics[width=.3\linewidth]{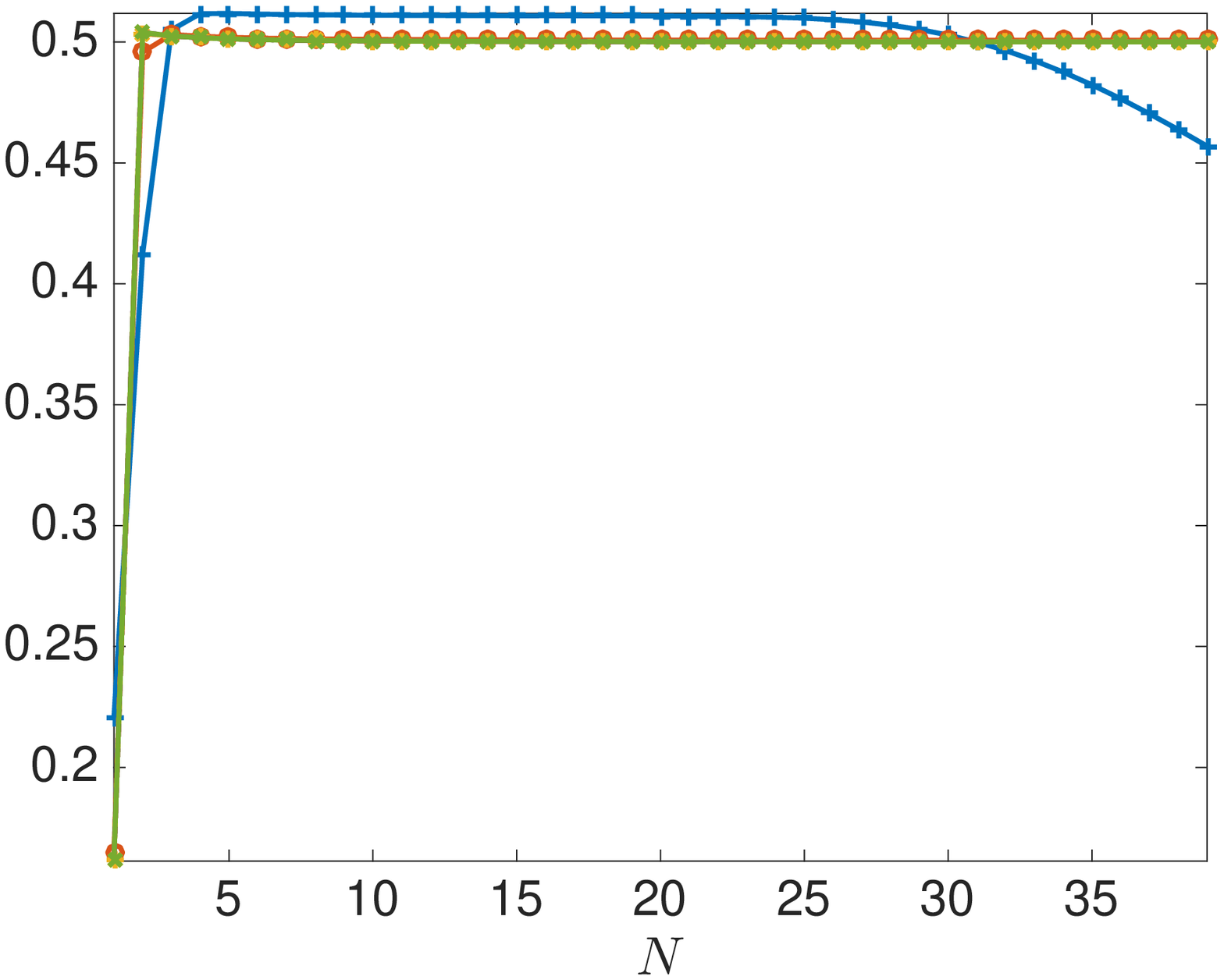}\label{fig:s3_2_10_fig8}}
		\end{subfigure}%
		\caption{Results from Example \ref{2eg:test3}.  Top figures: $\|e^N\|$.  Bottom figures: $\frac{\|e^{N+1}\|}{\|e^N\|\eps}$.}
		\label{fig:2eg_test3_2_e}
	\end{figure}

	\begin{figure}
		\centering
		\begin{subfigure}[$t=0.1$, $g^{(3)}$]
			{\includegraphics[width=.31\linewidth]{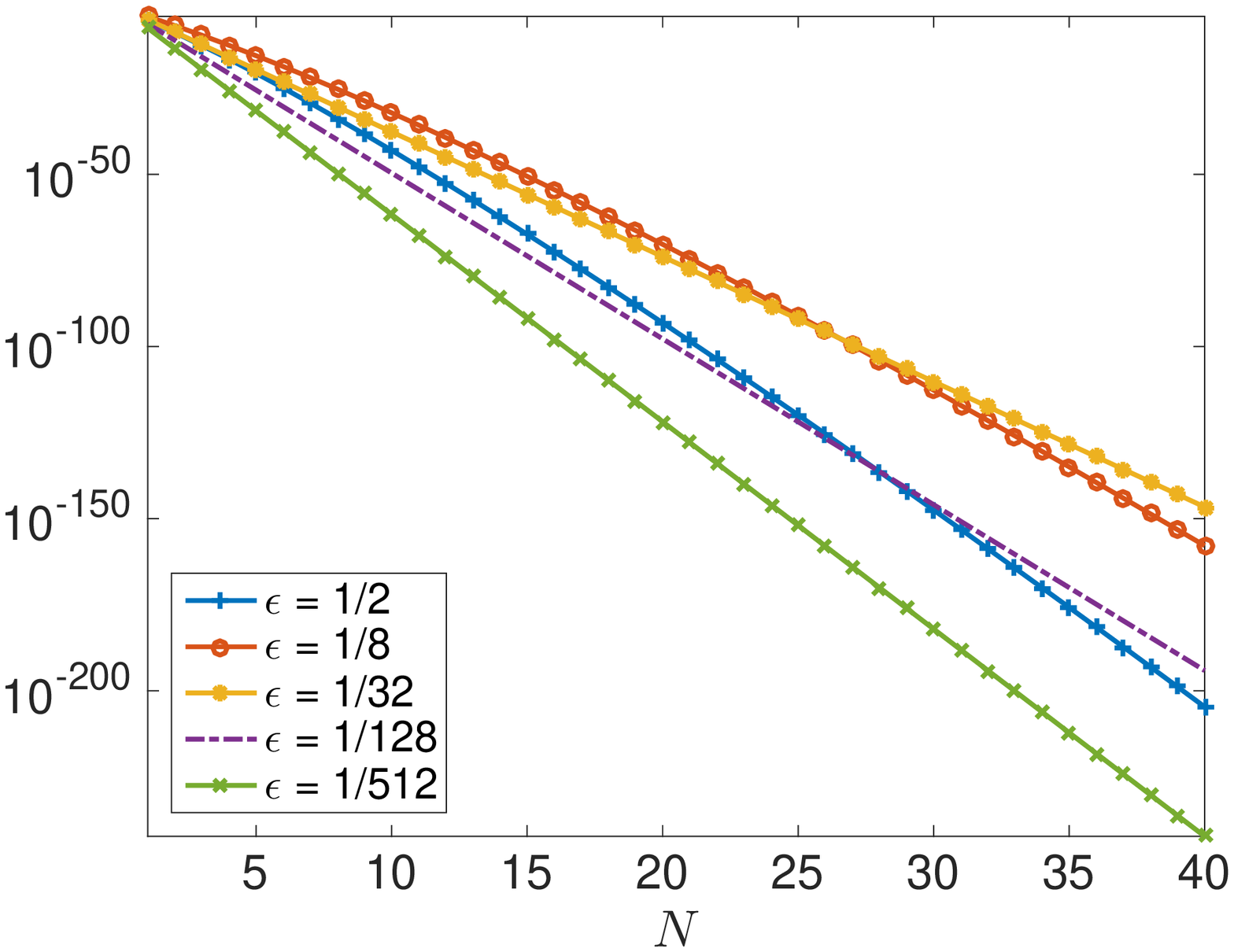}\label{fig:s3_2_001_fig05}}
		\end{subfigure}%
		\hspace{10pt}
		\begin{subfigure}[$t=1$, $g^{(3)}$]
			{\includegraphics[width=.31\linewidth]{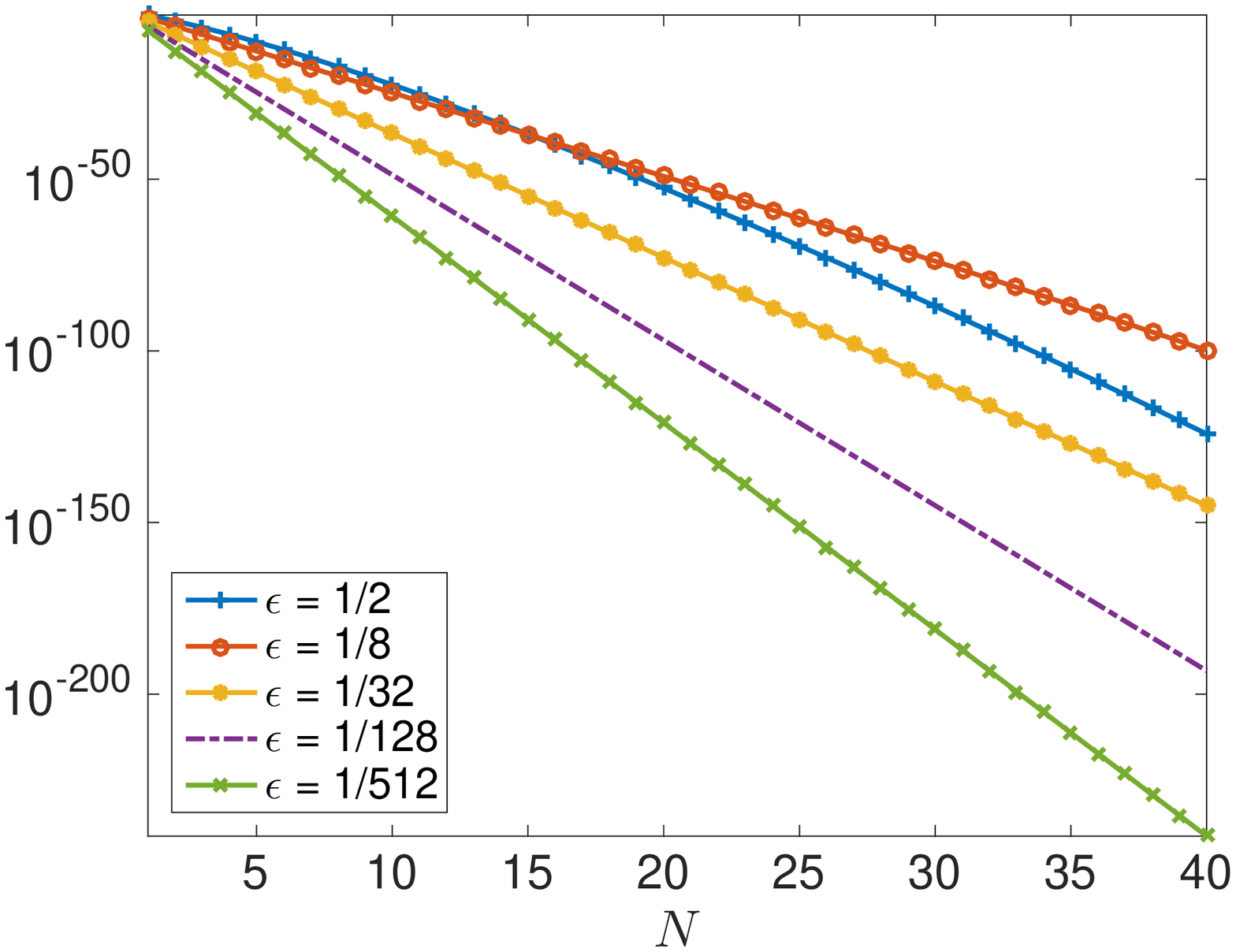}\label{fig:s3_2_1_fig05}}
		\end{subfigure}%
		\hspace{10pt}
		\begin{subfigure}[$t=10$, $g^{(3)}$]
			{\includegraphics[width=.31\linewidth]{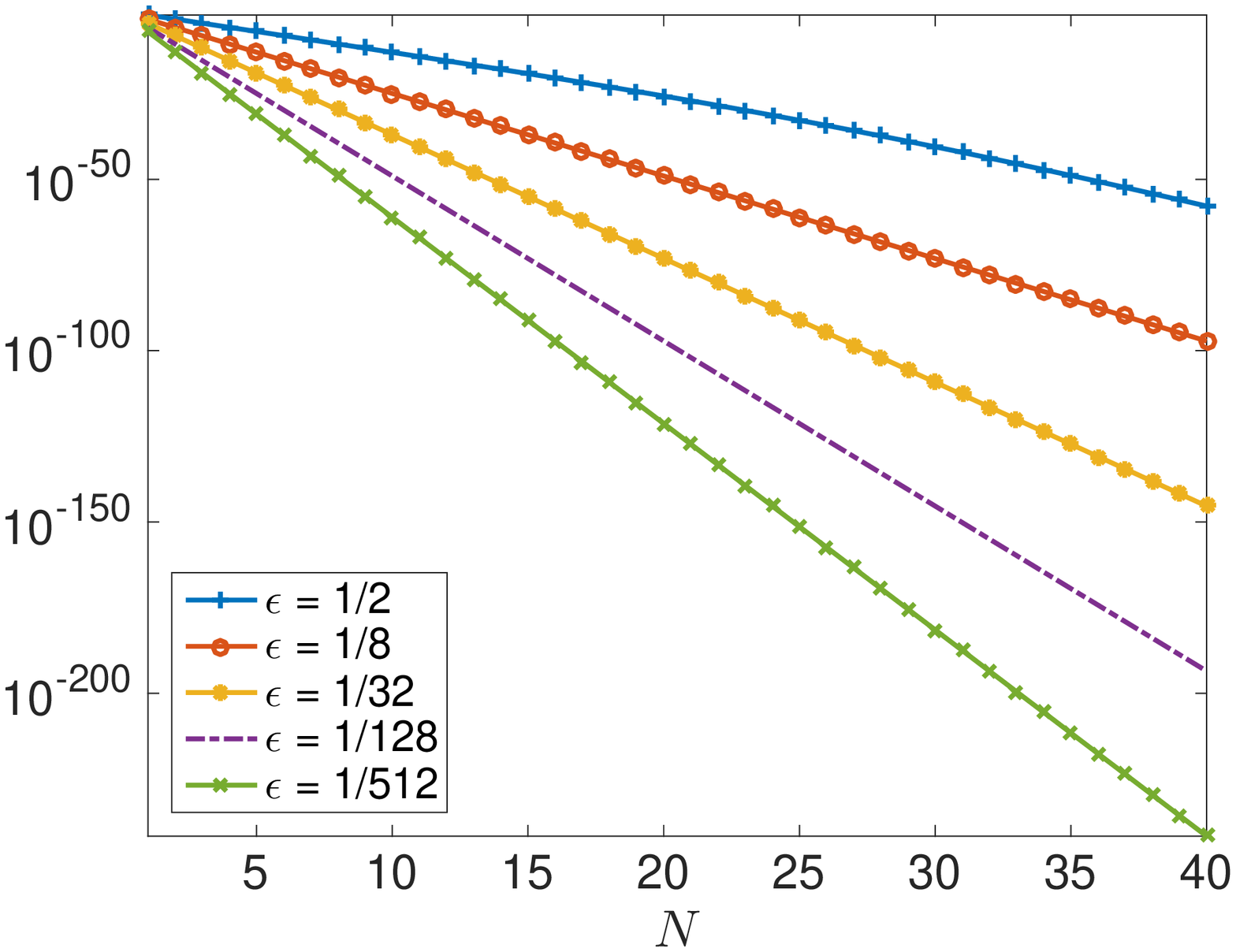}\label{fig:s3_2_10_fig05}}
		\end{subfigure}%
		\\
		\hspace{4pt}
		\begin{subfigure}[$t=0.1$, $g^{(3)}$]
			{\includegraphics[width=.3\linewidth]{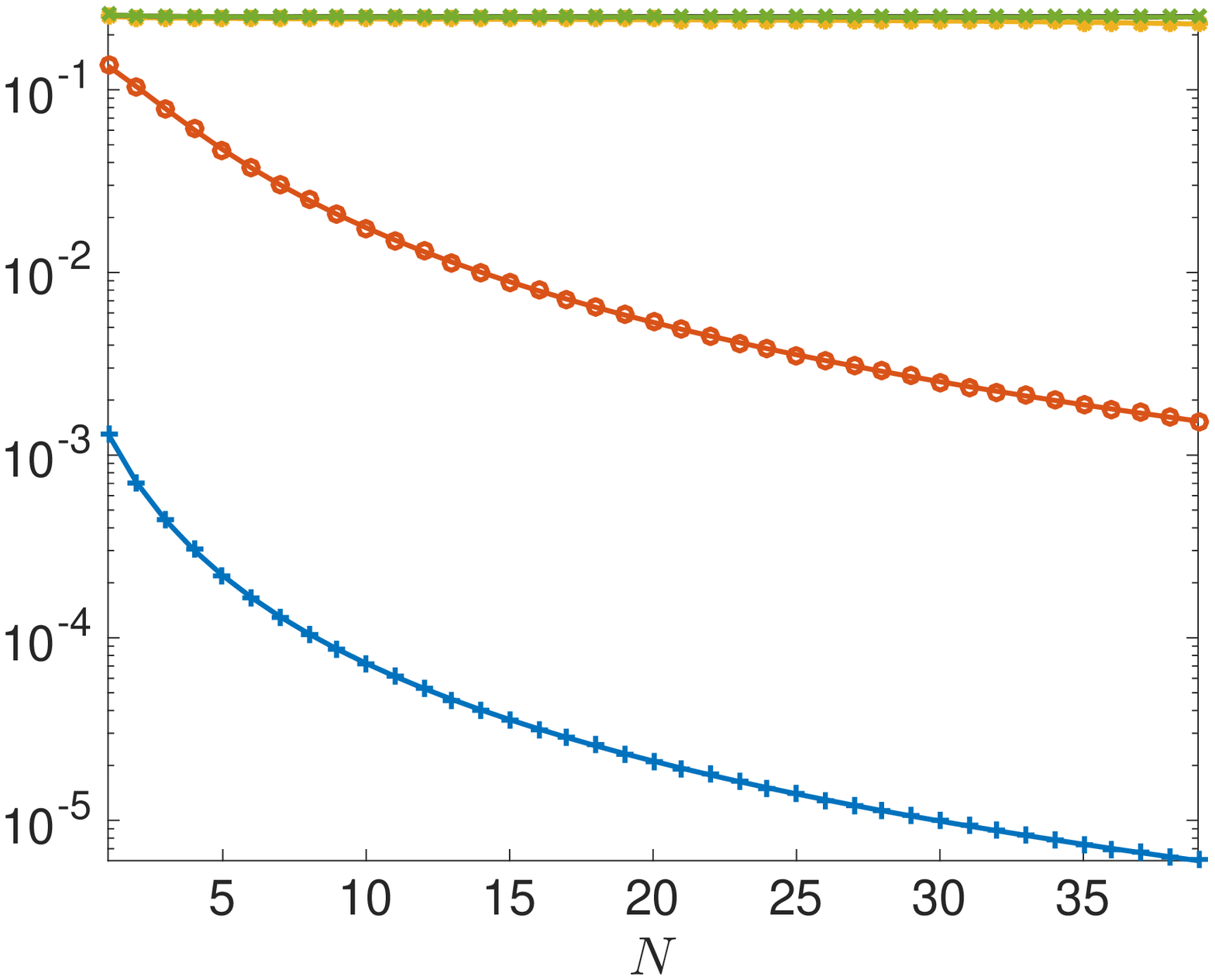}\label{fig:s3_2_001_fig08}}
		\end{subfigure}%
		\hspace{13pt}
		\begin{subfigure}[$t=1$, $g^{(3)}$]
			{\includegraphics[width=.3\linewidth]{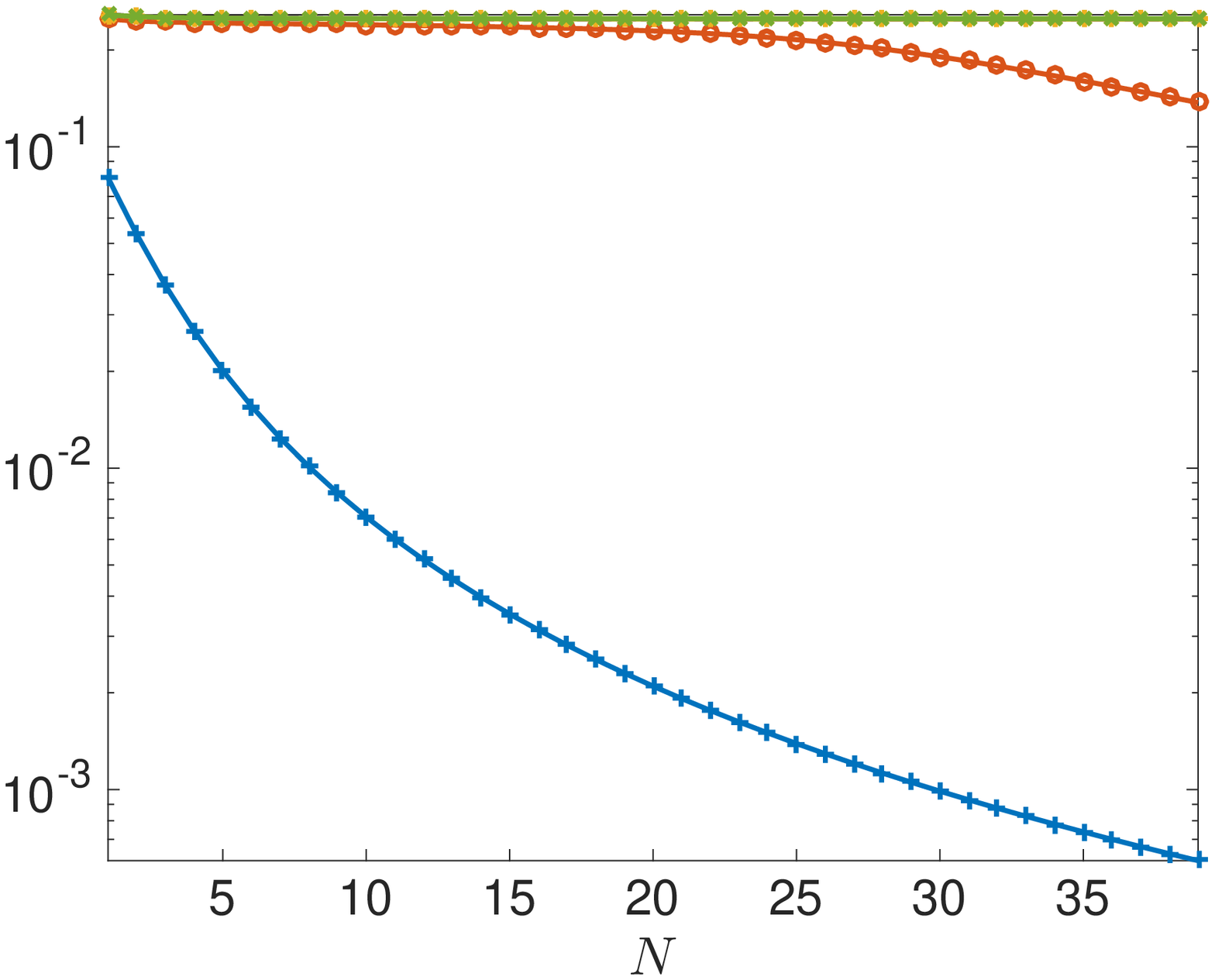}\label{fig:s3_2_1_fig08}}
		\end{subfigure}%
		\hspace{13pt}
		\begin{subfigure}[$t=10$, $g^{(3)}$]
			{\includegraphics[width=.3\linewidth]{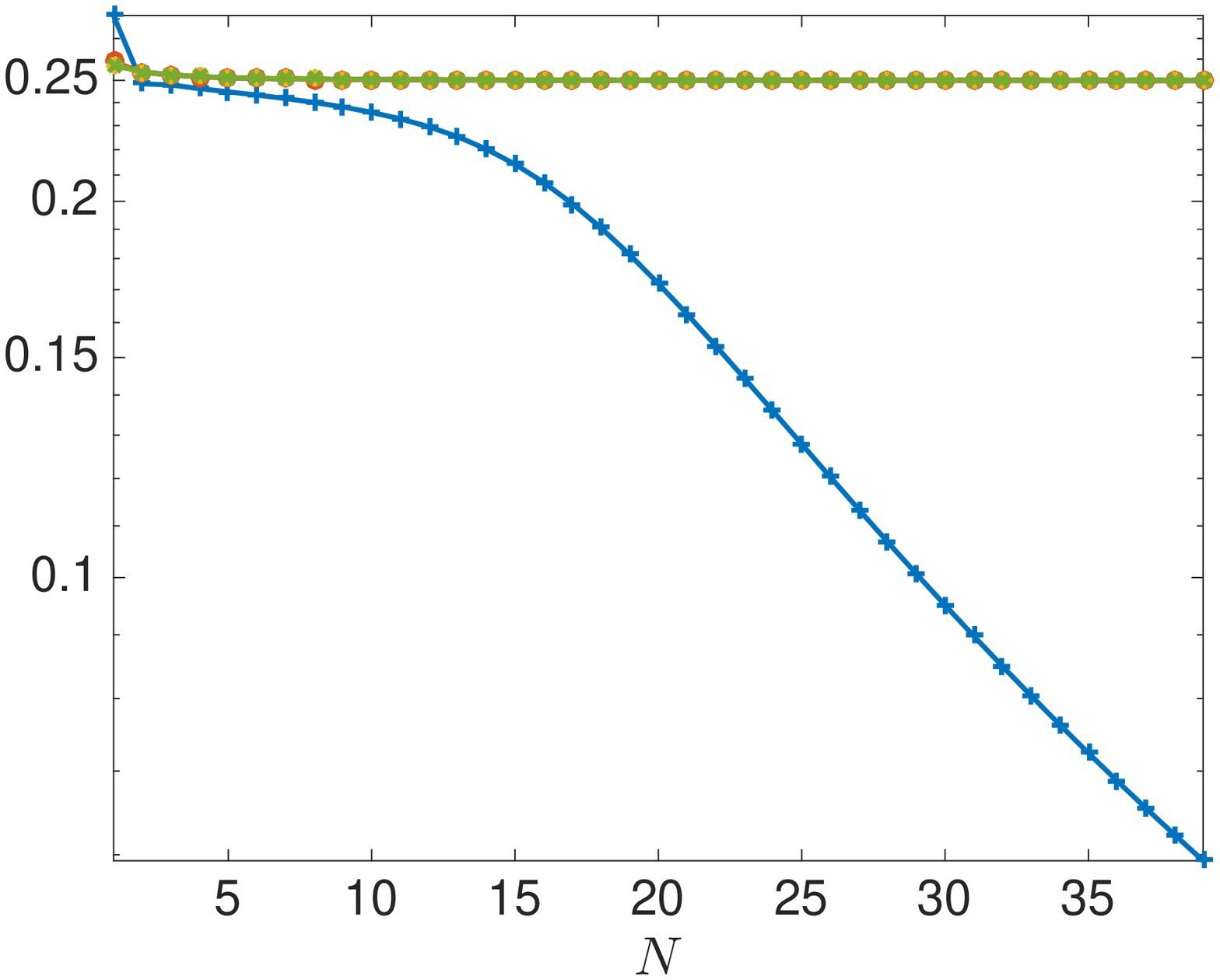}\label{fig:s3_2_10_fig08}}
		\end{subfigure}%
		\caption{Results from Example \ref{2eg:test3}.  Top figures: $\|\eNzeroN\|$. Bottom figures: $\frac{\|\eNzeroNpo\|}{\|\eNzeroN\|\eps^2}$.
		}
		\label{fig:2eg_test3_2_err_f0}
	\end{figure}
	
	\begin{figure}
		\centering
		\begin{subfigure}[$t=0.1$, $g^{(3)}$]
			{\includegraphics[width=.31\linewidth]{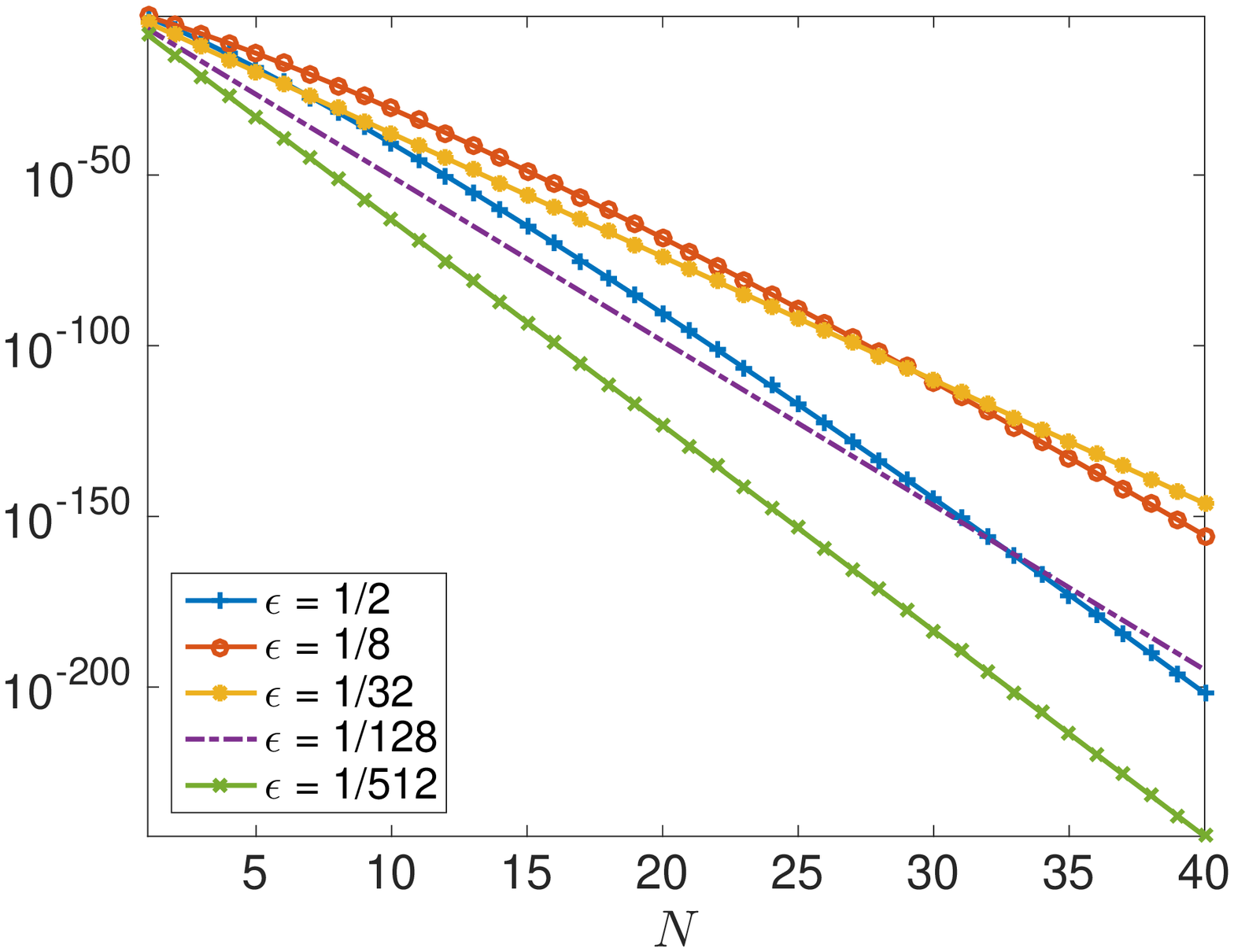}\label{fig:s3_2_001_fig15}}
		\end{subfigure}%
		\hspace{10pt}
		\begin{subfigure}[$t=1$, $g^{(3)}$]
			{\includegraphics[width=.31\linewidth]{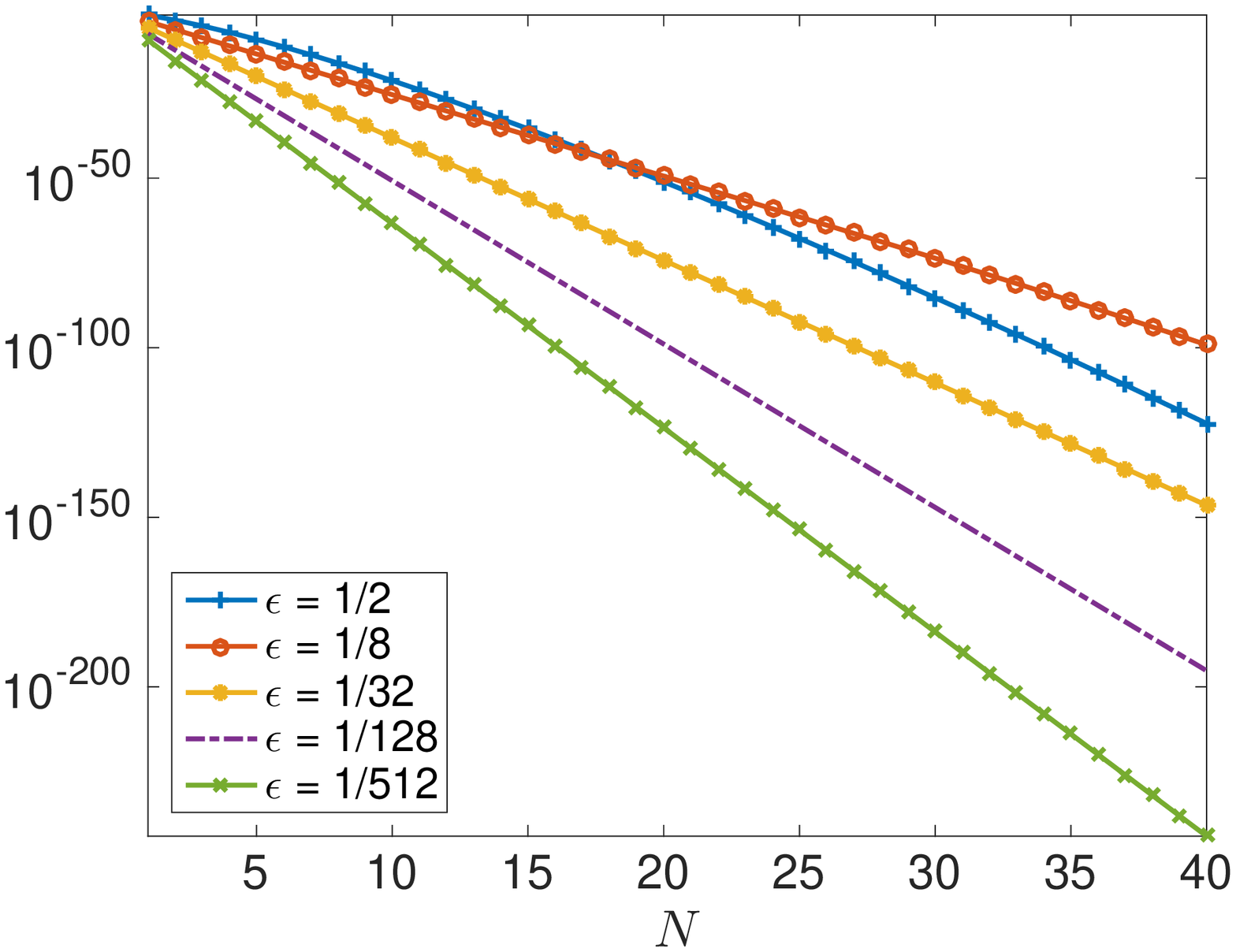}\label{fig:s3_2_1_fig15}}
		\end{subfigure}%
		\hspace{10pt}
		\begin{subfigure}[$t=10$, $g^{(3)}$]
			{\includegraphics[width=.31\linewidth]{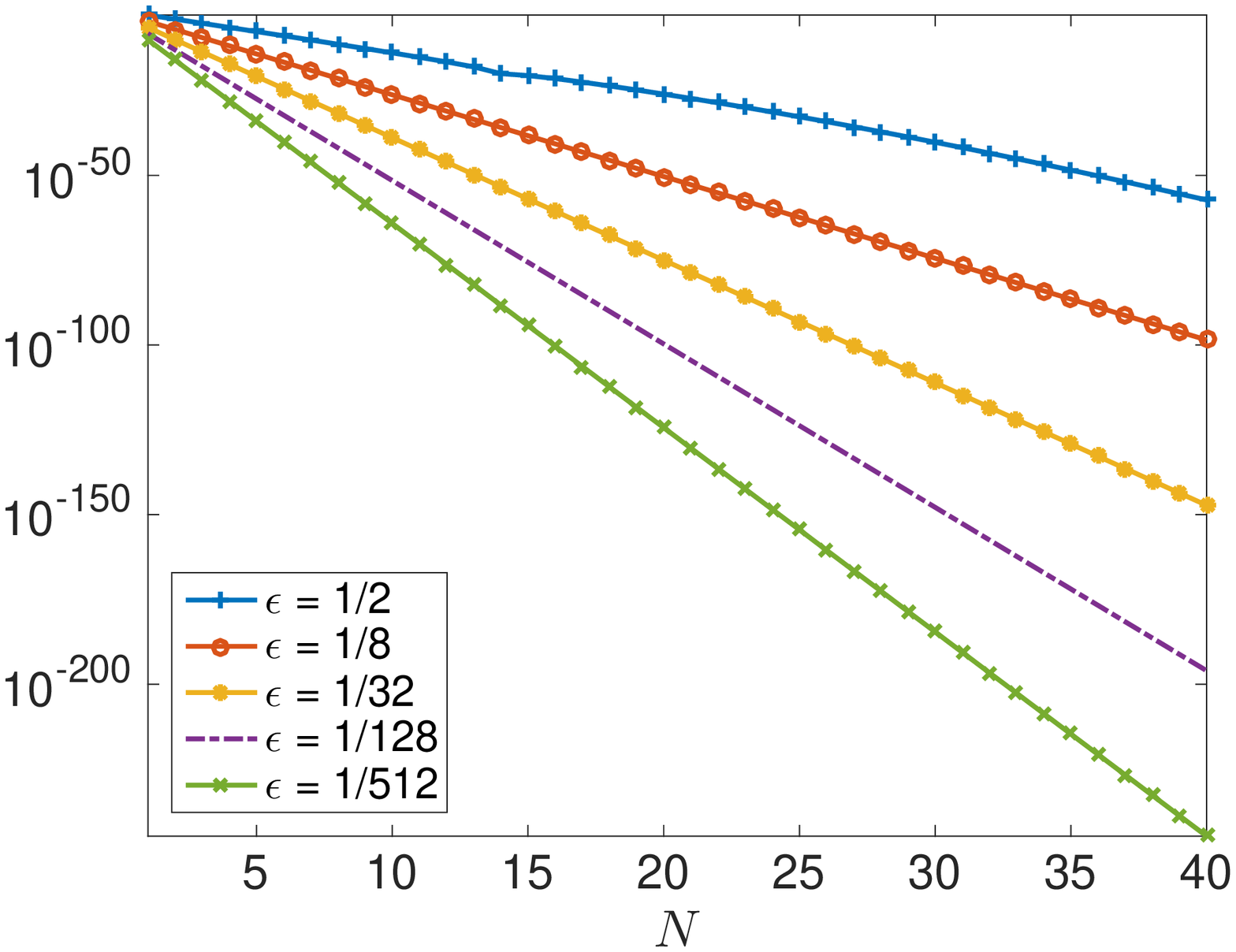}\label{fig:s3_2_10_fig15}}
		\end{subfigure}%
		\\
		\hspace{4pt}
		\begin{subfigure}[$t=0.1$, $g^{(3)}$]
			{\includegraphics[width=.3\linewidth]{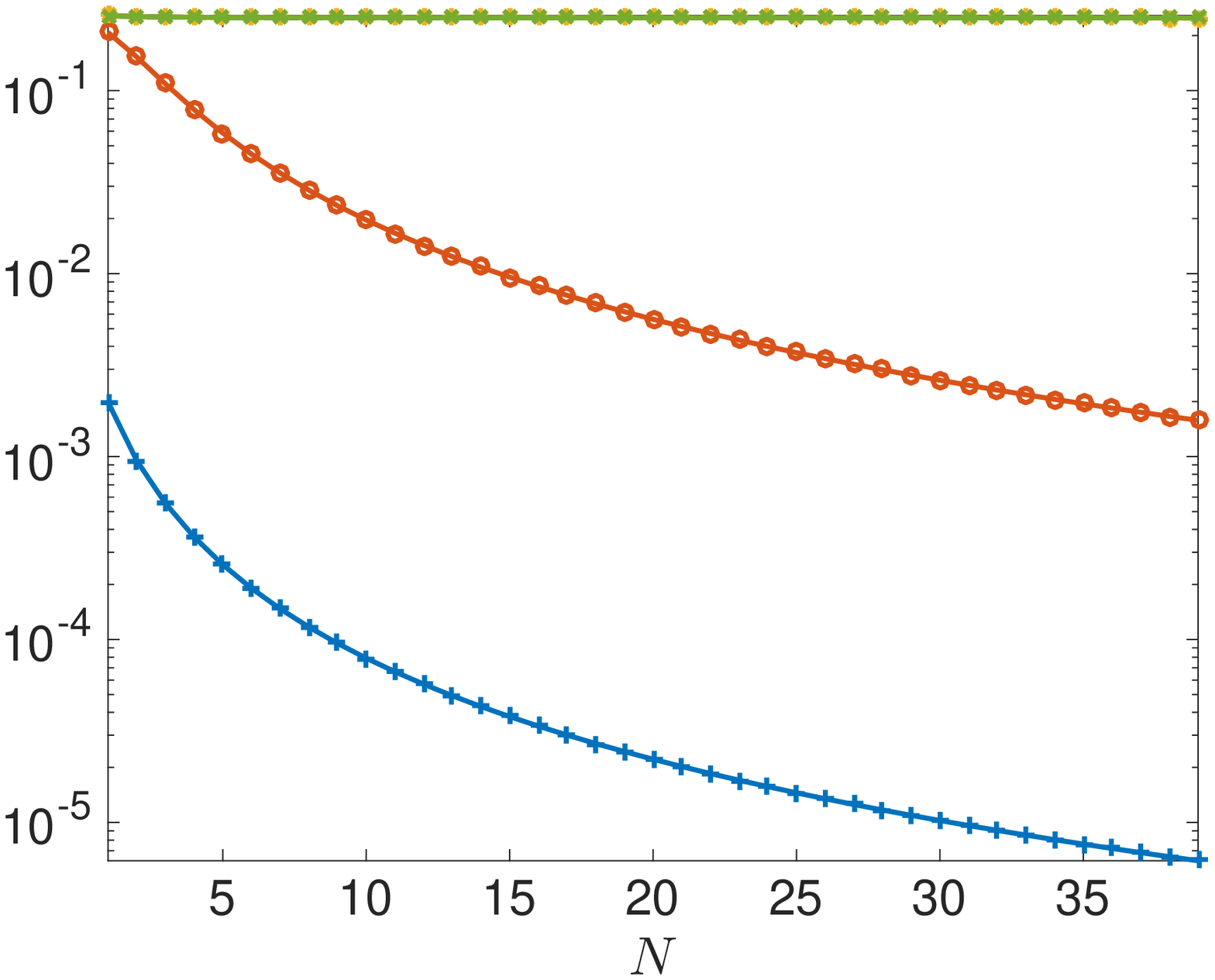}\label{fig:s3_2_001_fig18}}
		\end{subfigure}%
		\hspace{13pt}
		\begin{subfigure}[$t=1$, $g^{(3)}$]
			{\includegraphics[width=.3\linewidth]{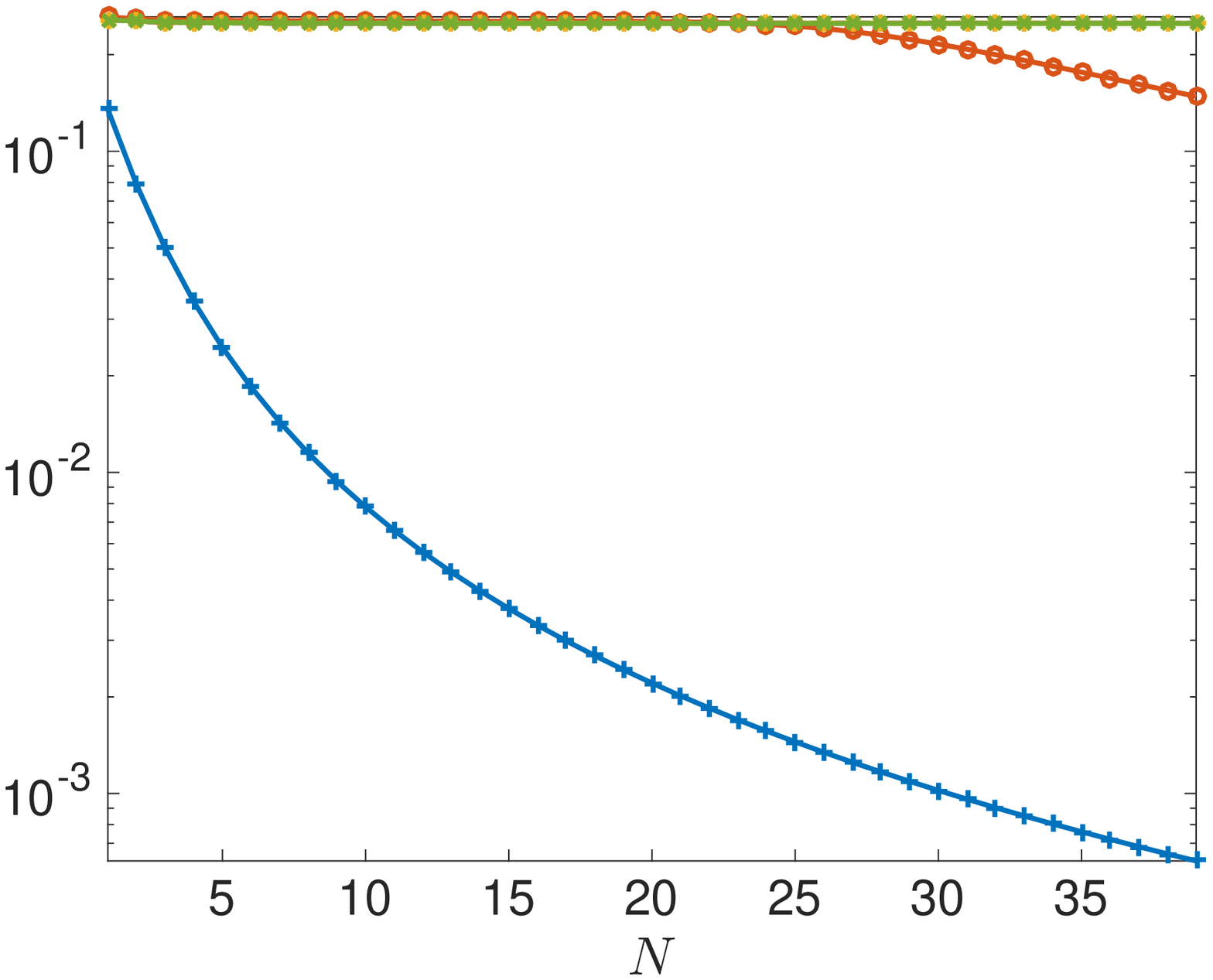}\label{fig:s3_2_1_fig18}}
		\end{subfigure}%
		\hspace{13pt}
		\begin{subfigure}[$t=10$, $g^{(3)}$]
			{\includegraphics[width=.3\linewidth]{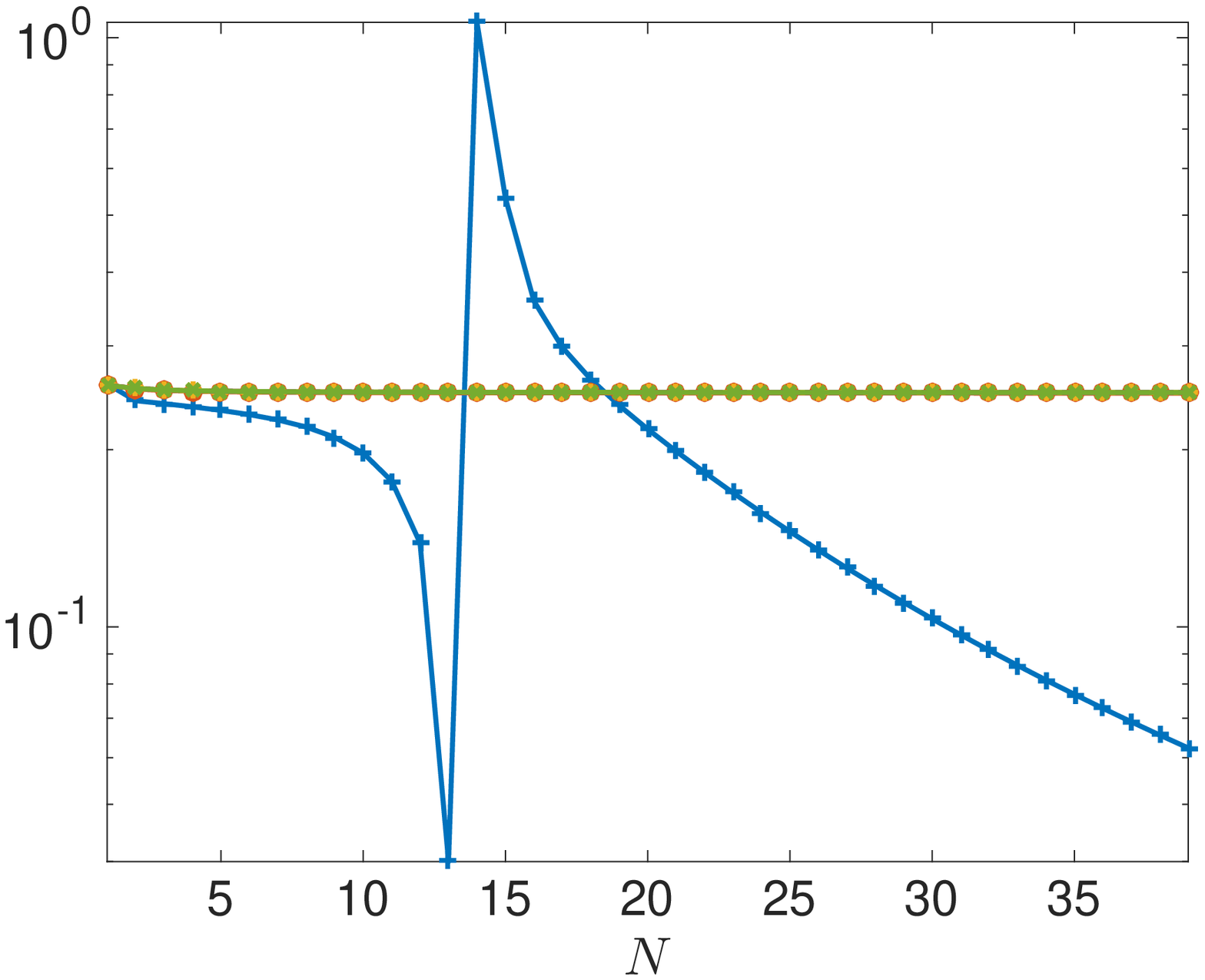}\label{fig:s3_2_10_fig18}}
		\end{subfigure}%
		\caption{Results from Example \ref{2eg:test3}.  Top figures: $\|\eNoneN\|$.  Bottom figures: $\frac{\|\eNoneNpo\|}{\|\eNoneN\|\eps^2}$.}
		\label{fig:2eg_test3_2_err_f1}
	\end{figure}
	
	\begin{figure}
		\centering
		\begin{subfigure}[$t=0.1$, $g^{(3)}$]
			{\includegraphics[width=.31\linewidth]{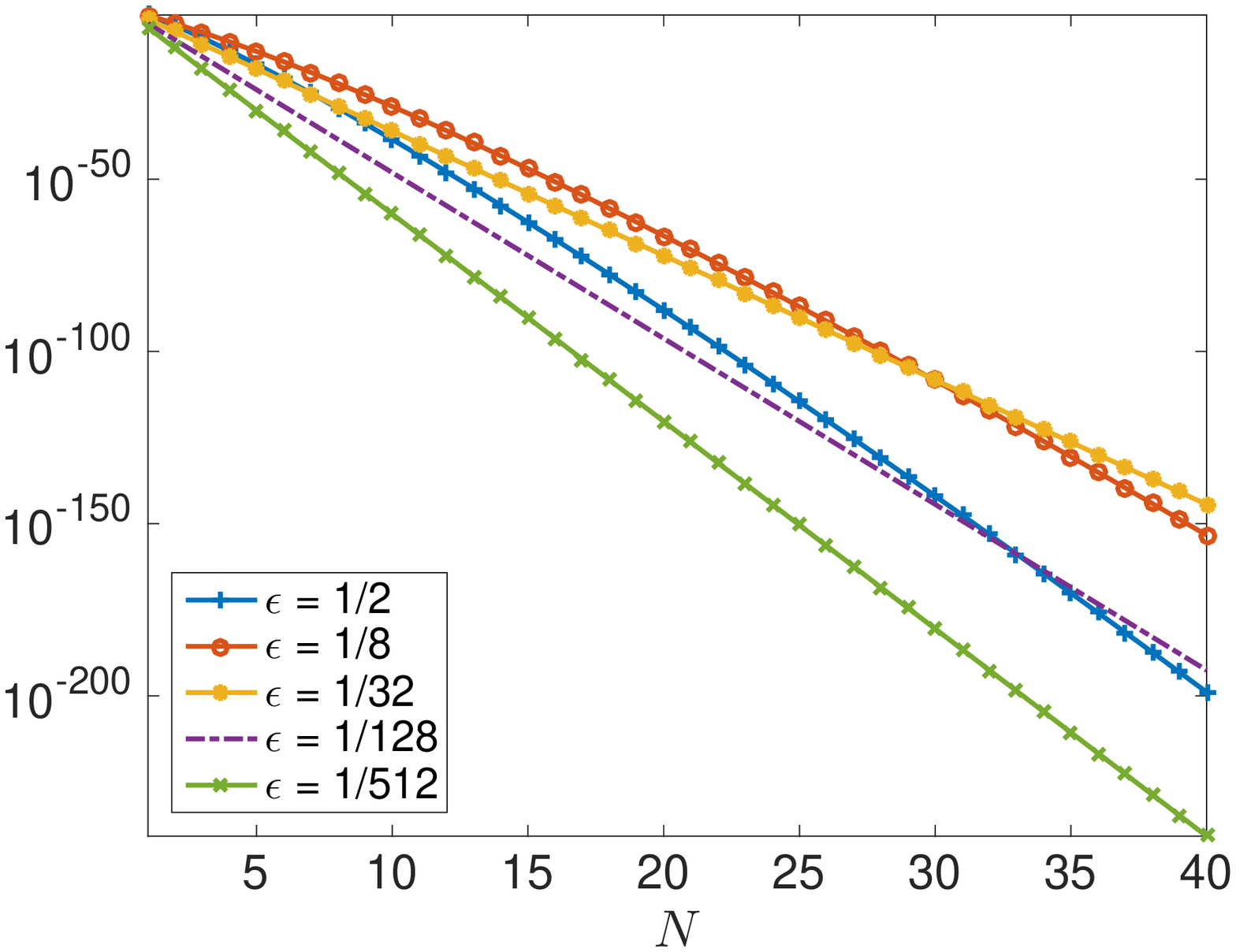}\label{fig:s3_2_001_fig25}}
		\end{subfigure}%
		\hspace{10pt}
		\begin{subfigure}[$t=1$, $g^{(3)}$]
			{\includegraphics[width=.31\linewidth]{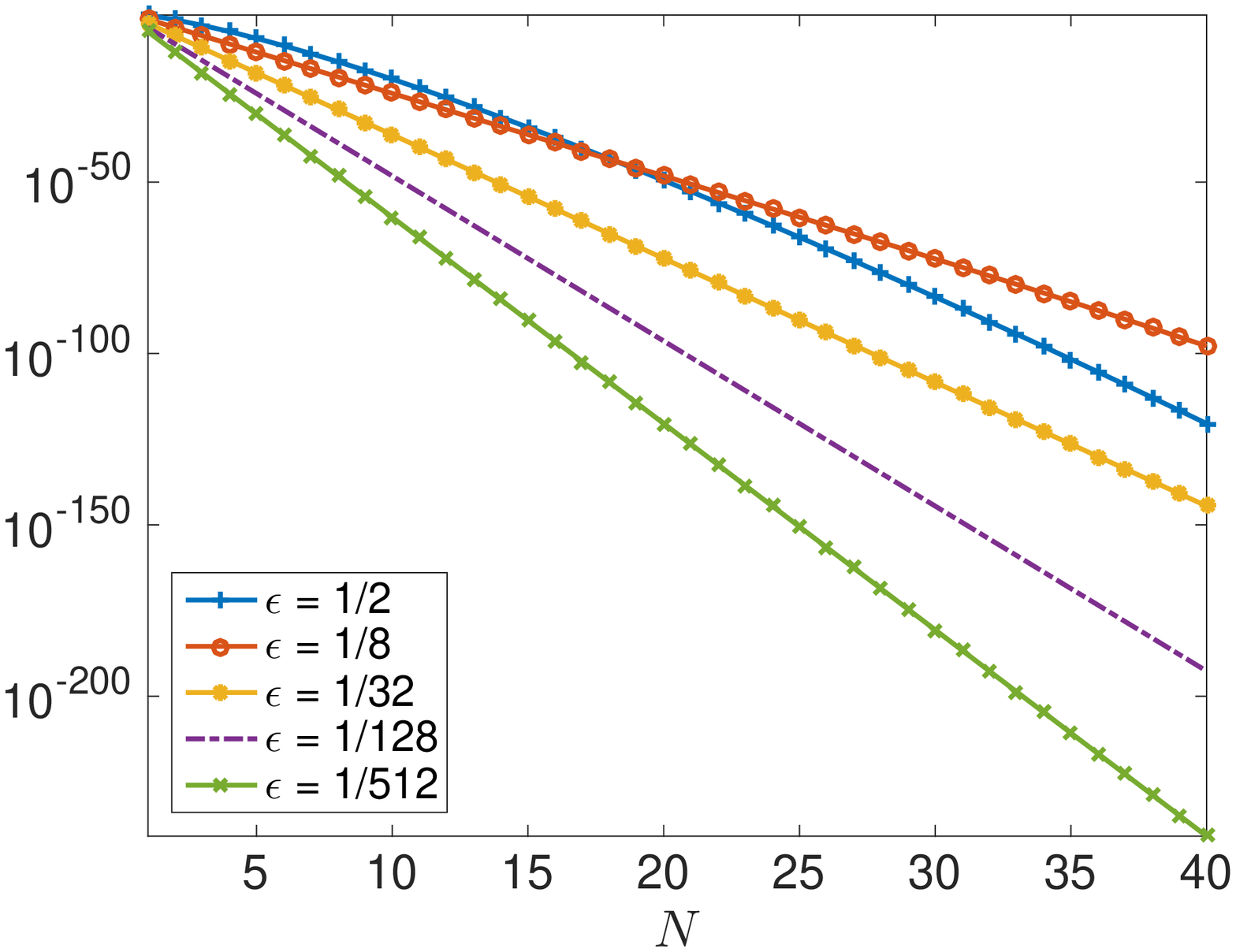}\label{fig:s3_2_1_fig25}}
		\end{subfigure}%
		\hspace{10pt}
		\begin{subfigure}[$t=10$, $g^{(3)}$]
			{\includegraphics[width=.31\linewidth]{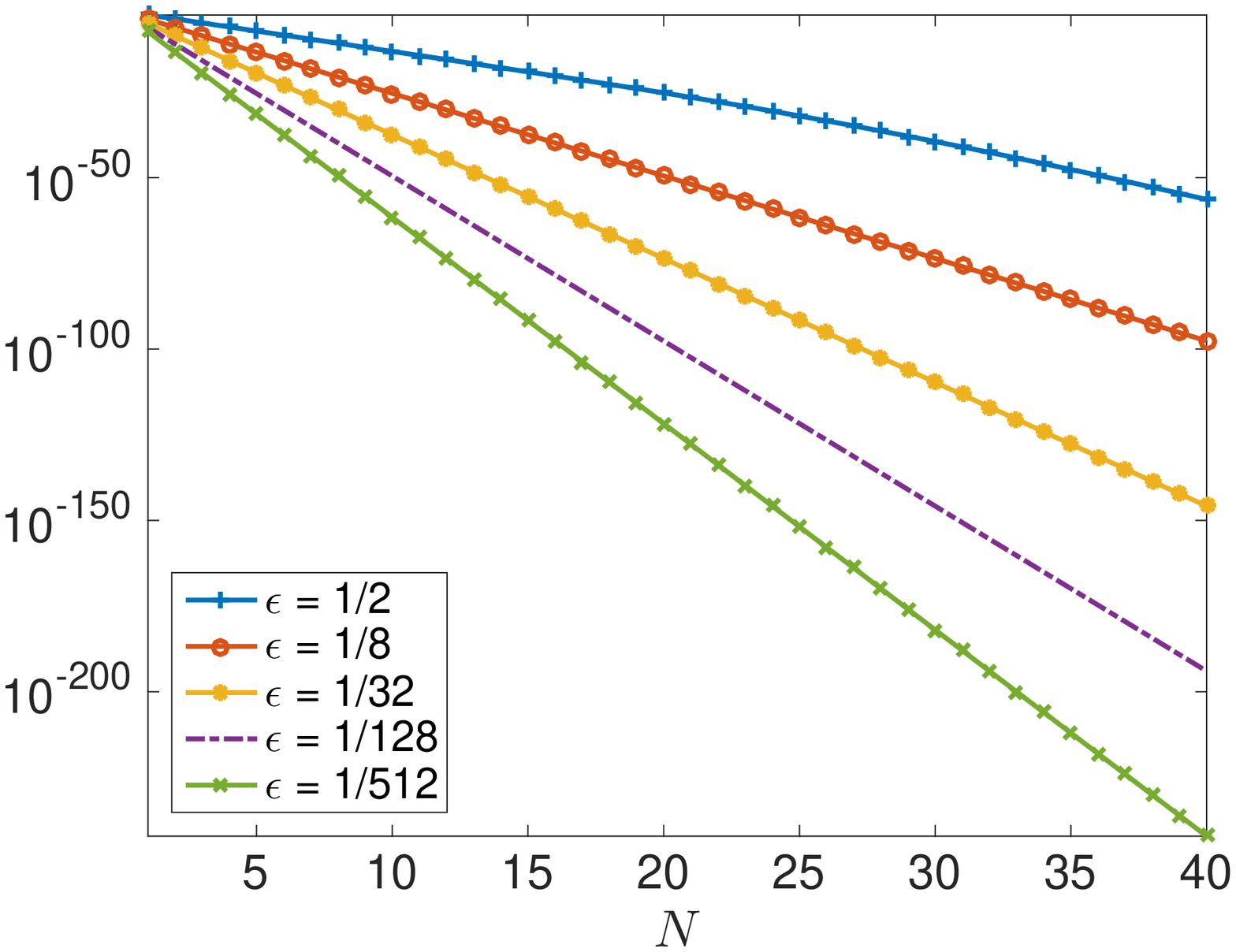}\label{fig:s3_2_10_fig25}}
		\end{subfigure}%
		\\
		\hspace{4pt}
		\begin{subfigure}[$t=0.1$, $g^{(3)}$]
			{\includegraphics[width=.3\linewidth]{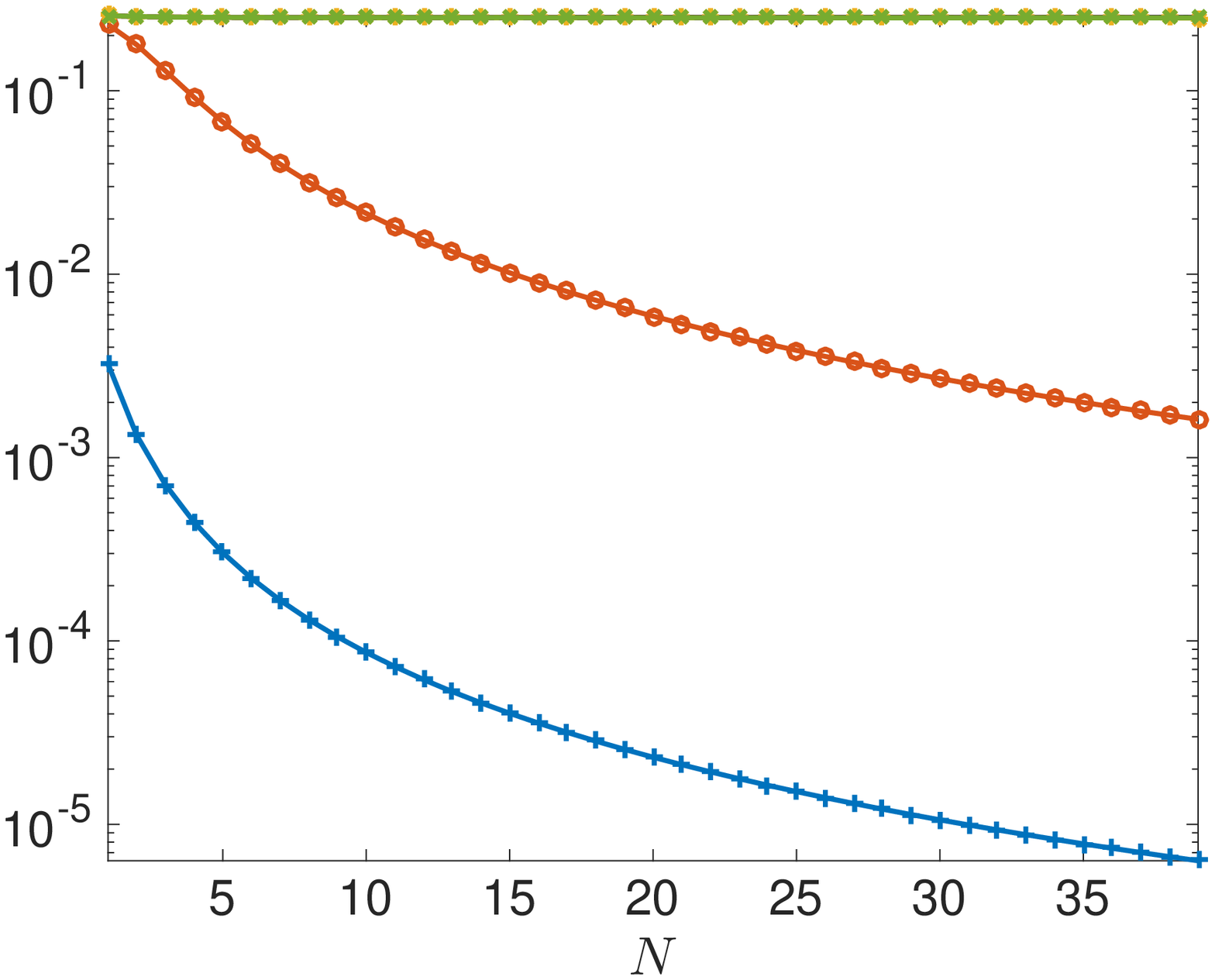}\label{fig:s3_2_001_fig28}}
		\end{subfigure}%
		\hspace{13pt}
		\begin{subfigure}[$t=1$, $g^{(3)}$]
			{\includegraphics[width=.3\linewidth]{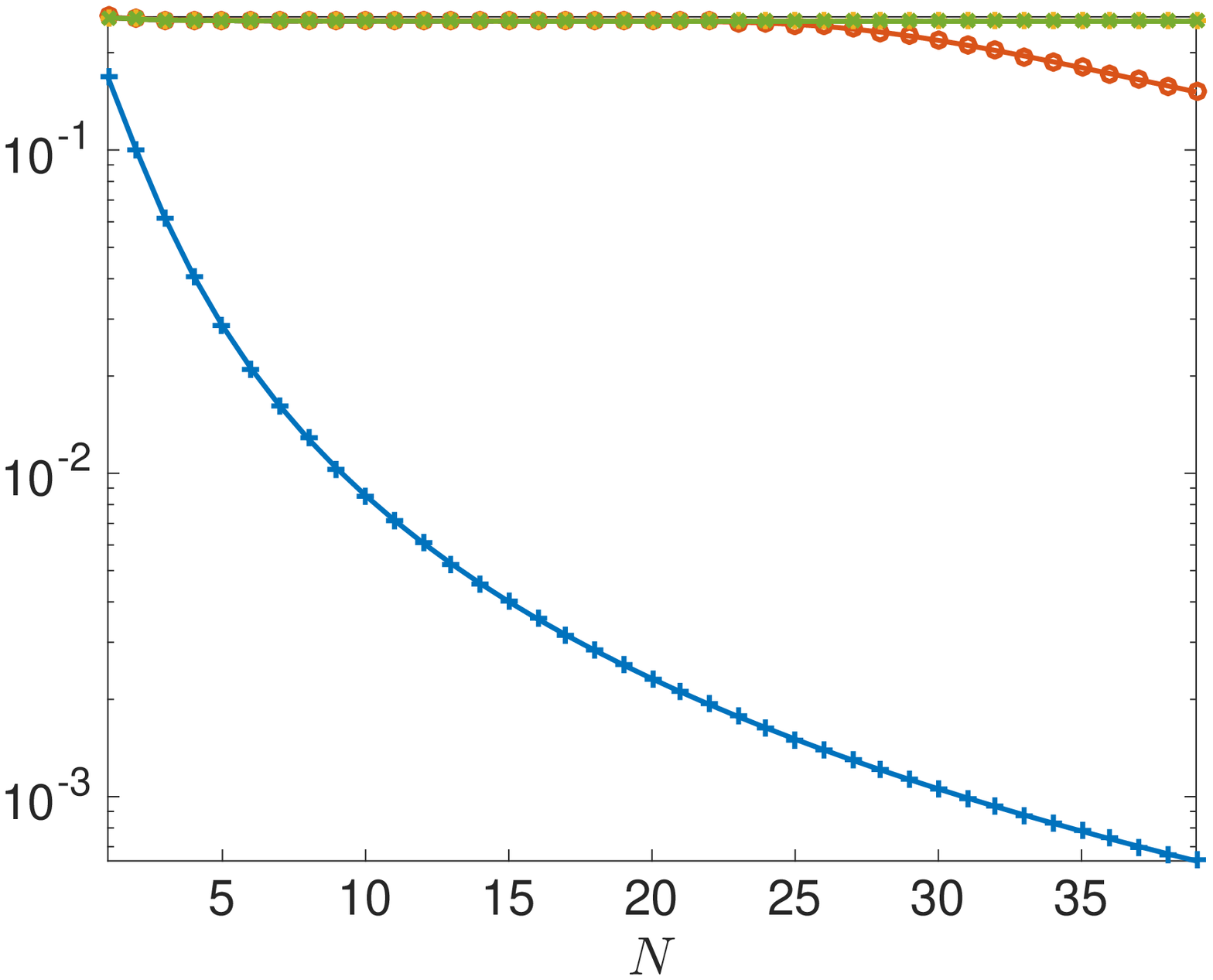}\label{fig:s3_2_1_fig28}}
		\end{subfigure}%
		\hspace{13pt}
		\begin{subfigure}[$t=10$, $g^{(3)}$]
			{\includegraphics[width=.3\linewidth]{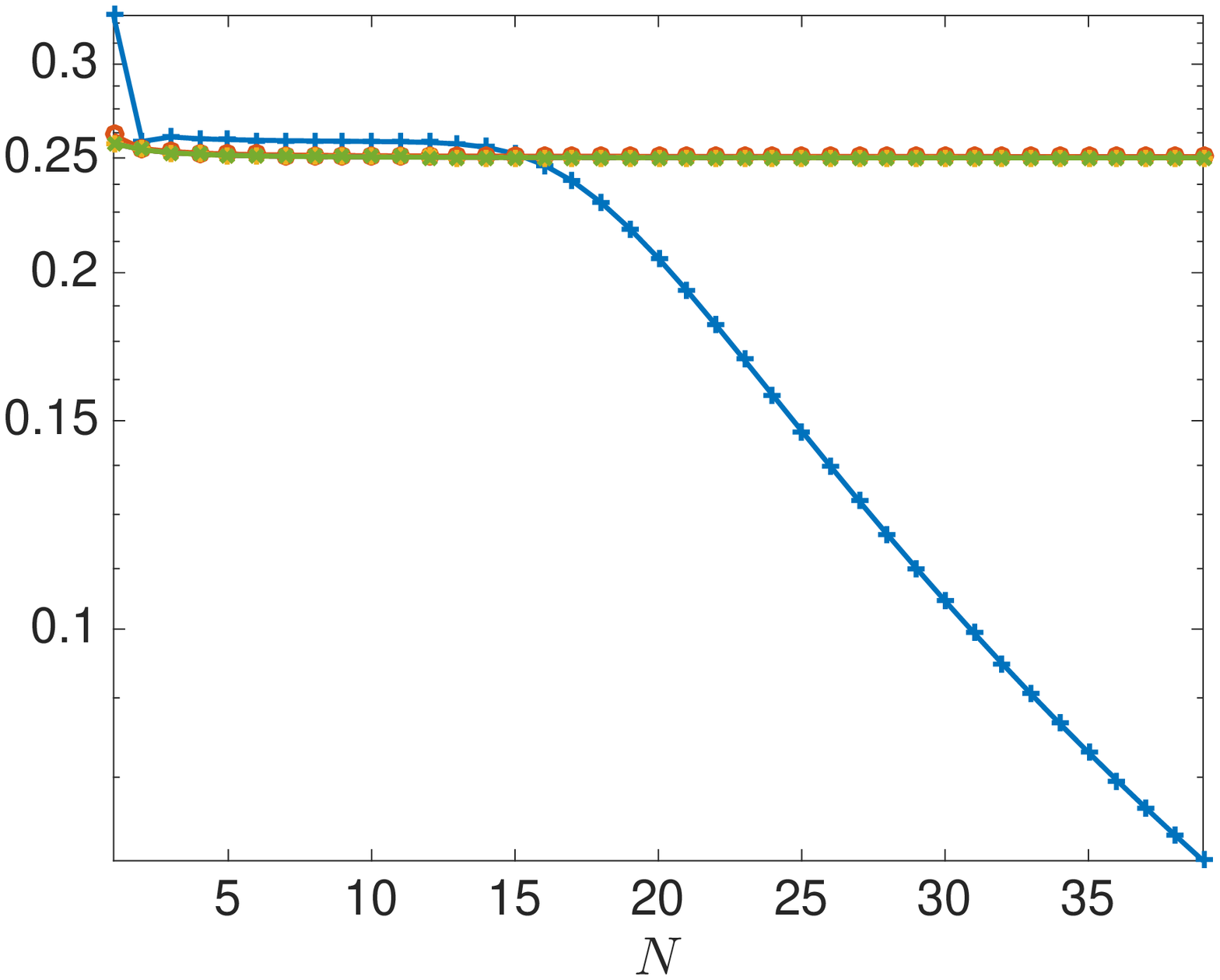}\label{fig:s3_2_10_fig28}}
		\end{subfigure}%
		\caption{Results from Example \ref{2eg:test3}.  Top figures: $\|\eNtwoN\|$.  Bottom figures: $\frac{\|\eNtwoNpo\|}{\|\eNtwoN\|\eps^2}$.}
		\label{fig:2eg_test3_2_err_f2}
	\end{figure}

 	In Figure \ref{fig:2eg_test3_2_e}, the ratio in \eqref{eqn:spectral_error_increase_N_eps} (normalized by $\eps$) is plotted as a function of $N$.  In Figures \ref{fig:2eg_test3_2_err_f0}--\ref{fig:2eg_test3_2_err_f2}, the ratio in \eqref{eqn:moment_error_increase_N_eps} (normalized by $\eps^{2}$) is plotted for $\ell=0,1,2$ as a function of $N$.  As in the previous example, profiles of the normalized error ratios appear to convergence at $\eps$ decreases.  However, unlike the previous example, the ratios do not appear to decay significantly as time increases.  Indeed, they are already less than one for $t=0.1$.  Numerically, we see that $e^{N+1} \leq 0.5 \eps e^N$ and $\eNpolN \leq 0.25 \eps^2 \eNlN$ ($\ell = 0,1,2$) for all three tested value of $t$.  Also, we do not observe the plateaus and transitions seen in the previous example.
\end{eg}

\subsection{Quantifying coefficients in the error estimates}%
\label{sec:constant_math}
The manner in which the estimates in \eqref{eqn:L2_error} and \eqref{eqn:moment_error} depend on $N$ and $\ell$ can ultimately be traced back to the coefficient $F(g,n,t)$, defined in \eqref{eqn:C_glt}.  Indeed the results in \eqref{eqn:est_fl_low}, \eqref{eqn:est_eta_low}, \eqref{eqn:est_hlow_xi}, \eqref{eqn:est_xi_low_0} and \eqref{eqn:est_xi_low_1} all depend on $F(g,n,t)$ for some value of $n$:  in \eqref{eqn:est_fl_low}, $n = \ell$; in \eqref{eqn:est_eta_low}, $n = N+1$; in \eqref{eqn:est_hlow_xi}, $n = N+2$; in \eqref{eqn:est_xi_low_0}, $n = 3N$, and in \eqref{eqn:est_xi_low_1}, $n = 3N+4-2\ell$.  The dependence of $F(g,n,t)$ on $n$ arises via the term $\an(2\lambda_2 t)$, where (recall that) $\lambda_2 = 4/45$ and 
\beq \label{eqn:ant}
\an(s) : = \sum_{k>0} \left\{(Ak)^{2n} e^{-k^2 s}\right\}.  
\eeq 
For example, according to  \eqref{eqn:L2_error}  and \eqref{eqn:D_gnt}, after an initial layer, 
\beq
\label{eqn:e_upper_bound_strategy}
\|e^N\| (t) 
\leq \tilde{c}(t) \, \sqrt{a_{N+2}(2\lambda_2 t)} \, \eps^{N+1} ,
\quad \text{where} \quad \tilde{c}(t) = 2 \left(\sqrt{2} + \frac{\sqrt{t}}{A}\right) \left( 24 \max_{k>0} \Hzerok(g) \right)^{1/2} .
\eeq
Similarly, it follows from \eqref{eqn:moment_error}, \eqref{eqn:E_def}, and \eqref{eqn:C_glt} that  after an initial layer, 
\begin{align}
\label{eqn:el_upper_bound_strategy}
\| \eNlN \| (t) 
&\leq 
\tilde{d}   \left( \frac{N-n_\ell+2}{e A^2 \lambda_2}\right)^\frac{N-n_\ell+2}{2} 
\sqrt{a_{3N+4-2n_\ell}(\lambda_2 t)}
\, \eps^{2N+2-n_\ell},
\end{align}
where
\beq
n_\ell = \begin{cases}
	2,  & \ell = 0\\
	\ell, & 1 \leq \ell \leq N,
\end{cases}
\qquand 
\tilde{d} = 8 \left( 6 \, e^{\lambda_2 }  \max\limits_{k >0} \Hzerok(g) \right)^{1/2}.
\eeq

By interpreting the right-hand side of \eqref{eqn:ant} as a Riemann sum, we bound $\an$ as follows:  
\begin{multline}
\label{eqn:ans_integral_bound}
\an(s)  \leq A^{2n} \left( e^{-s}
+ \int_1^\infty (x+1)^{2n} e^{-s x^2 } dx 	\right) 
 \leq A^{2n} \left( e^{-s}
 + \int_1^\infty (2x)^{2n} e^{-s x^2 } dx 	\right) \\
 \leq (2A)^{2n} \left(   e^{-s}
 + \int_1^\infty x^{2n+1} e^{-s x^2 } dx 	\right) 
= (2A)^{2n} \left(  e^{-s}   + \frac{1}{2} e^{-s}  b_n(s)  \right), \qquad\qquad\quad
\end{multline}
where 
\beq
b_n(s) = \sum\limits_{k=0}^n \left(\frac{1}{s}\right)^{k+1} \frac{n!}{(n-k)!}, 
\qquad 	n \geq 0.
\eeq
Setting \eqref{eqn:ans_integral_bound} into \eqref{eqn:e_upper_bound_strategy} gives
\beq  \label{eqn:bounds_En}
\|e^{N}\|(t) \leq \tilde{c}(t) (2A)^{N+2}  e^{-\lambda_2 t}  \left( \frac{1}{2} b_{N+2}(2\lambda_2 t) + 1 \right)^{1/2} \, \eps^{N+1} 
=:E^N(t),
\eeq
and setting \eqref{eqn:ans_integral_bound} into \eqref{eqn:el_upper_bound_strategy} gives
\begin{align}
\| \eNlN \| (t) 
&\leq 
\tilde{d}   \left( \frac{N-n_\ell+2}{e A^2 \lambda_2}\right)^\frac{N-n_\ell+2}{2} 
(2A)^{3N+4-2n_\ell}  e^{-\lambda_2 t/2}
\left( \frac12 b_{3N+4-2n_\ell} (\lambda_2 t) + 1 \right)^{1/2}
\eps^{2N+2-n_\ell}
=: E_\ell^N (t). \label{eqn:def_ElN}
\end{align}
Thus, the error-bound ratios
\begin{equation}
\label{eq:error_bound_ratio}
\frac{E^{N+1}(t)}{E^N(t)} = 2A \left(\frac{b_{N+3}(2\lambda_2 t) + 2 }{b_{N+2}(2\lambda_2 t) + 2 }\right)^{1/2} \eps 
\end{equation}
and 
	\begin{align} \label{est:ratio_est_l}
	\frac{E_\ell^{N+1}(t) }{E_\ell^N (t)} 
	& = 
	 \left( \frac{N-n_\ell+3}{e A^2 \lambda_2^2}  \right)^{1/2} 
	\left( 1 + \frac{1}{N-n_\ell+2} \right)^\frac{N-n_\ell+2}{2}  	(2A)^3
	\left( \frac{b_{3N+7-2n_\ell}(\lambda_2 t)+2}{b_{3N+4-2n_\ell} (\lambda_2 t) +2}  \right)^{1/2}  \eps^2 \nonumber\\
	& \leq 8A^2
	\left( \frac{N-n_\ell+3}{ \lambda_2^2}  \right)^{1/2}   
	\left( \frac{b_{3N+7-2n_\ell}(\lambda_2 t)+2}{b_{3N+4-2n_\ell} (\lambda_2 t) +2}  \right)^{1/2}  \eps^2
	\end{align}
	can be used to quantity how much the estimates of $\|e^{N}\|$ and $\|e^{N}_\ell\|$ improve as $N$ increases. 
	
	It is easy to verify that $b_n$ satisfies the following recurrence formula:
	\beq
	b_0(s) = \frac{1}{s} \qquand b_{n+1}(s) = \frac{n+1}{s} b_n(s) + \frac{1}{s}, \quad\text{for } n \geq 0.
	\eeq
	Hence, for any $n \geq 1$, 
	\begin{equation}
	\label{eq:bn_ratio}
	\left(\frac{b_{n+1}(s) + 2 }{b_n(s) + 2 }\right)
	=  \left(\frac{\frac{n+1}{s} b_n(s) + \frac{1}{s} + 2 }{b_n(s) + 2 }\right)
	= \frac{n+1}{s}  \left(\frac{ b_n(s) + \frac{1}{n+1} }{b_n(s) + 2 }\right)
	+ \left(\frac{2 }{b_n(s) + 2 }\right)
	\leq \frac{n+1}{s}  + 1.
	\end{equation}
When applied to \eqref{eq:error_bound_ratio}, \eqref{eq:bn_ratio} with $n = N+2$ implies that 
\begin{equation}
\frac{E^{N+1}(t)}{E^N(t)} 
\leq  2A \left( \frac{N+3}{2\lambda_2 t} + 1 \right)^{1/2} \eps.  \label{eqn:est_ratio_l_1}
\end{equation}
This dependence on $t$ suggests that the normalized true-error ratio $\|e^{N+1}\|/ (\eps \| e^{N}\| )$ decreases as $t$ increases, as observed in Example \ref{2eg:test3}.  Similarly, using \eqref{eq:bn_ratio} in \eqref{est:ratio_est_l} with $n = 3N+4-2n_\ell, \dots, 3N+6-2n_\ell$ gives
\begin{align}
\frac{E_\ell^{N+1}(t) }{E_\ell^N (t)} 
&\leq  8A^2 \left( \frac{N-n_\ell+3}{\lambda_2^2}  \right)^{1/2}  
\left( \frac{3N+5-2n_\ell}{\lambda_2 t} + 1 \right)^{1/2}  \left( \frac{3N+6-2n_\ell}{\lambda_2 t} + 1 \right)^{1/2}   \left( \frac{3N+7-2n_\ell}{\lambda_2 t} + 1 \right)^{1/2}  \eps^2 \nonumber\\
&\leq  8A^2 \left( \frac{N-n_\ell+3}{\lambda_2^2}  \right)^{1/2}    \left( \frac{3N+7-2n_\ell}{\lambda_2 t} + 1 \right)^{3/2}  \eps^2 , \label{eqn:est_ratio_l_1}
\end{align}
This also suggests that the normalized true-error ratio $\|\eNpolN\|/ (\eps^2 \| \eNlN\| )$ decreases as $t$ increases, as observed for the first three moments in Example \ref{2eg:test3}.  
However, in both cases, $t$ needs to be sufficiently large in order for these ratios to be small.  In particular, any increase in the coefficient $\lambda_2=4/45$ will yield better bounds for $E^{N+1}/E^{N}$ and $E_\ell^{N+1} / E_\ell^N$. The  numerical results in the following and final example suggest that this value of $\lambda_2$, which is established in Lemma \ref{thm:energy_f}, is probably not optimal.

\begin{eg}\normalfont \label{2eg:test2}
	We investigate the ratio $a_{n+1}/a_n$ numerically using the finite sum
	\beq \label{eqn:anK}
	\anK(s) = \sum_{0<k\leq K} \left\{(Ak)^{2n} e^{-k^2 s}\right\}.
	\eeq 
	Numerical test suggest that $\anK$ converges as  $K \to \infty$ and that $K=1000$ is sufficient to capture the behavior of the infinite sum in \eqref{eqn:ant}, and therefore use $\an \approx \an^{1000}$  in the remainder of the computation.
	We compute $a_n(s)$ for $n = 1,\dots, 120$ and different values $s$.  We then plot the ratios $a_{n+1}/a_n$ in Figure \ref{fig:2eg_test2} and make the following observations:
\begin{itemize}

\item[1)]  It appears from the plots in Figure \ref{fig:2eg_test2} that 
\beq
\label{eqn:ans_num_ratio}
\frac{a_{n+1}(s)}{a_n (s)} \sim \frac{n+1}{s} .
\eeq
This approximation is consistent with the theoretical bound  in \eqref{eq:bn_ratio} for large $n$, and the profiles of the two ratios match quite well.   
	
\item[2)]  	Recall again from Lemma \ref{thm:energy_f} that $\lambda_2 = 4/45$.  Thus if we set $s = s(t) = 2 \lambda_2 t$, the first values of $s = 8/450, 8/45, 80/45$ in Figure \ref{fig:2eg_test2}(a)-(c) correspond to the values $t = 0.1, 1, 10$ that are used in Examples \ref{2eg:test3} and \ref{2eg:test4}.  As $s$ increases (\ref{fig:2eg_test2}(d)-(f)), we begin to see plateaus connected by sharp transitions.  This behavior is most notable in Figures \ref{fig:2eg_test2}(d)--\ref{fig:2eg_test2}(f), and it is reminiscent of the profiles of the normalized error ratios from Example \ref{2eg:test3} (cf. plots (e) and (f) of Figures \ref{fig:2eg_test3_e}--\ref{fig:2eg_test3_err_f2}), albeit at smaller values of $t$. Currently, we do not have any explanation for these jumps or their locations.  However, the fact that this behavior emerges for larger values of $s$ suggests that it may be possible to prove Lemma \ref{thm:energy_f} with a larger value of $\lambda_2$.

		\begin{figure}
			\centering
			\begin{subfigure}[$s= \frac{8}{450}$]{
			\includegraphics[width=.31\linewidth]{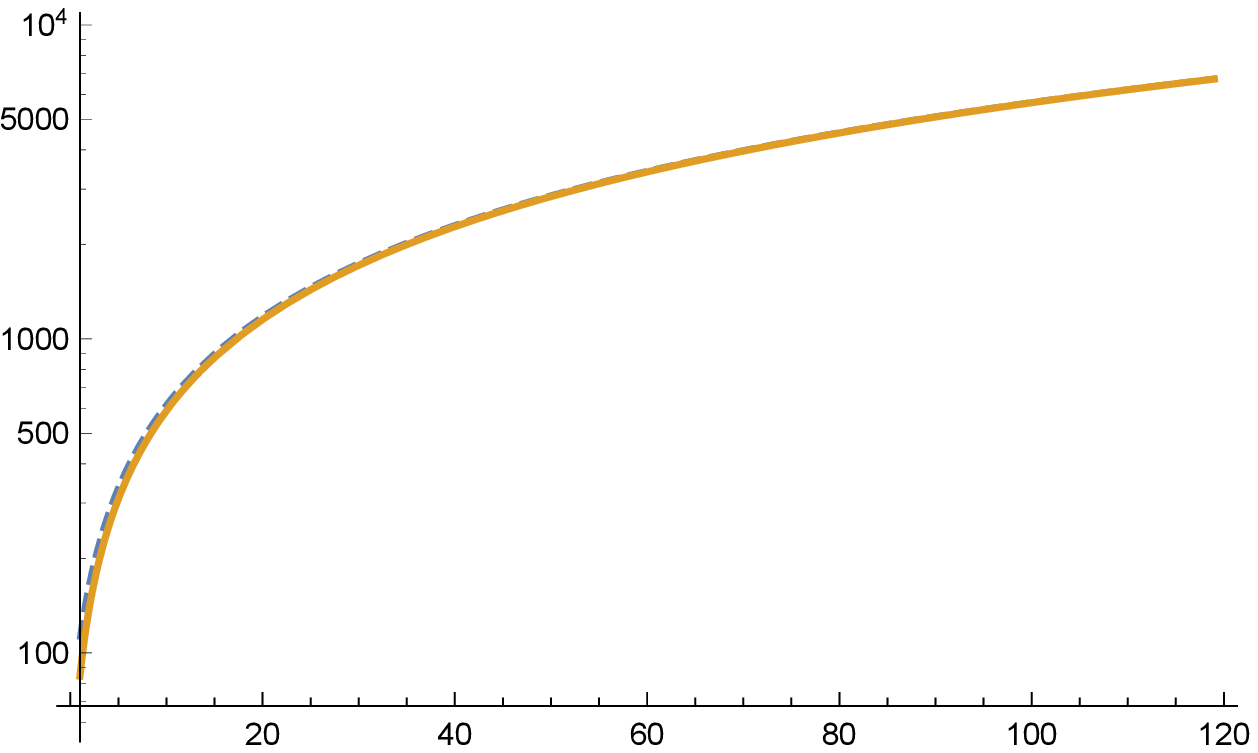}}
			\end{subfigure}%
			\begin{subfigure}[$s= \frac{8}{45}$]{
			\includegraphics[width=.31\linewidth]{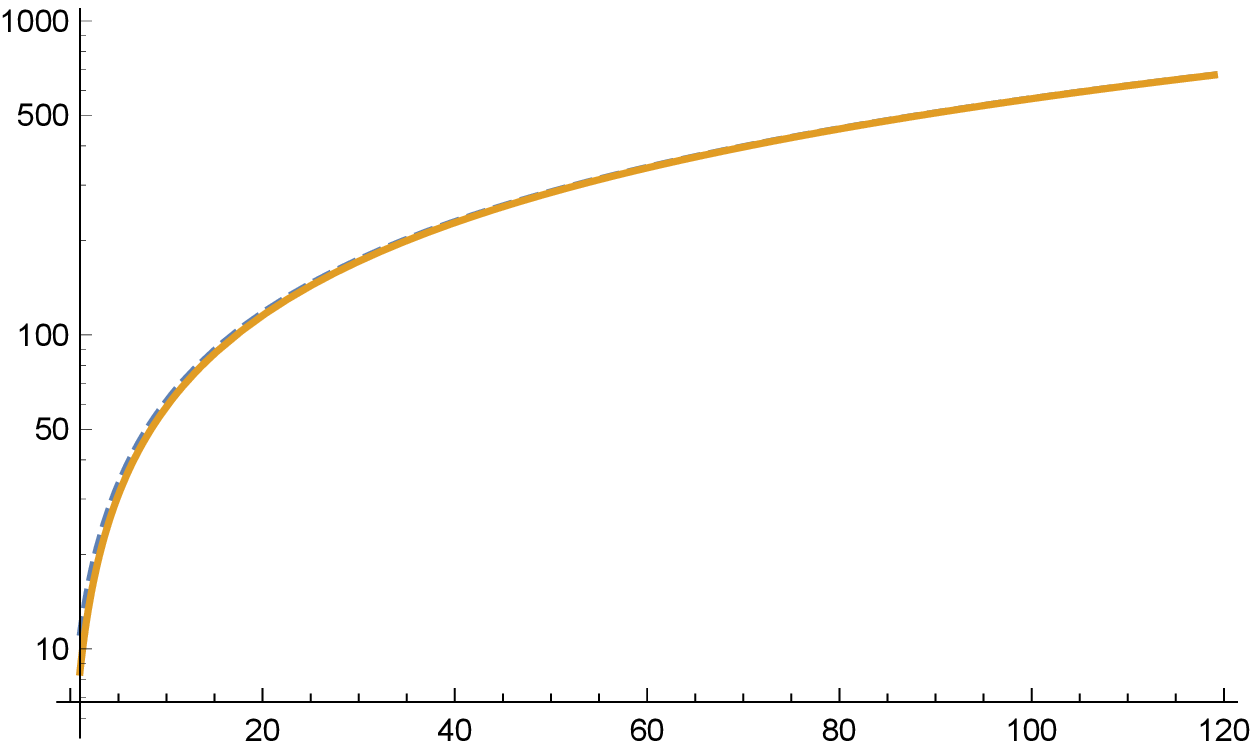}}
			\end{subfigure}%
			\begin{subfigure}[$s= \frac{80}{45}$]{
			\includegraphics[width=.31\linewidth]{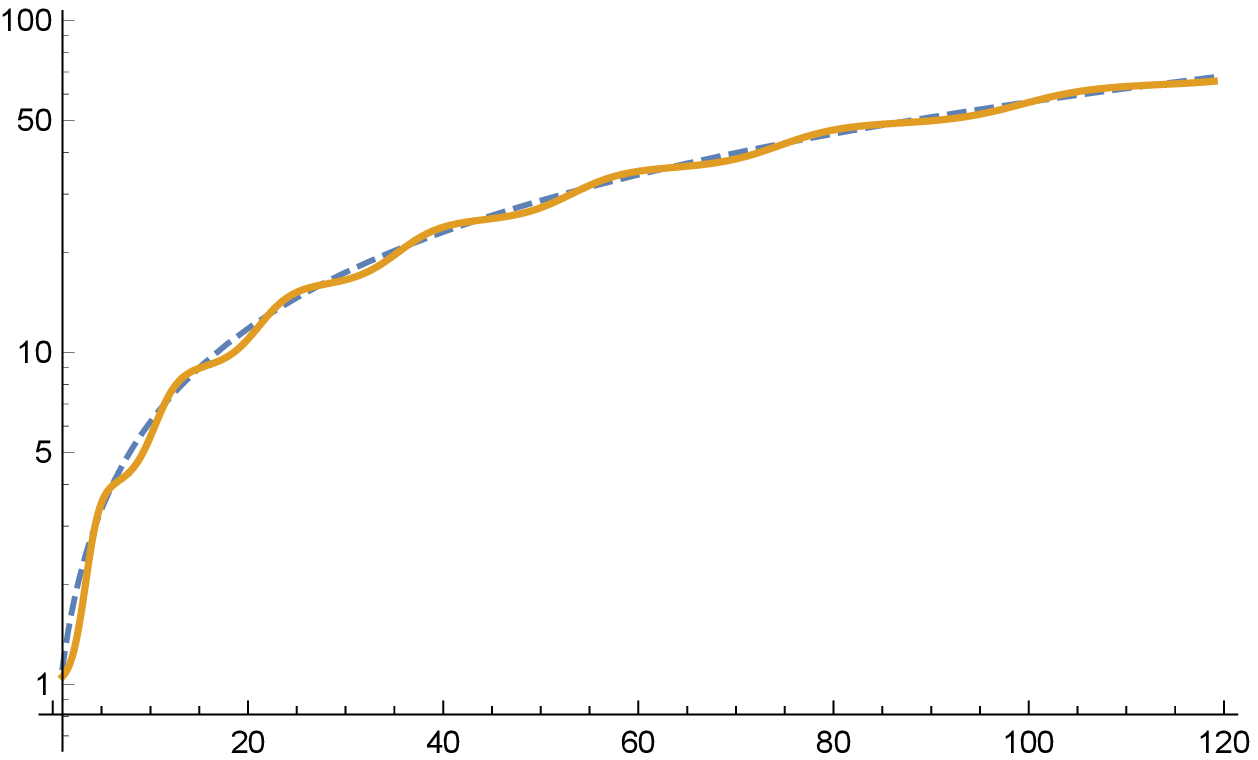}}
			\end{subfigure}\\
			\begin{subfigure}[$s=5$]{
			\includegraphics[width=.31\linewidth]{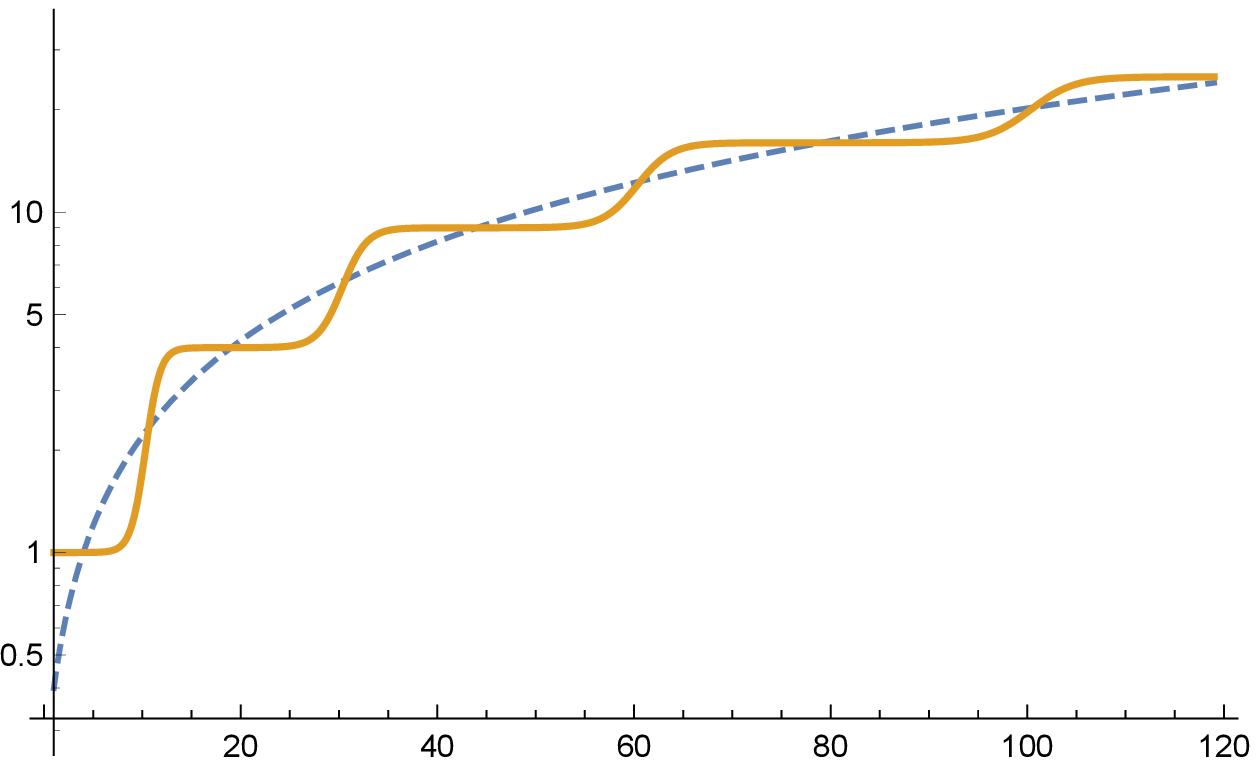}}
			\end{subfigure}%
			\begin{subfigure}[$s=10$]{
			\includegraphics[width=.31\linewidth]{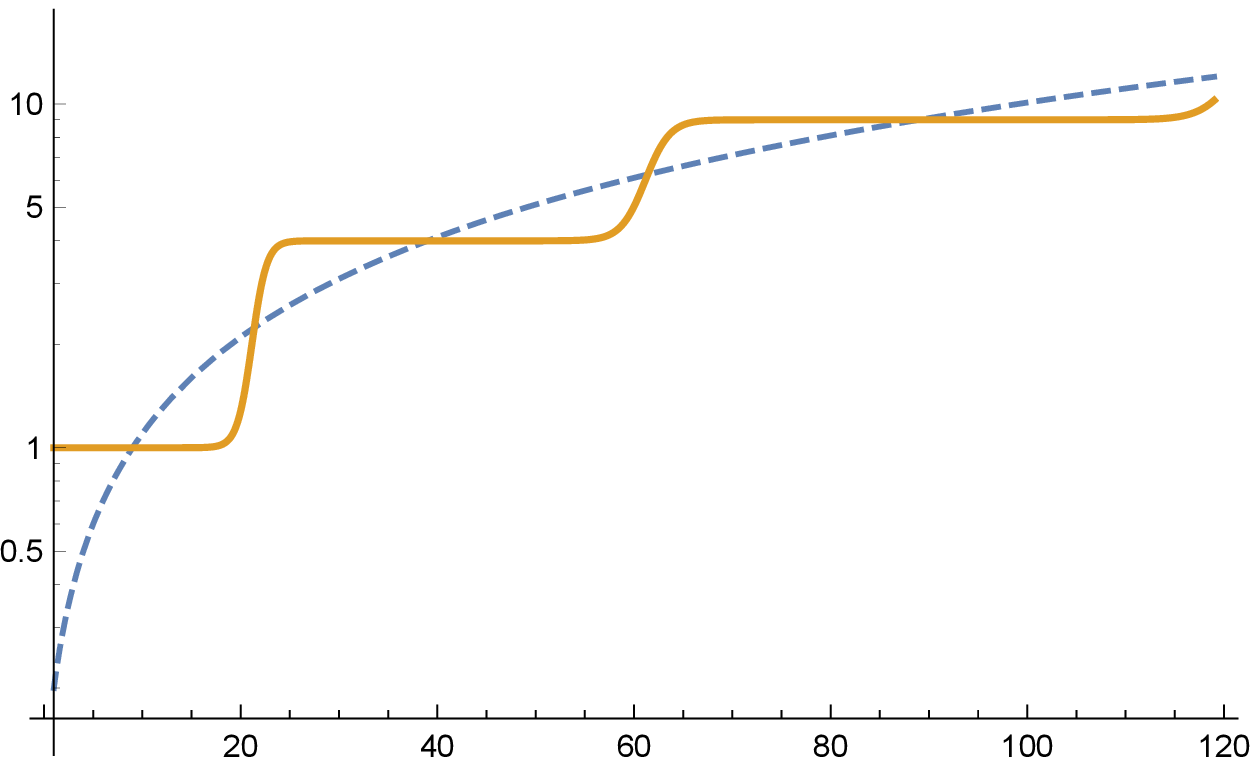}}
			\end{subfigure}%
			\begin{subfigure}[$s=15$]{
			\includegraphics[width=.31\linewidth]{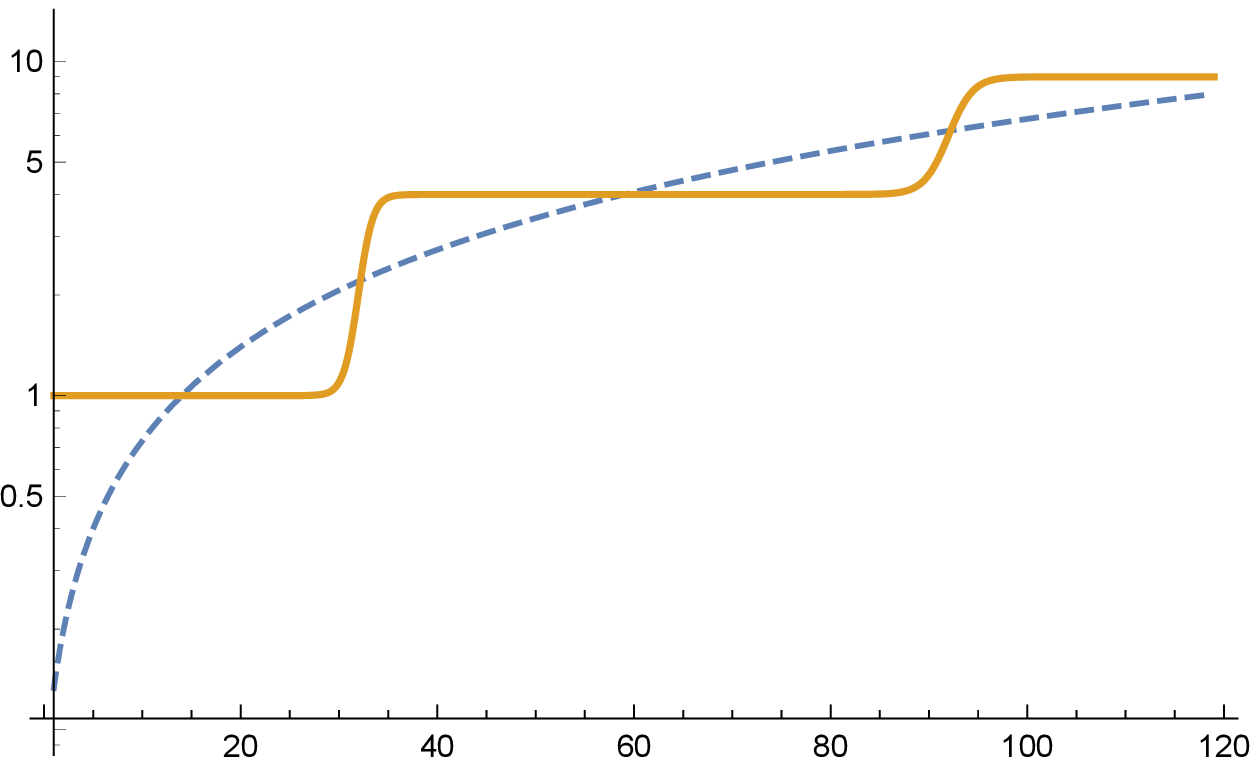}}
			\end{subfigure}
			\caption{Results from Example \ref{2eg:test2} for different values of $s$.  The orange curves are $a_{n+1}/a_n$ vs. $n$.  The blue dashed curves are of $(n+1)/s$.  }
			\label{fig:2eg_test2}
		\end{figure}
		
\end{itemize}
\end{eg}

%% file: sections/conclusion.tex
\section{Conclusion} 
\label{sec:conclusion}
In this paper, we give error estimates, in terms of a multiscale parameter $\eps$, for the spectral approximation in the velocity variable of an idealized kinetic model.  This approximation yields a linear, symmetric hyperbolic system of partial differential equations for the expansion coefficients, which are functions of $x$ and $t$.
Under the assumption that the initial data $g$ is isotropic, with $g\in L^2(d\mu dx)$ and $\partial_x g  \in L^2(d\mu dx)$, we prove that the error in the spectral approximation with $N$ modes is $\cO(\eps^{N+1})$.  In additional, we prove super-convergent results for the expansion coefficients.  We also provide numerical results that support the theoretical estimates.  These results exhibit the predicted order of convergence even when $\partial_x g \notin L^2(d\mu dx)$.  Thus it remains open whether this condition is necessary for our result. 

The coefficients of the error estimates are independent of $\eps$ but not $N$.  Thus, in an effort to demonstrate the practical benefit when increasing $N$, we investigate these coefficients both theoretically and numerically.  In particular, we find that the ratio of successive error bounds in $N$ is itself bounded above by the product $ \eps \alpha_N(t)$, where 
\begin{equation}
\alpha_N(t)
\leq  2A \left( \frac{N+3}{2\lambda_2 t} + 1 \right)^{1/2},
\end{equation}
with  $\lambda_2 = 4/45$ and $A = A(\lambda_2) \simeq 1.2.$ 
Meanwhile, the ratio in the error estimate for the moments is bounded above by the product $ \eps^2 \beta_{N,\ell}(t)$, where 
\beq
\beta_{N,\ell} (t) = 8A^2 \left( \frac{N-n_\ell+3}{\lambda_2^2}  \right)^{1/2}    \left( \frac{3N+7-2n_\ell}{\lambda_2 t} + 1 \right)^{3/2}.  
\eeq
Thus for reasonable (but not too large) values of $N$ and $t$ sufficiently large, our estimate of the spectral error improves significantly as $N$ is increased. In our analysis, we are able to prove our theoretical results with $\lambda_2 = 4/45$.  However, numerical results suggest that a larger value of $\lambda_2$ is possible and demonstrate that the theoretical benefit of having a larger value is significant.

In the future we intend to establish the theoretical results of this paper with a larger value of $\lambda_2$.   In addition, we will explore the $\eps$-dependent behavior of the error under more general initial conditions such as anisotropic initial conditions, real boundary conditions, non-zero absorption and sources, spatially dependent scattering, and higher-dimensional problems. We also hope to investigate alternative angular discretizations and nonlinear systems.

%% file: sections/appendix.tex
\section{Spectral Error Estimate} 
\label{sec:appendix}

The purpose of the section is to show that, with sufficient regularity on the initital condition $g$, the standard estimate \eqref{eqn:spectral_est} holds with a constant $C$ that is independent of $\eps \in [0,1]$.

\begin{defn}
	Let $r$, $q$, $s$, and $S$ be non-negative integers.  For any $u \in L^2(d \mu dx)$, define the shorthand $u^{(r,q)} = \partial_x^r \partial_\mu^q u$ and the semi-norm $| u|_{{r,q}} = \| u^{(r,q)} \|_{L^2(d \mu dx)}$.  Then define the space 
	\begin{equation}
	V^s(d \mu dx) = \left\{ u \in L^2(d \mu dx)  \colon \sum\limits_{q=0}^s |u|_{{s-q,q}} < \infty \right\}
	\end{equation}
	with the associated semi-norm $|\cdot|_{V^s(d \mu dx)} = \sum\limits_{q=0}^s |\cdot|_{{s-q,q}}$.  Finally, let 
	\begin{equation}
		H^S(d \mu dx) = \left\{ u \in L^2(d \mu dx)  \colon \sum_{s=0}^{S} | \cdot |_{V^s(d \mu dx)} < \infty\right\}
	\end{equation}
	be the usual Sobolev space with norm $\| \cdot \|_{H^S} = \sum_{s=0}^{S} | \cdot |_{V^s(d \mu dx)}$.
\end{defn}

\begin{lem} \label{thm:est_der_f}
	Let $f$ solve \eqref{eqn:transport_simple} with initial condition $g \in V^{s}(d \mu dx)$ for some positive integer $s$.  Then $f \in C([0,\infty);V^{s}(d \mu dx))$ with
	\begin{equation}
	\label{eqn:f-H-bound}
	| f|_{{s-q,q}} (t) \leq (q+1)! \, |g|_{V^s(d \mu dx)}
	\end{equation}
	for all integers $q \in [0, s]$ and $t \geq 0$.
\end{lem}

\begin{proof}
	Given $h \in C([0,\infty);L^2(d \mu dx))$ and $v \in L^2(d \mu dx)$, the equation 
	\begin{subnumcases}{\label{eqn:generic}}
	\partt u(x, \mu, t) + \frac{1}{\eps} \mu \partx u(x, \mu, t) + \frac{1}{\eps^2} u(x, \mu, t)  =  h(x, \mu, t),
	& $(x,\mu,t) \in [-\pi,\pi) \times [-1,1] \times (0,\infty),$ \qquad \qquad\\
	u(x, \mu, 0) = v(x,\mu),
	& $(x,\mu) \in  [-\pi,\pi) \times [-1,1],$
	\end{subnumcases}	
	has a mild solution (see, for example,\cite[p.402]{renardy2006introduction}) $u \in C([0,\infty);L^2(d \mu dx))$, given by
	\beq\label{eqn:solution}
	u(x,\mu,t) = e^{-\frac{t}{\eps^2}} v(x-\frac{1}{\eps}\mu t,\mu) + \int_0^t \, e^{-\frac{t-\tau}{\eps^2}} h(x-\frac{1}{\eps}\mu (t-\tau),\mu,\tau)\, d\tau.
	\eeq
	where the argument $x-\eps^{-1}\mu t$ is understood with respect to the periodicity of the spatial domain.
	Applying the triangle equality to \eqref {eqn:solution} gives, for each $t \geq 0$,
	\begin{align} 
	||u||_{L^2{(d \mu dx)}} (t) &\leq e^{-\frac{t}{\eps^2}} ||v||_{L^2{(d \mu dx)}} + \int_0^t \, e^{-\frac{t-\tau}{\eps^2}} ||h||_{L^2{(d \mu dx)}} (\tau)\, d\tau \nonumber \\
	&\leq e^{-\frac{t}{\eps^2}} ||v||_{L^2{(d \mu dx)}} + \eps^2 (1 - e^{-\frac{t}{\eps^2}} )  \max\limits_{\tau \in [0,t]} ||h||_{L^2{(d \mu dx)}} (\tau). \label{ineq:est}
	\end{align}

	We now proceed by induction on $q$.  If $q = 0$, then differentiation of \eqref{eqn:transport_simple_1} in $x$ gives 
	\begin{equation}
	\eps \partt f^{(r,0)} + \mu \partx f^{(r,0)} + \frac{1}{\eps} f^{(r,0)} = \frac{1}{\eps} \overline{f^{(r,0)}}
	\end{equation} 
	for any integer $r \geq 0$.  Hence  $u = f^{(s,0)}$ satisfies \eqref{eqn:generic} with source $h = \frac{1}{\eps^2} \overline{f^{(s,0)}}  \in C([0,\infty);L^2(d \mu dx))$ and initial condition $v = g^{(s,0)} \in L^2(d \mu dx)$.  Thus \eqref{ineq:est} gives
	\begin{align} 
	\label{eqn:fs0_bnd}
	|f|_{s,0} (t) 
	& \leq e^{-\frac{t}{\eps^2}} |g|_{s,0}+ (1 - e^{-\frac{t}{\eps^2}} )  \max\limits_{\tau \in [0,t]} ||\overline{f^{(s,0)}}||_{L^2{(d \mu dx)}} (\tau) \nonumber \\
	& \leq e^{-\frac{t}{\eps^2}} |g|_{V^s(d \mu dx)} + (1 - e^{-\frac{t}{\eps^2}} )  \max\limits_{\tau \in [0,t]} |f|_{s,0} (\tau),
	\end{align}

	Let $t_\ast \in [0,t]$ be such that $|f|_{s,0} (t_\ast) = \max\limits_{\tau \in [0,t]} |f|_{s,0} (\tau)$.  Then $|f|_{s,0} (t_\ast) = \max\limits_{\tau \in [0,t_\ast]} |f|_{s,0} (\tau)$ so that, according to \eqref{eqn:fs0_bnd},
	\begin{equation}
	|f|_{s,0} (t_\ast)  \leq e^{-\frac{t_\ast}{\eps^2}} |g|_{V^s(d \mu dx)} + (1 - e^{-\frac{t_\ast}{\eps^2}} ) |f|_{s,0} (t_\ast).
	\end{equation}
	Therefore
	\begin{equation}
	|f|_{s,0} (t) \leq |f|_{s,0} (t_\ast)  \leq |g|_{V^s(d \mu dx)},
	\end{equation}
	which verifies \eqref{eqn:f-H-bound}.

	Next assume that \eqref{eqn:f-H-bound} holds for $q = q_0$, with $0 \leq q_0 < s$.  Differentiation of \eqref{eqn:transport_simple_1} in $x$ and $\mu$ gives 
		\begin{equation}
		\eps \partt f^{(r,q_0+1)} + \mu \partx f^{(r,q_0+1)} + \frac{1}{\eps} f^{(r,q_0+1)} = - (q_0+1) f^{(r+1,q_0)}
		\end{equation}
	for any $r \geq 0$.  Therefore $u = f^{(s-(q_0+1),q_0+1)}$ satisfies \eqref{eqn:generic} with the source $h = -\frac{q_0+1}{\eps} f^{(s-q_0,q_0)} \in C([0,\infty);L^2(d \mu dx))$ and initial condition $ v = g^{(s-(q_0+1),q_0+1)} \in L^2(d \mu dx)$. Thus \eqref{ineq:est} gives
	\begin{align} 
	|f|_{s-(q_0+1),q_0+1} (t)
	& \leq e^{-\frac{t}{\eps^2}} |g|_{s-(q_0+1),q_0+1} +  \eps  (1 - e^{-\frac{t}{\eps^2}} ) (q_0+1) \max\limits_{\tau \geq 0} |f|_{s-q_0,q_0} (\tau) 
	\nonumber \\
	& \leq  |g|_{V^s(d \mu dx)} +  \eps (q_0+1)   (q_0 + 1)! \, |g|_{V^s(d \mu dx)} 
	\nonumber \\ 
	&\leq (q_0 + 2)! \, |g|_{V^s(d \mu dx)}.
	\end{align}
	
\end{proof}

\begin{rem}
	For sufficiently small $\eps$, the bound
	\begin{equation}
	\label{eqn:f-H-bound_eps}
	| f|_{{s-q,q}} (t) \leq \left( \prod_{i=0}^{q} (1+i\eps) \right) |g|_{V^s(d \mu dx)}
	\end{equation}
	provides a sharper estimate than \eqref 
	{eqn:f-H-bound}.  The proof of this alternative bound uses the same arguments.
\end{rem}

\begin{thm}\label{thm:spectral}
	Suppose that $g \in H^{1+q}(d \mu dx)$ for some integer $q>0$.  Then there exists a constant 
	$C = C(g,q)$, 
	such that 
	\beq
	\label{eqn:spectral_bound}
	\| f - f^N \|_{L^2(d \mu dx)}(t) \leq C (1+ t^{1/2}) N^{-q} , \quad \forall t \geq 0. 
	\eeq
\end{thm}
\begin{proof}
	We begin by estimating $\xi = Pf - f^N$ in terms of $\eta$.  A direct calculation using \eqref{eqn:transport_simple} and \eqref{eqn:spectral} shows that
	\beq\label{eqn:eqn_xi}
	\partt \xi + \frac{1}{\eps} \cP \left( \mu \partx \xi \right) + \frac{1}{\eps^2} (\xi  - \bar{\xi}) = -\frac{1}{\eps} \cP \left(  \mu \partx \eta \right),
	\eeq
	which is equivalent to the system \eqref{eqn:PN_error}. Integrating \eqref{eqn:eqn_xi} against $\xi$ on the left gives
	\begin{align}
		\frac{1}{2} \partt \| \xi \|_{L^2(d \mu dx)}^2 + \frac{1}{\eps^2} \| \tilde{\xi} \|_{L^2(d \mu dx)}^2
		&= -\frac{1}{\eps} \iint \xi \, \cP \left( \mu \partx \eta \right)  \, d\mu dx \nonumber \\
		&= -\frac{1}{\eps} \iint \xi  \, \mu \partx \eta   \, d\mu dx 
		\nonumber \\
		&=   -\frac{1}{\eps} \iint \bar{\xi} \, \mu \partx \eta \, d\mu dx
		-\frac{1}{\eps} \iint \tilde{\xi}  \, \mu \partx \eta   \, d\mu dx, \label{eqn:eqn_1}
	\end{align}
	where $\tilde{\xi} = \xi - \bar{\xi}$.  
	For $N  \geq 1$, $\mu$ and  $\eta$ are orthogonal; hence the first term in the last line of \eqref{eqn:eqn_1} is zero.   Meanwhile Young's inequality yields a bound on the second term:
	\beq
	-\frac{1}{\eps} \iint \tilde{\xi}  \, \mu \partx \eta   \, d\mu dx 
	\leq \frac{1}{2\eps^2} \| \tilde{\xi} \|_{L^2(d \mu dx)}^2 + \frac{1}{2} \|  \mu \partx \eta \|_{L^2(d \mu dx)}^2.
	\eeq
	Hence, \eqref{eqn:eqn_1} reduces to
	\beq 
	\partt \| \xi \|_{L^2(d \mu dx)}^2 + \frac{1}{\eps^2} \| \tilde{\xi} \|_{L^2(d \mu dx)}^2
	\leq  \|  \mu \partx \eta \|_{L^2(d \mu dx)}^2, 
	\eeq
	and therefore,
	\beq \label{ineqn:xi}
	\partt \| \xi \|_{L^2(d \mu dx)}^2 
	\leq  \|  \mu \partx \eta \|_{L^2(d \mu dx)}^2.
	\eeq
	Since $\xi|_{t=0}=0$, integrating \eqref{ineqn:xi} in time gives
	\beq
	\label{eqn:bound_in_eta}
	\| \xi \|_{L^2(d \mu dx)}^2(t) 
	\leq t \sup_{\tau \geq 0} \left \{ \|  \mu \partx \eta \|_{L^2(d \mu dx)}^2(\tau)\right\}
	\leq t \sup_{\tau \geq 0} \left\{ \| \partx \eta \|_{L^2(d \mu dx)}^2(\tau)\right\}
	\eeq
	Thus it remains only to bound $ \|  \partx \eta \|_{L^2(d \mu dx)}$.

	We now turn to polynomial approximation theory:  given a function $\psi \in H^q(d \mu)$, where $H^q(d \mu)$ is the Sobolev space of functions with $q$ weak derivatives in $L^2(d \mu)$, there exists a constant $K_1 > 0$ such that  \cite[Lemma 2.2]{ben1998spectral}
	\begin{equation}\label{eqn:bounds_eta_u}
		\|\psi - \cP \psi \|_{L^2(d \mu)} \leq K_1 \| \psi \|_{H^q(d \mu)} N^{-q}.
	\end{equation}
	We apply this result to $\partx \eta = \partial_x f - \cP \partial_x f$, using also Lemma \ref{thm:est_der_f}, to find that
		\begin{align}
		\|  \partx \eta \|_{L^2(d \mu dx)} (\tau)
		&\leq   K_1 \| \partx f \|_{L^2(dx;H^q (d \mu))} (\tau) N^{-q} \nonumber\\
		&=  K_1  \left(\sum\limits_{0 \leq r \leq q}  | f |_{{1,r}} (\tau) \right) \, N^{-q} \nonumber\\
		&\leq  K_1  \sum\limits_{0 \leq r \leq q} (r+1)!\, | g |_{V^{1+r}(d \mu dx)} \, N^{-q} \nonumber\\
		&\leq K_1   (q+1)!\,  \| g \|_{ H^{1+q}(d \mu dx)} \, N^{-q}.
		\label{eqn:poly_error_f}
		\end{align}
	This bound is independent of $t$.   Thus combining \eqref{eqn:bound_in_eta} and \eqref{eqn:poly_error_f} gives
	\beq
	\label{eqn:xi_final}
	\| \xi \|_{L^2(d \mu dx)}(t) \leq  K_1 (q+1)! \, \, t^{1/2}  \| g \|_{ H^{1+q}(d \mu dx)} N^{-q}.
	\eeq

To complete the proof, we estimate $\eta = f - \cP f$ using \eqref{eqn:bounds_eta_u} and Lemma \ref{thm:est_der_f}, 
\begin{equation}
\label{eqn:eta_final}
\| \eta\|_{L^2(d \mu dx)} (t) \leq K_1 \| f \|_{L^2(dx;H^q(d \mu))}(t) N^{-q} 
\leq  K_1  (q+1)! \, \| g \|_{ H^{q}(d \mu dx)} N^{-q} , \quad \forall t \geq 0. 
\end{equation}
Combining \eqref{eqn:xi_final} and \eqref{eqn:eta_final} recovers \eqref{eqn:spectral_bound} with 
$C = K_1 (q+1)! \, \| g \|_{ H^{1+q}(d \mu dx)}$. 

\end{proof}